\documentclass[a4paper,11pt]{amsart}

\usepackage{amssymb,amsmath,amsthm,mathrsfs,enumerate,graphicx, color}
\usepackage[pdfpagelabels,backref,colorlinks,linkcolor=MidnightBlue,citecolor=OliveGreen,urlcolor=green]{hyperref}
\usepackage[dvipsnames]{xcolor}

\usepackage{esint}
\usepackage{stix}
\usepackage{hyperref}
\renewcommand\eqref[1]{(\ref{#1})} 

\newtheorem{theorem}{Theorem}[section]
\newtheorem{proposition}[theorem]{Proposition}
\newtheorem{lemma}[theorem]{Lemma}
\newtheorem{corollary}[theorem]{Corollary}
\theoremstyle{definition}
\newtheorem{definition}[theorem]{Definition}

\newtheorem{remark}[theorem]{Remark}

\numberwithin{equation}{section}

\newcommand{\supp}{\operatorname{supp}}

\newcommand{\f}{\frac}
\newcommand{\les}{\lesssim}

\newcommand{\ZZ}{\mathbb Z}
\newcommand{\vc}{\infty}
\newcommand{\FZ}{\mathcal F_{\mathbb Z}}
\newcommand{\mm}{{\mathfrak m}}
\newcommand{\B}{\dot{B}}
\newcommand{\F}{\dot{F}}
\newcommand{\Dd}{\Delta_d}

\begin{document}

\title[Function spaces  associated with a discrete Laplacian]{Hardy spaces, Besov spaces and Triebel--Lizorkin spaces  associated with a discrete Laplacian and applications}

\author[T. A. Bui]{The Anh Bui}
\address{The Anh Bui
\endgraf
School of Mathematical and Physical Sciences, Macquarie University
\endgraf
NSW 2109, Australia}
\email{the.bui@mq.edu.au}
\author[X. T. Duong]{Xuan Thinh Duong}
\address{Xuan Thinh Duong
	\endgraf
	School of Mathematical and Physical Sciences, Macquarie University
	\endgraf
	NSW 2109, Australia}
\email{xuan.duong@mq.edu.au}






\begin{abstract}
Consider	the discrete Laplacian $\Dd$ defined on the set of integers $\ZZ$ by
	\[
	\Delta_df(n) = -f(n+1) + 2f(n) -f(n-1), \ \ \ \ n\in \ZZ,
	\]
	where $f$ is a function defined on $\ZZ$.
In this paper, we define Hardy spaces, Besov spaces and Triebel--Lizorkin spaces associated with $\Dd$ and then show that these function spaces coincide with the classical function spaces defined on $\ZZ$. As applications, we prove the boundedness of the spectral multipliers and the Riesz transforms  associated with $\Dd$ on these function spaces. 
\end{abstract}

\maketitle

\tableofcontents

\section{Introduction and the statement of main results}
\allowdisplaybreaks

Let $\ZZ$ be the set of integers equipped with the counting measure. In this case, $\ZZ$ can be viewed as a space of homogeneous type in the sense of Coifman and Weiss \cite{CW}.  For every $0<p\le \vc$, we denote by $\ell^p(\ZZ)$ the Lebesgue space on $\ZZ$ with respect to the counting measure, i.e., if $f\in \ell^p(\ZZ)$, then
\[
\|f\|_{\ell^p(\ZZ)} = \Big(\sum_{n\in \ZZ} |f(n)|^p\Big)^{1/p}.
\] 
The harmonic analysis on the set of all integers $\ZZ$ is an interesting topic in harmonic analysis. See for example \cite{FT, BZ, Boza,Torres, CGRTV, ABM, KS} and the references therein. We would like to refer to \cite{FT, Torres, Sun} for the  theory of Besov and Triebel--Lizorkin spaces and \cite{BZ, Boza} for  the theory of Hardy spaces. Moreover, the  boundedness of the Fourier multipliers on $\ZZ$ was investigated in \cite{Torres, KS}. 

\medskip

In this paper, we consider the discrete Laplacian $\Delta_d$ on $\ZZ$  defined by
\[
\Delta_df(n) = -f(n+1) + 2f(n) -f(n-1), \ \ \ \ n\in \ZZ,
\]
where $f$ is a function defined on $\ZZ$.  

Let us emphasize that the study of the discrete Laplacian is an interesting topic in harmonic analysis and PDEs. See for example \cite{ABM, BG, Boza, BZ, JW, K,KS, Sch} and the references therein. Recently, in \cite{CGRTV}, Ciaurri et. al. studied the boundedness of some operators including the maximal functions, the square functions and the Riesz transforms associated to the discrete Laplacian $\Dd$ on $\ell^p(\ZZ)$ with $1<p<\vc$. Later, the boundedness of  the maximal functions, the square functions and the spectral multipliers of Laplace transform type associated with $\Dd$ on Hardy spaces $H^p(\ZZ), 0<p\le 1,$ was obtained in \cite{ABM}. Although harmonic analysis on $\ZZ$ has been studied intensively, there are a number of open problems such as 
\begin{itemize}
	\item the boundedness of general spectral multiplier of $\Dd$ and the Riesz transforms on some function spaces such as Hardy spaces, Besov spaces and Triebel--Lizorkin spaces;
	\item new characterizations of the Hardy spaces, Besov spaces and Triebel--Lizorkin spaces via square functions associated with the discrete Laplacian $\Dd$.
\end{itemize}
The main aim of this paper is to develop the harmonic analysis on integers by using a new approach. Our approach was inspired by theory of function spaces adapted to differential operators which has been developed recently. See for example \cite{ADM, AMR, DY, DY1, HLMMY, HM, JY, BBD, AA}. We would like to summarize the main results in this paper.

\begin{enumerate}
	\item \textit{Theory of Hardy spaces associated with $\Dd$:} Motivated by the works \cite{ADM, DY, DY1, HM, HLMMY} we define the Hardy space associated with $\Dd$ via the area square function   
	\[
	S_{\Delta_d}f(n)=\Big(\int_0^\vc \sum_{m:|m-n|<t}\big|t^2\Delta_de^{-t^2\Delta_d}f(m)\big|^2\f{dt}{t(t+1)}\Big)^{1/2},
	\]
	and show that the Hardy spaces also admit molecular decomposition which is analogous to the decomposition of the classical Hardy spaces. Then we will show that the new Hardy spaces coincide with the classical Hardy spaces $H^p(\ZZ)$ for $0<p\le 1$. It is important to note that in our setting the kernel of $e^{-t\Dd}$ satisfies neither Davies--Gaffney estimates nor Gaussian upper bounds. Besides that, it is not clear if the kernel of $e^{-t\Dd}$ satisfies the H\"older continuity. Hence,  new ideas and significant modifications would be required to handle these challenges. More precisely, in this case, the   kernel of $e^{-t\Dd}$ enjoys the polynomial decay. However, this is not enough to develop the theory of Hardy space for $p$ close to $0$. Fortunately, by rigorous computation we will prove that higher order derivatives in time of the kernel of $e^{-t\Dd}$ will present more decays. This plays a key role in the paper.  We also emphasize that in comparison with the classical atoms/molecules, our new molecules of the Hardy spaces have advantages in proving the boundedness of certain singular integrals such as the spectral multipliers of $\Dd$. See Section 3 and Section 5.1.

	\item \textit{Theory of Besov and Triebel--Lizorkin spaces associated with $\Dd$:} let us remind that the Besov and Triebel--Lizorkin spaces on $\ZZ$ were initiated by \cite{FT, Torres, Sun}. Their approaches based on the connection between the function spaces on $\mathbb R$ and the function spaces on $\ZZ$. Once the connection has been set up, properties of Besov and Triebel--Lizorkin spaces on $\mathbb R$ will be transferred into the function spaces on $\ZZ$. However, it is not clear how to adapt this method to prove the boundedness of the spectral multipliers of $\Dd$ and the Riesz transforms of $\Dd$. In this paper, we approach the problem by introducing the new Besov and Triebel--Lizorkin spaces adapted to $\Dd$. We also obtain new molecular decomposition of these function spaces. More importantly, we can show that the new function spaces and the function spaces introduced in  \cite{FT, Torres, Sun} coincide. Moreover, similarly to the case of the Hardy spaces, we are able to prove the boundedness of the spectral multipliers and the Riesz transforms by making use of molecular decomposition with new molecules. See Sections 4 and 5.    
	
	\item \textit{Boundedness of the spectral multipliers of $\Dd$:} Note that $\Delta_d$ is a non-negative and self-adjoint operator on $\ell^2(\ZZ)$. Moreover, it is straightforward to see that $\|\Delta_df\|_{\ell^p(\ZZ)}\le 4\|f\|_{\ell^p(\ZZ)}$. It follows that the spectral of $\Delta_d$ is contained in $[0,4]$. In fact, it is well-known that for every $\theta\in [-\pi,\pi]$, we have
	\[
	\Delta_d e_\theta = 2(1-\cos\theta)e_\theta, 
	\]
	where $e_\theta(n)=e^{in\theta}$.

	By the spectral theory, for a bounded Borel measurable function $m$ on $[0,\vc)$ we can define
	\[
	\begin{aligned}
		F(\Dd) = \int_{0}^{\vc} F(\lambda)dE(\lambda),
	\end{aligned}
	\]
	which is bounded on $\ell^2(\ZZ)$, where  $\{E(\lambda): \lambda \geq 0\}$ is the spectral resolution of $\Delta_d$. Since the spectrum of $\Dd$ is contained in $[0,4]$, the integral above is vanishing over $(4,\infty]$.

	For $f\in \ell^1(\mathbb Z)$, the Fourier transform $\mathcal F_{\mathbb Z}(f)$ of $f$ is defined by
	\[
	\FZ(f)(\theta) =\sum_{n\in \mathbb Z} f(n)e^{in\theta}, \ \ \ \ \theta\in [-\pi,\pi].
	\] 
	The inverse operator $\FZ^{-1}$ of $\FZ$ is determined by
	\[
	\FZ^{-1}(\varphi)(n) = \f{1}{2\pi} \int_{-\pi}^\pi \varphi(\theta) e^{-in\theta} d\theta, \ \ \ \ \varphi\in L^2[-\pi,\pi].
	\]

	Let ${\mathfrak m}\in L^\infty(0,\infty)$. We define
	\[
	T_{\mm} f = \FZ^{-1}\big[\mm(\sqrt{2(1-\cos\theta)})\FZ(f)\big].
	\]
	It can be verified that 
	\[
	T_{\mm} f =\mm (\sqrt \Delta_d)f.
	\]
	In \cite{ABM}, the spectral multiplier of Laplace transform type was considered corresponding to the case $\mm(\lambda) = \lambda^2\int_0^\vc e^{-\lambda^2 t}\Psi(t)dt$, $\lambda>0$ and $\Psi \in L^\vc(0,\vc)$. It was proved that 
	\[
	\|\mm (\sqrt \Delta_d)\|_{H^p(\ZZ)\to H^p(\ZZ)}\le C_\Psi, \ \ \ \ 0<p\le 1.
	\] 
 In the particular case when $\psi(t) = -\f{1}{\Gamma(is)}t^{is}$, the constant $C_\Psi$ is an exponential function of $s$. It is clear that in comparison with the bound obtained by the general spectral multiplier theorem (see for example \cite{DOS}), we should expect that $C_\Psi={\rm constant}\times s^{\f{1}{p}-\f{1}{2}+\epsilon}$ for any $\epsilon>0$. In Section 5 (Theorem \ref{mainthm-spectralmultipliers}), by using our theory of new function spaces, we are able to prove the sharp estimate for the spectral multipliers $\mm (\sqrt \Delta_d)$ on Hardy spaces, Besov spaces and Triebel--Lizorkin spaces.  
	
	\item \textit{Boundedness of the Riesz transforms:} Let $D$ and $D^*$ be operators defined by
	\[
	Df(n) = f(n+1) -f (n), \ \ \ \  D^*f(n) = f (n)-f(n-1).
	\]
	These operators can be viewed as the first ``discrete derivatives'' on $\ZZ$. It is clear that 
	$-\Delta_d =D^*D=DD^*$ and
	\[
	\displaystyle \langle \Delta_d f, f\rangle_{\ell^2(\ZZ)}=\langle Df, Df\rangle =\langle D^*f, D^*f\rangle_{\ell^2(\ZZ)}, 
	\] 
	where $\displaystyle \langle f, g\rangle_{\ell^2(\ZZ)} =\sum_{n\in \ZZ}f(n)g(n)$.
	
	In \cite{CGRTV}, the Riesz transforms are defined via the limits
	\[
	\mathcal R = \lim_{\alpha\to (1/2)^-}D\Dd^{-\alpha} \ \ \ \ \text{and} \ \ \ \ \widetilde{\mathcal R} = \lim_{\alpha\to (1/2)^-}D^*\Dd^{-\alpha}. 
	\]
	This is because the subordination formula
	\[
	\Delta_d^{-1/2}f = c\int_0^\vc \sqrt t e^{-t\Delta_d}f\f{dt}{t}.
	\]
	is not absolutely convergent in $\ell^2(\ZZ)$. However, by using new characterizations of the distributions in  Section 2.3, we are able to prove that the subordination formula is well-defined in the sense of the distributions. This allows us to use this formula to prove the boundedness of the Riesz transforms $D\Delta_d^{-1/2}$ and $D^*\Delta_d^{-1/2}$ on Besov and Triebel--Lizorkin spaces which include the Hardy spaces. Note that even the boundedness of the Riesz transforms on the Hardy spaces $H^p(\ZZ)$ is new.
\end{enumerate}

The organization of the paper is as follows. In Section 2, we first prove some kernel estimates for the semigroup $e^{-t\Dd}$ and a functional calculus of $\Dd$. Then some properties of distributions will be established. Section 3 will set up the theory of Hardy spaces associated with $\Dd$. The theory of Besov and Triebel--Lizorkin spaces will be addressed in Section 4. Finally, Section 5 will be devoted in proving the boundedness of the spectral multipliers and the Riesz transforms.

\medskip

Throughout the paper, we usually use $C$ and $c$ to denote positive constants that are independent of the main parameters involved but whose values may differ from line to line. We will write $A\les B$ if there is a universal constant $C$ so that $A\leq CB$ and $A\sim B$ if $A\les B$ and $B\les A$.

\section{Kernel estimates for heat semigroup and a functional calculus of $\Delta_d$ }
We begin with some notations which will be used frequently.
\begin{itemize}
	\item $a\vee b = \max\{a, b\}$;
	\item $a\wedge b = \min\{a, b\}$;
	\item $\mathbb N= \{0,1,2,\ldots\}$;
	\item $\mathbb N^*= \{1,2,\ldots\}$;
	\item $\mathbb Z= \{\ldots, -2, -1, 0,1,2,\ldots\}$;
	\item $\mathbb Z^-= \{-1,-2\ldots\}$ and $\mathbb Z^+ =\mathbb Z\backslash \mathbb Z^-$;
\item Let $m, n\in \ZZ$ with $m\le n$. Without any confusion, we use $[m,n]$ to indicate the interval in $\ZZ$, i.e., $[m,n]= \{m, m+1, \ldots, n\}$ and use $[m,n]_{\mathbb R}$ to indicate the usual interval in $\mathbb R$. For an interval  $I\subset \ZZ$ and  $\lambda>0$, we define $\lambda I = \lambda I_{\mathbb R}\cap \ZZ$. Let $I\subset \ZZ$ be an interval in $\ZZ$. We denote by $|I|$ the measure of $I$, which is equal to cardinality of $I$ and also is comparable with $\ell(I)+1$, where $\ell(I)$ denotes  the length  of $I$.

\item Let $I$ be an interval in $\ZZ$, for $j\in \mathbb N$ we define $S_j(I)=2^{j}I\backslash 2^{j-1}I$ if $j\ge 1$ and $S_0(I)=I$.

\item For $n\in \ZZ$ and $t>0$, denote $I_n(t) =[n-t, n+t]:= [n-t, n+t]_{\mathbb R}\cap \ZZ$.

\item For each $(\nu,k)\in \ZZ^-\times \ZZ$, denote $I_{\nu,k}=[k.2^{-\nu},k.2^{-\nu+1})$ which is a dyadic interval in $\ZZ$. We then set, for $\nu\in \ZZ^-$, 
\[
\mathcal I_\nu = \{I_{\nu,k}: k\in \ZZ\}
\]
and
\[
\mathcal I = \{\mathcal I_{\nu}: \nu\in \ZZ^-\}.
\]
\end{itemize}

For each $r>0$ the Hardy--Littlewood maximal function $\mathcal M_r$ is defined by setting
\[
\mathcal M_r f(n) = \sup_{I \ni n}\Big(\f{1}{|I|}\sum_{m\in I} |f(m)|^r\Big)^{1/r},
\]
where the supremum is taken over all intervals $I$ containing $n$.
It is well-known that for $p>r$, 
\begin{equation}
	\label{boundedness maximal function}
	\|\mathcal{M}_{r} f\|_{\ell^p(\ZZ)}\les  \|f\|_{\ell^p(\ZZ)}
\end{equation}
for all $f\in \ell^p(\ZZ)$.

Moreover, let $0<r<\vc$ and $p>r$. Then we have
\begin{equation}
	\label{boundedness maximal function 2}
	\|\mathcal{M}_{r} f\|_{\ell^p(\ZZ)}\les  \|f\|_{\ell^p(\ZZ)}
\end{equation}

We recall the Fefferman-Stein vector-valued maximal inequality in \cite{GLY}. For $0<p<\vc$, $0<q\leq \vc$, $0<r<\min \{p,q\}$, we then have for any sequence of measurable functions $\{f_\nu\}$,
\begin{equation}\label{FSIn}
	\Big\|\Big(\sum_{\nu}|\mathcal{M}_rf_\nu|^q\Big)^{1/q}\Big\|_{\ell^p(\ZZ)}\les  \Big\|\Big(\sum_{\nu}|f_\nu|^q\Big)^{1/q}\Big\|_{\ell^p(\ZZ)}.
\end{equation}
The Young's inequality and \eqref{FSIn} imply the following  inequality: If $\{a_\nu\} \in \ell^{q}\cap \ell^{1}$, then 
\begin{equation}\label{YFSIn}
	\Big\|\sum_{j}\Big(\sum_\nu|a_{j-\nu}\mathcal{M}_r f_\nu|^q\Big)^{1/q}\Big\|_{\ell^p(\ZZ)}\les  \Big\|\Big(\sum_{\nu}|f_\nu|^q\Big)^{1/q}\Big\|_{\ell^p(\ZZ)}.
\end{equation}

The following elementary estimates will be used frequently.
\begin{lemma}\label{lem-elementary}
	Let $\epsilon >0$. Then we have, for all function $f$, $n\in \ZZ$ and $s>0$,
	$$
	 \sum_{m\in \ZZ} \f{1}{s}\Big(1+\f{|m-n|}{s}\Big)^{-1-\epsilon} \les  1,
		$$
and
	$$
		\sum_{m\in \ZZ}\f{1}{s}\Big(1+\f{|m-n|}{s}\Big)^{-1-\epsilon}|f(m)| \les  \mathcal{M}f(n).$$
\end{lemma}
\begin{proof}
	(a) The proof of (a) is similar to (b) and hence we omit the details.

	(b) We have
	\[
	\begin{aligned}
		\sum_{m\in \ZZ}\f{1}{s}\Big(1+\f{|m-n|}{s}\Big)^{-1-\epsilon}|f(m)| &\le |f(n)| + \sum_{m\neq n}\f{1}{s}\Big(1+\f{|m-n|}{s}\Big)^{-1-\epsilon}|f(m)|\\
		&\le \mathcal M f(n) + \sum_{m\neq n}\f{1}{|m-n|^{1+\epsilon}}|f(m)|.
	\end{aligned}
	\]
	It is easy to see that 
	\[
	\begin{aligned}
		\sum_{m\neq n}\f{1}{|m-n|^{1+\epsilon}}|f(m)|&\le \sum_{k\ge 0} \sum_{m\in I_n(2^{k+1})\backslash I_n(2^{k})} \f{1}{|m-n|^{1+\epsilon}}|f(m)|\\
		&\les \sum_{k\ge 0} 2^{-k\epsilon}\mathcal M f(n)\\
		&\les \mathcal M f(n).
	\end{aligned}
	\]
	This completes our proof.
\end{proof}

\subsection{Heat kernel estimates}
Denote by $e^{-t\Delta_d}$ the semigroup generated by $\Delta_d$. Then we have
\begin{equation}\label{eq-heat kernel formula}
e^{-t\Delta_d}f(n) =\sum_{m\in \mathbb Z} h_t(n-m)f(m), \ \ \ \ n\in \ZZ,
\end{equation}
where $h_t(n) = e^{-2t}\mathfrak{I}_n(2t), t>0$ and $m\in \ZZ$ and $\mathfrak{I}_n$ is the modified Bessel function of first kind defined by
\begin{equation}
	\label{eq-1st from of I_n}
	\mathfrak{I}_n(z) = \f{z^n}{\sqrt{\pi} 2^n\Gamma(n+1/2)}\int_{-1}^1 e^{-zs}(1-s^2)^{n-1/2}ds, \ \ \ \ \ |\arg z|<\pi, \ \ \ n>-1/2.
\end{equation} 
The modified Bessel function $\mathfrak{I}_n$  is also represented in the following integral form
\begin{equation}
	\label{eq-2nd from of I_n}
	\mathfrak{I}_n(z) = \f{1}{\pi}\int_{0}^\pi e^{z\cos\theta}\cos(n\theta)d\theta, \ \ \ \ \ \ |\arg z|<\pi, \ \ \ n\in \ZZ.
\end{equation}

See for example \cite{Me, PBM}.

From \eqref{eq-heat kernel formula}, \eqref{eq-1st from of I_n} and \eqref{eq-2nd from of I_n}, we have, for $|\arg z|<\pi$ and $n\in \ZZ$,
\begin{align}
h_z(n)&=\f{z^n}{\sqrt{\pi} \Gamma(n+1/2)}\int_{-1}^1 e^{-2z(s+1)}(1-s^2)^{n-1/2}ds \label{eq1-heat kernel}\\
\text{and} \ \ h_z(n)&=\f{1}{\pi}\int_{0}^\pi e^{-2z(1- \cos\theta)}\cos(n\theta)d\theta. \label{eq2-heat kernel}
\end{align}
Moreover, since $\mathfrak{I}_{-n}(z)=\mathfrak{I}_n(z)$, we have $h_z(n)=h_z(-n)$ for all $n\in \ZZ$.

We now prove some estimates regrading the kernel $h_t(n)$.
\begin{lemma}\label{lem1-ht}
	For each $N \ge 0$, there exists $C_N>0$ such that 
	\[
	|h_t(n)|\le \f{C_N}{ \sqrt{t}}\Big(1+ \f{n}{\sqrt{t}}\Big)^{-2N-1}
	\]
	for $n=0$ and $n\in \ZZ$ with $|n|>N$, and $t>0$.
\end{lemma}
\begin{proof}
	By \eqref{eq1-heat kernel}, we have
	\[
	h_t(n)=\f{t^n}{\sqrt{\pi} \Gamma(n+1/2)} \int_{-1}^{1}e^{-2t(s+1)}(1-s^2)^{n-1/2}ds.
	\]
	
	Using the change of variable $w=\f{1}{2}t(1+s)$, we further obtain
	\[
	\begin{aligned}
		|h_t(n)|&\les  \f{t^{n}}{   \Gamma(n+1/2)}  \int_{0}^{t}\Big(\f{2w}{t}\Big)^{n-1/2} e^{-4w} \Big[2\Big(1- \dfrac{w}{t}\Big)\Big]^{n-1/2}\f{dw}{t}\\
		&\les   \f{4^nt^{n}}{   \Gamma(n+1/2)}  \int_{0}^{t}\Big(\f{w}{t}\Big)^{n-1/2} e^{-4w} \Big(1- \dfrac{w}{t}\Big)^{n-1/2}\f{dw}{t}\\
	\end{aligned}
	\]
	Applying the following inequality
	\[
	\Big(\f{w}{t}\Big)^{N+1/2} \Big(1-
	\dfrac{w}{t}\Big)^{n-1/2}\le \Big(\f{N+1/2}{N+n}\Big)^{N+1/2} \le \f{C_N}{n^{N+1/2}},
	\]
	we have
	\[
	\begin{aligned}
		|h_t(n)|&\le  \f{C_N}{n^{N+1/2}}  \f{4^nt^{n}}{   \Gamma(n+1/2)}  \int_{0}^{t}\Big(\f{w}{t}\Big)^{n-N-1} e^{-4w}  \f{dw}{t}\\
		&\le  \f{4^{N+1}C_N}{n^{N+1/2}}  \f{t^{N}}{   \Gamma(n+1/2)}  \int_{0}^{t}(4w)^{n-N-1} e^{-4w}  d(4w)\\
		&\le  \f{C_N t^{N}}{n^{N+1/2}}  \f{\Gamma(n-N)}{   \Gamma(n+1/2)}  \\
		&\le  \f{C_N t^{N}}{n^{N+1/2}}  \f{1}{  n^{N+1/2}} = \f{C_N}{\sqrt t}\Big(\f{\sqrt {t}}{n}\Big)^{2N+1},
	\end{aligned}
	\]
	which implies that 
	\begin{equation}\label{eq1-proof lem 1}
		|h_t(n)| \le  \f{C_N}{ \sqrt{t}}\Big(\f{\sqrt{t}}{n}\Big)^{2N+1}
	\end{equation}
	for $|n|>N$.

	On the other hand, from \eqref{eq2-heat kernel},
	\[
	h_t(n) =\f{1}{\pi}\int_0^\pi e^{-2t(1-\cos\theta)}\cos(n\theta) d\theta
	\]
	for all  $n\in \mathbb Z$.

	Using the fact that $(1-\cos\theta)\gtrsim \theta^2$  for $\theta\in (0,\pi)$ we have
	\[
	\begin{aligned}
		|h_t(n)| &\le\f{1}{\pi}\int_0^\pi e^{-2ct\theta^2} d\theta\\
		&\les \f{1}{\sqrt t}
	\end{aligned}
	\]
	for all $n\in \mathbb Z$.
	
	This, along with \eqref{eq1-proof lem 1}, implies the desired estimates, i.e.,
	\[
	|h_t(n)|\le \f{C_N}{ \sqrt{t}}\Big(1+ \f{n}{\sqrt{t}}\Big)^{-2N-1}
	\]
for 	$n=0$ and $n\in \ZZ$ with $|n|>N$, and $t>0$,

\end{proof}

\begin{remark}
	In Lemma \ref{lem1-ht}, if we choose $2N = \delta \in (0,1)$, then we have
	\[
	|h_t(n)|\le \f{C_N}{ \sqrt{t}}\Big(1+ \f{n}{\sqrt{t}}\Big)^{-1-\delta}
	\]
	for all $n\in \ZZ$ and $t>0$. 
	
	It follows, by using Lemma \ref{lem-elementary}, that 
	\[
	\sup_{t>0} |e^{-t\Delta_d}f|\les\mathcal M f. 
	\]
	
	This inequality implies the weak-type $(1,1)$ and the boundedness of the operator $f\mapsto \sup_{t>0} |e^{-t\Delta_d}f|$ on $\ell^p(\ZZ)$ for $1<p\le \vc$. Hence, this gives  not only a simpler proof than that in \cite[Theorem 2]{CGRTV} for $1<p<\vc$, but also the boundedness on $\ell^\vc(\ZZ)$.
\end{remark}

\bigskip

We now set
\[
\Phi_k(\theta) = \begin{cases}
	\sin \theta, \ \ \ \ \ &\text{if $k$ is odd},\\
	\cos \theta, \ \ \ \ \ &\text{if $k$ is even},\\
\end{cases} \ \ \ \ \ \ \  \theta\in (0,\pi).
\]
Then we have the following estimate regarding the time derivatives of the kernel $h_z(n)$.
\begin{lemma}\label{lem1-htk}
For any $\ell\in \mathbb N$ and $N>0$, there exists $C=C(\ell,N)$ such that 
	\begin{equation}\label{eq1-htk}
	|\partial^\ell _zh_{z}(n)|\le \f{C}{(|z|\cos\alpha)^{\ell+1/2} }\Big(1+\f{n\cos\alpha}{\sqrt{|z|\cos\alpha }}\Big)^{-N} 
	\end{equation}
for all $n\in \mathbb Z$ and all $z\in \mathbb C_+$ with $|z|\ge 1$, where $\alpha =\arg z$.
\end{lemma}
\begin{proof}
	We first prove \eqref{eq1-htk} for $z = t \ge 1$. To do this, we recall that 
	\[
	h_t(n) =\f{1}{\pi}\int_0^\pi e^{-2t(1-\cos\theta)}\cos(n\theta) d\theta
	\]
	for all  $n\in \mathbb Z$ and $t>0$.
	
	For $\vec m = (m_1,\ldots,m_k)\in \mathbb N^k$, we denote 
	\begin{align*}
		|\vec m|&=m_1+\ldots+m_k,\\
		|\vec m_{\rm odd}|&= m_1 + m_3 + \ldots + m_{2\lfloor (k+1)/2\rfloor-1},\\
		|\vec m_{\rm even}|&= m_2 + m_4 + \ldots + m_{2\lfloor k/2\rfloor}.
	\end{align*}
	Set $\varphi_t(\theta)= e^{-2t(1-\cos\theta)}$. It was proved in \cite[Lemma 2.1]{ABM} that 
	\[
		\partial^k_\theta \varphi_t(\theta)=\varphi_t(\theta)\sum_{(m_1,\ldots,m_k)\in \mathbb N^k\atop m_1+2m_2+\ldots+km_k=k}a_{\vec m} t^{|\vec m|}(\cos\theta)^{|\vec m_{\rm even}|}(\sin\theta)^{|\vec m_{\rm odd}|}
	\]
	for some constants $a_{\vec m}$.
	
	By Leibniz's rule we have,  for each $k\in \mathbb N$,
	\begin{equation}\label{eq-chainrule}
	\begin{aligned}
		\partial^k_\theta\partial^\ell_t \varphi_t(\theta)=\partial^\ell_t\partial^k_\theta \varphi(t,\theta)=&\sum_{j=0}^\ell [-2(1-\cos\theta)]^je^{-2t(1-\cos\theta)}\\
		&\times \sum_{(m_1,\ldots,m_k)\in \mathbb N^k\atop m_1+2m_2+\ldots+km_k=k}a_{\vec m,\ell} t^{|\vec m| -\ell+ j}(\cos\theta)^{|\vec m_{\rm even}|}(\sin\theta)^{|\vec m_{\rm odd}|}.
	\end{aligned}
	\end{equation}
	Note that if $k$ is odd, then
	\[
	\partial^k_\theta\partial^\ell_t \varphi(t,0)=\partial^k_\theta\partial^\ell_t \varphi(t,\pi)=0 
	\]
	for all $t>0$. 
	
	Therefore, by using the integration by parts we have
	\begin{equation}\label{eq-integration by parts fomular for the time derivative of heat kernel}
	\partial_t^\ell h_t(n) =\f{(-1)^{\lfloor (k+1)/2\rfloor}}{\pi n^k}\int_0^\pi \partial_\theta^k\partial_t^\ell  \varphi_t(\theta)\Phi_k(n\theta) d\theta.
	\end{equation}
	This, along with \eqref{eq-chainrule} and the facts that  $(1-\cos\theta)\gtrsim \theta^2$, $|\sin\theta|\le \theta$ and $|\cos\theta|\le 1$ for $\theta\in (0,\pi)$, we have
	\[
	\begin{aligned}
		|\partial_t^\ell h_t(n)|\les \sum_{j=0}^\ell\sum_{(m_1,\ldots,m_k)\in \mathbb N^k\atop m_1+2m_2+\ldots+km_k=k} \f{1}{t^\ell n^k}\int_0^\pi(t\theta^2)^je^{-ct\theta^2} t^{|\vec m|}\theta^{|\vec m_{\rm odd}|}d\theta.
	\end{aligned}
	\]
	Since $t\ge 1$ and $m_1+2m_2+\ldots+km_k=k$, it is easy to see that 
	\begin{equation}\label{eq-compare 1s - proof of lem1}
	t^{|\vec m|}\le t^{k/2}t^{|\vec m_{\rm odd}|/2}.
	\end{equation}
	Therefore,
	\[
	\begin{aligned}
		|\partial_t^\ell h_t(n)|&\les \sum_{j=0}^\ell\sum_{(m_1,\ldots,m_k)\in \mathbb N^k\atop m_1+2m_2+\ldots+km_k=k} \f{t^{k/2}}{t^\ell n^k}\int_0^\pi(t\theta^2)^je^{-ct\theta^2} t^{|\vec m_{\rm odd}|/2}\theta^{|\vec m_{\rm odd}|}d\theta\\
		&\les\f{t^{k/2}}{t^{\ell+1/2} n^{k}}
	\end{aligned}
	\]
	for all $k\in \mathbb N$ and $t\ge 1$.
	
	This follows \eqref{eq1-htk} for the case $z=t\ge 1$.
	
	For the general case $z\in \mathbb C$ with $\Re z\ge 1$, we have
	\[
	h_z(n) =\f{1}{\pi}\int_0^\pi e^{-2z(1-\cos\theta)}\cos(n\theta) d\theta.
	\]
Similarly to \eqref{eq-chainrule}, we have
\begin{equation}\label{eq-chainrule-complex}
	\begin{aligned}
		\partial^k_\theta\partial^\ell_z \varphi_z(\theta) =&\sum_{j=0}^\ell [-2(1-\cos\theta)]^je^{-2z(1-\cos\theta)}\\
		& \sum_{(m_1,\ldots,m_k)\in \mathbb N^k\atop m_1+2m_2+\ldots+km_k=k}a_{\vec m,\ell} z^{|\vec m| -\ell+ j}(\cos\theta)^{|\vec m_{\rm even}|}(\sin\theta)^{|\vec m_{\rm odd}|}.
	\end{aligned}
\end{equation}	
	At this stage, arguing similarly to the real case we have
	\[
	\begin{aligned}
		|\partial_z^\ell h_z(n)|&\les \sum_{j=0}^\ell\sum_{(m_1,\ldots,m_k)\in \mathbb N^k\atop m_1+2m_2+\ldots+km_k=k} \f{|z|^{k/2}}{|z|^\ell n^k}\int_0^\pi(|z|\theta^2)^j |e^{-cz\theta^2}| |z|^{|\vec m_{\rm odd}|/2}\theta^{|\vec m_{\rm odd}|}d\theta\\
		&\les \sum_{j=0}^\ell\sum_{(m_1,\ldots,m_k)\in \mathbb N^k\atop m_1+2m_2+\ldots+km_k=k} \f{|z|^{k/2}}{|z|^\ell n^k}\int_0^\pi(|z|\theta^2)^j e^{-c(|z|\cos\alpha)\theta^2} |z|^{|\vec m_{\rm odd}|/2}\theta^{|\vec m_{\rm odd}|}d\theta\\
		&\les \sum_{j=0}^\ell\sum_{(m_1,\ldots,m_k)\in \mathbb N^k\atop m_1+2m_2+\ldots+km_k=k} \f{|z|^{k/2}}{|z|^\ell n^k}\f{1}{(\cos\alpha)^{j+|\vec m_{\rm odd}|/2}}\\
		& \ \ \ \ \ \ \ \times\int_0^\pi\big[(|z|\cos\alpha)\theta^2\big]^j e^{-c(|z|\cos\alpha)\theta^2} (|z|\cos\alpha)^{|\vec m_{\rm odd}|/2}\theta^{|\vec m_{\rm odd}|}d\theta.
	\end{aligned}
	\]
	We now use the following inequalities
	\[
	(\cos\alpha)^{j+|\vec m_{\rm odd}|/2}\ge (\cos\alpha)^{\ell+k/2}
	\]
	and
	\[
	\int_0^\pi\big[(|z|\cos\alpha)\theta^2\big]^j e^{-c(|z|\cos\alpha)\theta^2} (|z|\cos\alpha)^{|\vec m_{\rm odd}|/2}\theta^{|\vec m_{\rm odd}|}d\theta\les (|z|\cos\alpha)^{-1/2}
	\]
	to obtain
		\[
	\begin{aligned}
		|\partial_z^\ell h_z(n)| &\les  \f{|z|^{k/2}}{|z|^\ell n^k}\f{1}{(\cos\alpha)^{\ell+k/2}}\f{1}{(|z|\cos\alpha)^{1/2}}\\
		&\les  \f{1}{(|z|\cos\alpha)^{\ell+1/2} }\Big(\f{\sqrt{|z|\cos\alpha}}{n\cos\alpha}\Big)^k.
	\end{aligned}
	\]
	This completes our proof.
\end{proof}

\begin{lemma}\label{lem2-htk}
	For any $\ell\in \mathbb N$, there exists $C=C(\ell)$ such that 
	\begin{equation}\label{eq1-htk}
		| \partial^\ell _th_{t}(n)|\le \f{C}{t^{\ell+1/2}}\Big(1+\frac{|n|}{\sqrt{t}}\Big)^{-\ell}
	\end{equation}
	for all $n\in \mathbb Z$ and all $t>0$.	
\end{lemma}
\begin{proof}
	
If $t\ge 1$, then the estimate follows directly from Lemma \ref{lem1-htk}.

Hence, it remains to show the estimate for $0<t<1$. To do this, recall from \eqref{eq2-heat kernel} that 	\[
h_t(n) =\f{1}{\pi}\int_0^\pi e^{-2t(1-\cos\theta)}\cos(n\theta) d\theta
\]
for all  $n\in \mathbb Z$.

Using the fact that $(1-\cos\theta)\gtrsim  \theta^2$ for $\theta\in (0,\pi)$, we have
\begin{equation}\label{eq-bound for ht}
\begin{aligned}
	|h_t(n)| &\le\f{1}{\pi}\int_0^\pi e^{-2ct\theta^2} d\theta\\
	&\les \f{1}{\sqrt{t}}.
\end{aligned}
\end{equation}	
On the other hand, since
\[
\partial^\ell _th_{t}(n) = \Delta_d^\ell h_{t}(n) =\sum_{j=-\ell}^\ell c_jh_{t}(n+j), 
\]
we have
\begin{equation}\label{eq1-lem2 hkz}
	\begin{aligned}
		|\partial^\ell _th_{t}(n)|&\le \f{C}{t^{1/2}}= C\f{t^{\ell}}{t^{\ell+1/2}}
	\end{aligned}
\end{equation}
for $t\in (0,1)$.

On the other hand, arguing similarly to the proof of  Lemma \ref{lem1-htk} in which instead of \eqref{eq-compare 1s - proof of lem1} we use the following inequality
\begin{equation*}\label{eq-compare - proof of lem1}
	t^{|\vec m|}\le t^{|\vec m_{\rm odd}|/2} \ \ \ \text{for $0<t<1$},
\end{equation*}
we can show that for every $a\in (0,1)$,
\begin{equation}\label{eq-partial l ht for  t small}
	\begin{aligned}
		|\partial_t^\ell h_t(n)| 
		&\les  \f{1}{t^{\ell+1/2} }\Big(\f{1}{|n|}\Big)^{ 2\ell}.
	\end{aligned}
\end{equation}
This, along with \eqref{eq1-lem2 hkz}, implies that 
\[
		|\partial_t^\ell h_t(n)| 
\les  \f{1}{t^{\ell+1/2} }\Big(\f{\sqrt{t}}{n}\Big)^{\ell }\sim \f{1}{t^{\ell+1/2} }\Big(1+\f{n}{\sqrt{t}}\Big)^{-\ell }.
\] 
This completes our proof.
\end{proof}


\begin{lemma}\label{lemma- derivative estimate - Taylor}
	For $k,\ell \in \mathbb N$ and $t\ge 1$, we define
	\[
	g_{t,k}(x) =\f{(-1)^{\lfloor (k+1)/2\rfloor}}{\pi x^k}\int_0^\pi \partial_\theta^k\partial_t^\ell  \varphi_t(\theta)\Phi_k(x\theta) d\theta,  \ \ \ x\in \mathbb R\backslash \{0\},
	\]
	and
	\[
	g_{t,0}(x) =\f{1}{\pi}\int_0^\pi e^{-2t(1-\cos\theta)}\cos(x\theta) d\theta,  \ \ \ x\in \mathbb R.
	\]
	
	Then for each $j\in \mathbb N$, $k\in \mathbb N^*$, there exists $C>0$ such that
	\[
	|\partial_x^j g_{t,k}(x)|\le \f{C}{|x|^{j}t^{\ell+1/2} }\Big(\f{|x|}{\sqrt{t}}\Big)^{-k} +  \f{C}{t^{\ell+j/2+1/2} }\Big(\f{|x|}{\sqrt{t}}\Big)^{-k}
	\]
	for all $x\in \mathbb R\backslash\{0\}$.
	
	In particular, if  $k=0$ and $j\in \mathbb N$ then we have
	\[
	|\partial_x^j g_{t,0}(x)|\le  \f{C}{t^{\ell+j/2+1/2} }
	\]
	for all $x\in \mathbb R$.
\end{lemma}
\begin{proof}
	The case $k=0$ is similar and even easier. Therefore, we need only to prove the first estimate for $k\ge 1$.
	
	Indeed, we have
	\[
	\begin{aligned}
		|\partial_x^j g_{t,k}(x)|&\les  \sum_{i=0}^j\Big|\partial_x^i\Big(\f{1}{x^k}\Big)\Big|\int_0^\pi |\partial_\theta^k\partial_t^\ell  \varphi_t(\theta)\theta^j \Phi^{(j-i)}_k(x\theta)| d\theta.
	\end{aligned}
	\]
	It is clear that 
	\[
	\Big|\partial_x^i\Big(\f{1}{x^k}\Big)\Big|\les \f{1}{|x|^{k+i}}, \ \ x\ne 0.
	\]
	On the other hand, arguing similarly to the proof of Lemm \ref{lem2-htk}, we come up with
	\[
	\int_0^\pi |\partial_\theta^k\partial_t^\ell  \varphi_t(\theta)\theta^j \Phi^{(j-i)}_k(x\theta)| d\theta\les \f{t^{k/2}}{t^{\ell+(j-i)/2+1/2}}.
	\]
	Therefore,
	\[
	\begin{aligned}
		|\partial_x^j g_{t,k}(x)|&\les  \sum_{i=0}^j \f{1}{|x|^{k+i}}\f{t^{k/2}}{t^{\ell+(j-i)/2+1/2}} =\sum_{i=0}^j \f{C}{t^{\ell+j/2+1/2} }\Big(\f{|x|}{\sqrt{t}}\Big)^{-k-i}\\
		&\les \f{1}{|x|^{j}t^{\ell+1/2} }\Big(\f{|x|}{\sqrt{t}}\Big)^{-k} +  \f{C}{t^{\ell+j/2+1/2} }\Big(\f{|x|}{\sqrt{t}}\Big)^{-k}
	\end{aligned}
	\]
	for  all $x\in \mathbb R\backslash\{0\}$.
	
	This completes our proof.
\end{proof}

The following estimate is about the difference derivative estimates of the heat kernel. We have:
\begin{lemma}\label{lem-htk and difference derivatives}
	For any $\ell\in \mathbb N$, there exists $C=C(\ell)$ such that 
	\begin{equation}\label{eq1-htk difference derivatives}
		| D\partial^\ell _th_{t}(n)|+| D^*\partial^\ell _th_{t}(n)|\le \f{C}{t^{\ell+1}}\Big(1+\frac{|n|}{\sqrt{t}}\Big)^{-\ell}
	\end{equation}
	for all $n\in \mathbb Z$ and all $t>0$.	
\end{lemma}
\begin{proof}
From \eqref{eq-integration by parts fomular for the time derivative of heat kernel}, we have, for any $k\in \mathbb N$,
	$$
		D\partial_t^\ell h_t(n) =g_{t,k}(n+1)-g_{t,k}(n).
	$$
	where $g_{t,k}$ is the function defined in Lemma \ref{lemma- derivative estimate - Taylor}.
	
	For $t\ge 1$ and $|n|\ge \sqrt t$, applying Lemma  \ref{lemma- derivative estimate - Taylor} and  the mean value theorem  for    $k=\ell$, we have
	\[
	\begin{aligned}
		|D\partial_t^\ell h_t(n)|&\les \f{1}{|n|t^{\ell+1/2} }\Big(\f{|n|}{\sqrt{t}}\Big)^{-\ell} +  \f{1}{t^{\ell+1} }\Big(\f{|n|}{\sqrt{t}}\Big)^{-\ell}\\
		&\les    \f{1}{t^{\ell+1} }\Big(\f{|n|}{\sqrt{t}}\Big)^{-\ell}.
	\end{aligned}
	\] 
		For $t\ge 1$ and $|n|< \sqrt t$, applying Lemma  \ref{lemma- derivative estimate - Taylor} and  the mean value theorem again  but for $k=0$, we have
	\[
	\begin{aligned}
		|D\partial_t^\ell h_t(n)|&\les \f{1}{t^{\ell+1} }.
	\end{aligned}
	\] 
		These two estimates imply that for $t\ge 1$,
	\[
	\begin{aligned}
		|D\partial_t^\ell h_t(n)|&\les \f{1}{t^{\ell+1} }\Big(1+\f{n}{\sqrt{t}}\Big)^{-\ell}.
	\end{aligned}
	\]

	For $t\in (0,1)$, from \eqref{eq-bound for ht}, we have
	\[
	\begin{aligned}
		|D\partial_t^\ell h_t(n)|=|D\Delta_d^\ell h_t(n)| & \les \f{1}{\sqrt t},
	\end{aligned}
	\]
	which implies the desired estimate for $t\in (0,1)$ and $|n|<\sqrt t$.
	
	For $t\in (0,1)$ and $|n|\ge \sqrt t$, from above inequality,
	\begin{equation}\label{eq1-D}
	|D\partial_t^\ell h_t(n)|=|D\Delta_d^\ell h_t(n)|\les \f{t^\ell}{t^{\ell+1}}.
	\end{equation}

	On the other hand, similarly to the proof of Lemma \ref{lemma- derivative estimate - Taylor}, we have, for every $k\in \mathbb N$ and $x\ne 0$,
		\[
	\begin{aligned}
		|\partial_x g_{t,k}(x)|&\les  \f{1}{|x|^{k+1}}\int_0^\pi |\partial_\theta^k\partial_t^\ell  \varphi_t(\theta)\Phi_k(x\theta)| d\theta+\f{1}{|x|^{k}}\int_0^\pi |\partial_\theta^k\partial_t^\ell  \varphi_t(\theta)\theta \Phi'_k(x\theta)| d\theta.
	\end{aligned}
	\]
	Then arguing similarly to \eqref{eq-partial l ht for  t small}, 
	\[
	\begin{aligned}
		|\partial_x g_{t,k}(x)|\les  \f{1}{|x|t^{\ell+1/2} }\Big(\f{1}{|x|}\Big)^{ k}+\f{1}{t^{\ell+1} }\Big(\f{1}{|x|}\Big)^{ k}
	\end{aligned}
	\]
	On the other hand, from \eqref{eq-integration by parts fomular for the time derivative of heat kernel},
	\[
	D\partial^\ell _th_{t}(n)= g_t(n+1) - g_t(n).
	\]
	These two estimates above and the mean value theorem imply that for each $k=2\ell$, 
	\[
	\begin{aligned}
			|D\partial^\ell _th_{t}(n)|&\les  \f{1}{|n|t^{\ell+1/2} }\Big(\f{1}{|n|}\Big)^{ 2\ell}+\f{1}{t^{\ell+1} }\Big(\f{1}{|n|}\Big)^{ 2\ell}\\
	&\les \f{1}{t^{\ell+1} }\Big(\f{1}{|n|}\Big)^{ 2\ell}
	\end{aligned}
		\]
		for $t\in (0,1)$ and $|n|\ge \sqrt t$.
		
		At this stage, interpolating this and \eqref{eq1-D}, we come up with
		\[
				| D\partial^\ell _th_{t}(n)|\le \f{C}{t^{\ell+1}}\Big(1+\frac{|n|}{\sqrt{t}}\Big)^{-\ell}
	\]
for all $t\in (0,1)$ and $|n|\ge \sqrt t$.

Similarly,
		\[
| D^*\partial^\ell _th_{t}(n)|\le \f{C}{t^{\ell+1}}\Big(1+\frac{|n|}{\sqrt{t}}\Big)^{-\ell}
\]
for all $n\in \mathbb Z$ and all $t>0$.

This completes our proof.
\end{proof}

For the higher difference derivative of the heat kernel, denote by $\mathfrak{D}$ either $D$ or $D^*$ and by $\mathfrak{D}^*$ the conjugate of $\mathfrak{D}$. Then we have:
\begin{lemma}\label{lem-htk and HIGHER difference derivatives}
	For any $\ell, k\in \mathbb N$, there exists $C=C(\ell, k)$ such that 
	\begin{equation}\label{eq1-htk HIGHER difference derivatives}
		| \mathfrak{D}^k \partial^{k\ell} _th_{t}(n)|\le \f{C}{t^{k\ell+(k+1)/2}}\Big(1+\frac{|n|}{\sqrt{t}}\Big)^{- \ell}
	\end{equation}
	for all $n\in \mathbb Z$ and all $t>0$.	
\end{lemma}
\begin{proof}
	We need only to give the proof for the case $k=2$. The case $k\ge 3$ can be done by induction. Note that 
	\[
	\mathfrak{D}e^{-t\Delta_d} f= e^{-t\Delta_d} \mathfrak{D}^*f.
	\]
	It follows that 
	\[
	\mathfrak{D}^2\partial^{2\ell}e^{-t\Delta_d}=\mathfrak{D}\partial^{\ell}e^{-\f{t}{2}\Delta_d} \circ \mathfrak{D}^*\partial^{\ell}e^{-\f{t}{2}\Delta_d}.
	\]
	This, along with Lemma \ref{lem-htk and difference derivatives}, implies the desired estimates.
	
	This completes our proof.
\end{proof}
\subsection{A functional calculus of $\Delta_d$}

In this section, we will investigate some kernel estimates of the functional calculus of $\Delta_d$. Firstly, we have:
\begin{lemma}
	\label{lem1-spectral multiplier}
	For any $q\in (4,\vc]$,	 there exists $C=C(q)>0$ such that 
	\[
	\sum_{m\in \mathbb Z} | F(\sqrt{\Delta_d})(m,n)|^2\le CR\|\delta_RF\|_q^2
	\]
	for any Borel function $F$ such that $\supp F\subset [0,R]$ with $R\in (0,2]$, where $\delta_RF=F(R\cdot)$.
\end{lemma}
\begin{proof}
We have
\[
F(\sqrt{\Delta_d})(m,n) = \mathcal F_{\ZZ}^{-1}\big[F(\sqrt{2(1-\cos\theta)})\big](m-n)
\]
Therefore,
\[
\begin{aligned}
\sum_{m\in \mathbb Z} | F(\sqrt{\Delta_d})(m,n)|^2 &=\big\|\mathcal F_{\ZZ}^{-1}\big[F(\sqrt{2(1-\cos\theta)})\big]\big\|^2_{\ell^2(\ZZ)}\\
&=\big\| F(\sqrt{2(1-\cos\theta)}) \big\|^2_{L^2[-\pi,\pi]}\\
&=\int_{-\pi}^\pi |F(\sqrt{2(1-\cos\theta)})|^2 d\theta\\
&=\int_{-\pi}^\pi |F(2\sin(|\theta|/2))|^2 d\theta \sim \int_{0}^{\pi/2} |F(2\sin\theta)|^2 d\theta.
\end{aligned}
\]	
By H\"older's inequality, for $p\in (2,\vc]$,
\[
\begin{aligned}
	\int_{0}^{\pi/2} |F(2\sin\theta)|^2 d\theta& \le \Big[\int_{0}^{\pi/2} |F(2\sin\theta)|^{2p} \cos\theta d\theta\Big]^{1/p}\Big[\int_{0}^{\pi/2}  (\cos\theta)^{1-p'} d\theta\Big]^{1/p'}.
\end{aligned}
\]
Since $\displaystyle \int_{0}^{\pi/2}  (\cos\theta)^{1-p'} d\theta<\vc$ for $p\in (2,\vc]$,
\[
\begin{aligned}
	\int_{0}^{\pi/2} |F(2\sin\theta)|^2 d\theta& \les \Big[\int_{0}^{\pi/2} |F(2\sin\theta)|^{2p} \cos\theta d\theta\Big]^{1/p}\\
	& \sim \Big[\int_{0}^{2} |F(u)|^{2p} du\Big]^{1/p}\\
&\les R\|\delta_R F\|_{2p}^2	
\end{aligned}
\]
for every $p\in (2,\vc]$.

This completes our proof.
\end{proof}

\begin{remark}
	In the continuous case of the Laplacian on $\mathbb R$, Lemma \ref{lem1-spectral multiplier} holds true for $q=2$. In our discrete setting, it is not clear to us if Lemma \ref{lem1-spectral multiplier} 
	holds true for some $q\in [2,4]$.
\end{remark}

\begin{proposition}\label{prop1-spectral multiplier}
	Let $q\in (4,\vc]$, $R\in (0,2]$ and $s>0$. Then for any $\epsilon>0$ there exists a constant $C=C(q, s,\epsilon)$ such that
	\begin{equation}
		\label{eq-prop1-spectral multiplier}
			\sum_{m\in\mathbb Z} |F(\sqrt{\Delta_d})(m,n)|^2(1+R|m-n|)^{2s}\le CR\|\delta_RF\|_{W^{s+\epsilon}_q}^2
	\end{equation}
for all  functions $F\in W^{s+\epsilon}_q(\mathbb R)$ such that $\supp\, F\subset [R/2,R]$, where $\delta_RF=F(R\cdot)$.
\end{proposition}
\begin{proof}
	
	
	
	Set $G(\lambda)=F(R\sqrt \lambda)e^{\lambda}$. Then we have
	$$
	G( \Delta_d/R^2)e^{- \Delta_d/R^2}=\f{1}{2\pi}\int_\mathbb{R} e^{-R^{-2}(1-i\tau)\Delta_d}\widehat{G}(\tau)d\tau,
	$$
	where $\widehat{G}$ is the usual Fourier transform of $G$ on $\mathbb R$.
	
	Therefore,
	\[
	F(\sqrt{\Delta_d})=\f{1}{2\pi}\int_\mathbb{R} e^{-R^{-2}(1-i\tau)\Delta_d}\widehat{G}(\tau)d\tau,
	\]
	which yields that 
	\[
	F(\sqrt{\Delta_d})(m,n)=\f{1}{2\pi}\int_\mathbb{R} \widehat{G}(\tau)h_{(1-i\tau)/R^2}(m-n)d\tau.
	\]
	As a consequence,
	\[
	\begin{aligned}
		\Big(\sum_{m\in\mathbb Z} &|F(\sqrt{\Delta_d})(m,n)|^2(1+R|m-n|)^{2s}\Big)^{1/2}\\
		&\le \int_{\mathbb R} |\widehat{G}(\tau)|\Big(\sum_{m\in\mathbb Z} |h_{(1-i\tau)/R^2}(m-n)|^2(1+R|m-n|)^{2s}\Big)^{1/2}d\tau
	\end{aligned}
	\]

	Using Lemma \ref{lem1-htk}, for a fixed $\kappa>1$, we have
	\[
	\begin{aligned}
		\Big(\sum_{m\in\mathbb Z} &|F(\sqrt{\Delta_d})(m,n)|^2(1+R|m-n|)^{2s}\Big)^{1/2}\\
		&\le \int_{\mathbb R} |\widehat{G}(\tau)|\Big[\sum_{m\in\mathbb Z} R^2 \Big(1+\f{R|m-n|}{\sqrt{1+|\tau|^2}}\Big)^{-2s-\kappa}(1+R|m-n|)^{2s}\Big]^{1/2}d\tau.
	\end{aligned}
	\]
	It is easy to check that 
	\[
	\begin{aligned}
		\sum_{m\in\mathbb Z} R^2 \Big(1+\f{R|m-n|}{\sqrt{1+|\tau|^2}}\Big)^{-2s-\kappa}(1+R|m-n|)^{2s}
		&\le (1+|\tau|^2)^{s+\kappa/2}\sum_{m\in\mathbb Z} R^2  (1+R|m-n| )^{-\kappa}\\
		&\les (1+|\tau|)^{2s+\kappa}R.
	\end{aligned}
	\]
	Therefore, for $\epsilon>0$, by H\"older's inequality,
	\[
	\begin{aligned}
		\Big(\sum_{m\in\mathbb Z} |F(\sqrt{\Delta_d})(m,n)|^2(1+R|m-n|)^{s}\Big)^{1/2}
		&\les \sqrt {R}  \int_{\mathbb R} |\widehat{G}(\tau)| (1+|\tau|)^{s+\kappa/2} d\tau\\
			&\les \sqrt R \Big[\int_{\mathbb R} |\widehat{G}(\tau)|^2(1+|\tau|)^{2s+\kappa+1+\epsilon} d\tau\Big]^{1/2}\Big[\int_{\mathbb R} (1+|\tau|)^{-1-\epsilon} d\tau\Big]^{1/2} \\
		&\les \sqrt R\|G\|_{W_2^{s +(1+\kappa+\epsilon)/2}}\\
		&\les \sqrt R\|G\|_{W_q^{s +(1+\kappa+\epsilon)/2}}.
	\end{aligned}
	\]
	From this and Lemma \ref{lem1-spectral multiplier}, by using the trick of interpolation as in \cite{DOS, MM}, we deduce \eqref{eq-prop1-spectral multiplier}.

This completes our proof.	

\end{proof}

\begin{lemma}
	\label{lem1-pointwise- spectral multiplier}
	For $q\in (4,\vc]$, there exists $C=C(q)>0$ such that 
	\[
 | F(\sqrt{\Delta_d})(m,n)| \le CR\|F\|_q
	\]
	for any Borel function $F$ such that $\supp F\subset [0,R]$ with $R\in (0,2]$.
\end{lemma}
\begin{proof}
	We have
	\[
	F(\sqrt{\Delta_d}) = G(\sqrt{\Delta_d})e^{-\Delta_d/R^2},
	\]
	where $G(\lambda)=F(\lambda)e^{\lambda^2/R^2}$.
	
	This, along with H\"older's inequality, Lemma \ref{lem1-spectral multiplier}, Lemma \ref{lem1-ht} and Lemma \ref{lem-elementary}, implies that 
	\[
	\begin{aligned}
		|F(\sqrt{\Delta_d})(m,n)|&=\Big|\sum_{k\in \ZZ}G(\sqrt{\Delta_d})(m,k)h_{R^{-2}}(n-k)\Big|\\
		&\le \|G(\sqrt{\Delta_d})(m,\cdot)\|_{\ell^2(\ZZ)}\|h_{R^{-2}}\|_{\ell^2(\ZZ)}\\
		&\le R\|G\|_q \sim R\|F\|_q.
	\end{aligned}
	\]
	
		This completes our proof.
\end{proof}

\begin{proposition}\label{prop1-pointwise-spectral multiplier}
	Let $q\in (4,\vc]$, $R\in (0,2]$ and $s>0$. Then for any $\epsilon>0$ there exists a constant $C=C(q,s,\epsilon)$ such that
	\begin{equation}
		\label{eq-prop1-pointwise-spectral multiplier}
		 |F(\sqrt{\Delta_d})(m,n)|\le  R(1+R|m-n|)^{-s}\|\delta_RF\|_{W^{s+\epsilon}_q}
	\end{equation}
	for all  functions $F\in W^{s+\epsilon}_q(\mathbb R)$ such that $\supp\, F\subset [R/2,R]$.
\end{proposition}
\begin{proof}
	
	
We follow idea in \cite{DOS, MM}. Set $G(\lambda)=F(R\sqrt \lambda)e^{\lambda}$. Then we have
	$$
	G( \Delta_d/R^2)e^{- \Delta_d/R^2}=\f{1}{2\pi}\int_\mathbb{R} e^{-R^{-2}(1-i\tau)\Delta_d}\widehat{G}(\tau)d\tau,
	$$
	where $\widehat{G}$ is the usual Fourier transform of $G$ on $\mathbb R$.
	
	Therefore,
	\[
	F(\sqrt{\Delta_d})=\f{1}{2\pi}\int_\mathbb{R} e^{-R^{-2}(1-i\tau)\Delta_d}\widehat{G}(\tau)d\tau,
	\]
	which yields that 
	\[
	F(\sqrt{\Delta_d})(m,n)=\f{1}{2\pi}\int_\mathbb{R} \widehat{G}(\tau)h_{(1-i\tau)/R^2}(m-n)d\tau.
	\]
	Using Lemma \ref{lem1-htk} and H\"older's inequality, we have, for $\epsilon>0$,
	\[
	\begin{aligned}
		|F(\sqrt{\Delta_d})(m,n)|
		&\le \int_{\mathbb R} |\widehat{G}(\tau)| R \Big(1+\f{R|m-n|}{\sqrt{1+|\tau|^2}}\Big)^{-s}  d\tau\\
				&\les R \big(1+ R|m-n| \big)^{-s}  \int_{\mathbb R} |\widehat{G}(\tau)|(1+|\tau|)^{s}   d\tau\\
		&\les R \big(1+ R|m-n| \big)^{-s}\Big[\int_{\mathbb R} |\widehat{G}(\tau)|^2(1+|\tau|)^{2s+1+\epsilon} d\tau\Big]^{1/2}\Big[\int_{\mathbb R} (1+|\tau|)^{-1-\epsilon} d\tau\Big]^{1/2} \\
		&\les R\big(1+ R|m-n| \big)^{-s}\|G\|_{W_2^{s+(1+\epsilon)/2}}.
\end{aligned}
\]
Since $G(\lambda)=F(R\sqrt \lambda)e^{\lambda}$ and $\supp F\subset [R/2, R]$, we have $\|G\|_{W_2^{s+(1+\epsilon)/2}}\sim \|\delta_R F\|_{W_2^{s+(1+\epsilon)/2}}\les \|\delta_R F\|_{W_q^{s+(1+\epsilon)/2}}$ for $q\in (4,\vc]$. Therefore,
 	\[	|F(\sqrt{\Delta_d})(m,n)|\les R\big(1+ R|m-n| \big)^{-s}\|\delta_R F\|_{W_q^{s+(1+\epsilon)/2}}.
		\]
	From this and Lemma \ref{lem1-spectral multiplier}, by using the trick of interpolation as in \cite{DOS, MM}, we deduce \eqref{eq-prop1-pointwise-spectral multiplier}.
	
	This completes our proof.	

\end{proof}

 Proposition \ref{prop1-pointwise-spectral multiplier} directly deduces the following corollary.
\begin{corollary}
	\label{lem1}
	Let $\psi\in \mathcal{S}(\mathbb R)$ be a Schwartz function supported in $[2,8]$. Then we have
	\[
	|\psi(t\sqrt{\Delta_d})(m,n)|\le \f{C_N}{t}\Big(1+\f{|m-n|}{t}\Big)^{-N}
	\]
	for all $N, t>0$ and $m,n\in \ZZ$.
\end{corollary}	

\begin{lemma}
	\label{lem1-derivative of psi delta}
	Let $\psi\in \mathcal{S}(\mathbb R)$ be a Schwartz function  supported in $[2,8]$. Then we have
	\[
	|\mathfrak{D}^k\psi(t\sqrt{\Delta_d})(m,n)|\le \f{C_{N,k}}{t^{k+1}}\Big(1+\f{|m-n|}{t}\Big)^{-N}
	\]
	for all $N, t>0$, $k\in \mathbb N$ and $m,n\in \ZZ$.
\end{lemma}	
\begin{proof}
	We write
	\[
	\mathfrak{D}^k\psi(t\sqrt{\Delta_d}) =\mathfrak{D}^k(t^2\Delta_d)^{kN}e^{-t^2\Delta_d}\circ (t^2\Delta_d)^{kN}e^{-t^2\Delta_d}\psi(t\sqrt{\Delta_d}).  
	\]
	By Lemma \ref{lem-htk and difference derivatives} and Corollary \ref{lem1},
	\[
	|K_{\mathfrak{D}^k(t^2\Delta_d)^{kN}e^{-t^2\Delta_d}}(m,n)|\les \f{C_{N,k}}{t^{k+1}}\Big(1+\f{|m-n|}{t}\Big)^{-N}
	\]
	and
	\[
	|K_{(t^2\Delta_d)^{kN}e^{-t^2\Delta_d}\psi(t\sqrt{\Delta_d})}(m,n)|\les \f{C_{N,k}}{t}\Big(1+\f{|m-n|}{t}\Big)^{-N}
	\]
	for all $N, t>0$, $k\in \mathbb N$ and $m,n\in \ZZ$.
	
	It follows that
		\[
	|\mathfrak{D}^k\psi(t\sqrt{\Delta_d})(m,n)|\le \f{C_{N,k}}{t^{k+1}}\Big(1+\f{|m-n|}{t}\Big)^{-N}
	\]
	for all $N, t>0$, $k\in \mathbb N$ and $m,n\in \ZZ$.
	
	This completes the proof.
\end{proof}

\subsection{Distributions}\label{sec: distributions}

The class of test functions $\mathcal{S}(\ZZ)$  is defined as the set of all functions $f$ such that
\begin{equation}
	\label{Pml norm}
	 \|f\|_m:=\sup_{n\in \ZZ}(1+|n|)^m\sup_{0\le \ell\le m}|\Delta_d^\ell f(n)|<\vc, \ \ \forall m>0, \ell \in \mathbb{N}.
\end{equation}
It is easy to see that the norm $\|\cdot\|_m$ is equivalent to the norm $\|\cdot\|^*_m$ defined by
\[
	 \|f\|^*_m:=\sup_{n\in \ZZ}(1+|n|)^m|f(n)|<\vc, \ \ \forall m>0, \ell \in \mathbb{N}.
\]
It is easy to see  that $\mathcal{S}(\ZZ)$ is a complete locally convex space with topology generated by the family of semi-norms $\{\|f\|_m: \, m>0, \ell \in \mathbb{N}\}$. As usual, we define  the space of distributions $\mathcal{S}'(\ZZ)$ as the set of all continuous linear functional on $\mathcal{S}(\ZZ)$. It can be verified that $\mathcal{S}'(\ZZ)$ is the set of all functions $f$ such that 
\[
|f(n)|\le C_N (1+|n|)^N
\]
for all $n\in \ZZ$ and for some $N>0$.

It is obvious that if $f\in \mathcal S(\mathbb Z)$, then $\Delta_d^kf\in \mathcal S(\mathbb Z)$ for every $k\in \mathbb Z$. From Lemma \ref{lem1-ht}, we obtain $h_t(\cdot)\in \mathcal S(\ZZ)$ for every $t>0$, and hence $\partial_t^k h_t(\cdot)= (-\Delta_d)^kh_t(\cdot)\in \mathcal S(\ZZ)$ for every $t>0$ and $k\in \mathbb N$.

We also define 
$$
\mathcal P(\ZZ) = \{f: \mathbb Z \to \mathbb R: f(n) = p (n), n\in \mathbb Z \ \ \text{for some polynomial $p$}    \}.
$$
It is not difficult to show that 
\[
\mathcal P(\ZZ)=\{f\in \mathcal S'(\ZZ): \exists m \in \mathbb N, \Delta_d^m f =0\}.
\]

We now define the space $\mathcal{S}_\vc(\ZZ)$ as the set of all functions $f \in \mathcal{S}(\ZZ)$ such that for each $k\in \mathbb{N}$ there exists $g_k\in \mathcal{S}(\ZZ)$ so that $f=\Delta_d^kg_k$. Note that such an $g_k$, if exists, is unique. The topology in $\mathcal{S}_\vc(\ZZ)$ is generated by the following family of semi-norms 
\[
\|f\|_{m,k}=\|g_k\|_m, \ \forall m, k\in \mathbb{N}
\]
where $f=\Delta_d^k g_k$.

We will show the following important result.

\begin{lemma}\label{lem1-distribution}
Let $f\in \mathcal S_\infty(\ZZ)$. Suppose that $f = \Delta_d^k g$ for some $k\in \mathbb N$ and $g\in \mathcal S(\ZZ)$. Then for every $m\in \mathbb N$, 
\[
\|g\|_m\les \|f\|_{m+4k}.
\]
\end{lemma}
\begin{proof}
	By iteration, it suffices to prove the lemma for $k=1$. Assume that $f= \Delta_d g$. We need to claim that 
	\[
	\|g\|_m\les \|f\|_{m+4}
	\]
	for every $m\in \mathbb N$.
	
	Recall that  $D$ and $D^*$ are operators defined by
	\[
	Df(n) = f(n+1) -f (n), \ \ \ \  D^*f(n) = f (n)-f(n-1).
	\]
	Then we have $-\Delta_d =D^*D=DD^*$.
	
	Therefore, we can reduce to prove the following inequality
	\begin{equation}\label{eq0-norm equivalent Dh f}
		\sup_{n\in \mathbb Z} (1+|n|)^m |h(n)|\le C_m \min\Big\{ \sup_{n\in \mathbb Z} (1+|n|)^{m+2} |Dh(n)|, \sup_{n\in \mathbb Z} (1+|n|)^{m+2} |D^*h(n)|\Big\}
	\end{equation}
	for every $m\in \mathbb N$ and $h\in \mathcal S(\ZZ)$.
	
	We only prove the inequality \eqref{eq0-norm equivalent Dh f} for $Dh$ since the inequality for $D^*h$ can be done similarly.
	
	Let $h\in \mathcal S(\ZZ)$. Define $f= D h$ with $h\in \mathcal S(\ZZ)$. We need to prove that  
	\begin{equation}\label{eq1-norm equivalent Dh f}
	\sup_{n\in \mathbb Z} (1+|n|)^m |h(n)|\le C_m \sup_{n\in \mathbb Z} (1+|n|)^{m+2} |f(n)|
	\end{equation}
	for every $m\in \mathbb N$.
	
	Fix $n\in \ZZ$. If $n>0$, then we have
	$$
	h(n) =\sum_{k=n-1}^\vc f(k),
	$$
	which implies that 
	$$
	\begin{aligned}
		(1+|n|)^m |h(n)|&\le \sum_{k=n-1}^\vc (1+|n|)^m|f(k)|\\
		&\le \sum_{k=n-1}^\vc \f{(1+|k|)^{m+2}|f(k)|}{(1+|k|)^{2}}\\
		&\le \sup_{k\in \mathbb Z}(1+|k |)^{m+2} |f(k)|,
	\end{aligned}
	$$
	which implies \eqref{eq1-norm equivalent Dh f}.

If $n\le 0$, then we write
\[
h(n) =\sum_{k=n-1}^{-\vc} f(k).
\]
Similarly, we also obtain that 
$$
(1+|n|)^m |h(n)| \le \sup_{k\in \mathbb Z}(1+|k |)^{m+2} |f(k)|.
$$
This proves \eqref{eq1-norm equivalent Dh f}.

Therefore, this completes our proof.
\end{proof}

We  denote by $\mathcal{S}_\vc'(\ZZ)$ the set of all continuous linear functionals on $\mathcal{S}_\vc(\ZZ)$. We then have: 
\begin{proposition}
	The
	following
	identification
	is
	valid $\mathcal S'(\ZZ)/\mathcal P(\ZZ) = \mathcal{S}_\vc'(\ZZ)$.
\end{proposition}
\begin{proof}
	From Lemma \ref{lem1-distribution}, we have $\mathcal{S}_\vc(\ZZ)\hookrightarrow  \mathcal S(\ZZ)$. Consequently, if $f\in \mathcal{S}_\vc'(\ZZ)$, by the Hahn-Banach theorem, we can extend $f$ to $\widetilde f\in \mathcal S'(\ZZ)$ and the equivalence class, which will be denoted by $F\in \mathcal S'(\ZZ)/\mathcal P(\ZZ)$. Hence, $\mathcal{S}_\vc'(\ZZ)\subset  \mathcal S'(\ZZ)/\mathcal P(\ZZ)$.
	
	Conversely, let $F\in \mathcal S'(\ZZ)/\mathcal P(\ZZ)$. Assume that $f_1, f_2\in F$, then there exists $p\in \mathcal P(\ZZ)$ such that $f_1=f_2 +p$. Then we have, for each $g\in \mathcal S_\vc(\ZZ)$,
	\[
	\begin{aligned}
		\langle f_1, g\rangle &= \langle f_2 + p, g\rangle\\
		&= \langle f_2 , g\rangle + \langle p , g\rangle\\
		&= \langle f_2 , g\rangle.
	\end{aligned}
	\]
	It follows that we can associate with $F$ a unique bounded linear functional in $\mathcal S_\vc(\ZZ)$ defined by 
	\[
	F(g)=\langle f, g\rangle
	\]
	for every $f\in F$ and $g\in \mathcal S_\vc(\ZZ)$. It follows that $\mathcal S'(\ZZ)/\mathcal P(\ZZ) \subset \mathcal S_\vc'(\ZZ)$.
	
	This completes our proof.
\end{proof}
\medskip

From Corollary \ref{lem1}, let $\psi\in \mathscr{\mathbb R}$ with $\supp \psi\subset [2,8]$. Then $\psi(t\sqrt{\Delta_d})(\cdot, m), \psi(t\sqrt{\Delta_d})(m,\cdot)\in \mathcal S_\vc(\ZZ)$ for every $t>0$ and $m\in \ZZ$. Therefore, for any $f\in \mathcal S'_\vc(\ZZ)$ we can define
\[
\psi(t\sqrt{\Delta_d})f(n) = \langle \psi(t\sqrt{\Delta_d})(n,\cdot), f\rangle.
\]

\bigskip

In what follows, by a ``\textit{partition of unity}'' we shall mean a function $\psi\in \mathcal{S}(\mathbb{R})$ such that $\supp\psi\subset[2,8]$, $\int\psi(\xi)\,\f{d\xi}{\xi}\neq 0$ and
$$\sum_{j\in \mathbb{Z}}\psi_j(\lambda)=1 \textup{ on } (0,\infty), $$
where $\psi_j(\lambda):=\psi(2^{-j}\lambda)$ for each $j\in \mathbb{Z}$.

We have the following result which can be viewed as the Calder\'on reproducing formula regarding the functional calculus of the discrete Laplacian $\Delta_d$. 
\begin{proposition}
	\label{prop-Calderon1}
	Let $\psi$ be a partition of unity. Then for any $f\in \mathcal{S}_\vc'(\ZZ)$ we have
	\[
	f= \sum_{j\in \mathbb{Z}^-}\psi_j(\sqrt{\Delta_d})f \ \ \text{in $\mathcal{S}'(\ZZ)/\mathcal P(\ZZ)$}.
	\]
	Similarly, for any $f\in \mathcal{S}_\vc'(\ZZ)$ we have
	\[
	f= c_\psi \int_0^\vc \psi(t\sqrt{\Delta_d})f \f{dt}{t} \ \ \text{in $\mathcal{S}'(\ZZ)/\mathcal P(\ZZ)$},
	\]
	where $c_\psi =\Big[\int_0^\vc \psi(s)  \f{ds}{s}\Big]^{-1}$.
	
\end{proposition}
\begin{proof}
	By duality it suffices to prove that for each $f\in \mathcal{S}_\vc(\ZZ)$,
	\[
	f= \sum_{j\in \mathbb{Z}^-}\psi_j(\sqrt{\Delta_d})f \ \ \text{in $\mathcal{S}_\vc(\ZZ)$}.
	\]
	
	To do this, we fix $N>0$ and $k\in \mathbb N$. Since $f\in \mathcal{S}_\vc(\ZZ)$,  there exists $g_k\in \mathcal{S}(\ZZ)$ so that $f=\Delta_d^{k+\ell} g_k$, where $\ell\in \mathbb N$ will be fixed later. Then we have,	  
$$
\begin{aligned}
	\|\psi_j(\sqrt{\Delta_d})f\|_{N,k}&=\|\Delta_d^\ell\psi_j(\sqrt{\Delta_d})g_k\|_{N}\\
	&=2^{2j\ell} \|\psi_{j,\ell}(\sqrt{\Delta_d})g_k\|_{N}\\
	&=2^{2j\ell}\sup_{n\in \ZZ}\sup_{0\le i\le N}(1+|n|)^N |\Delta_d^i \psi_{j,\ell}(\sqrt{\Delta_d})g_k(n)|\\
	&=2^{2j(\ell+i)}\sup_{n\in \ZZ}\sup_{0\le i\le N}(1+|n|)^N | \psi_{j,\ell+i}(\sqrt{\Delta_d})g_k(n)|,
\end{aligned}
$$
where $=\psi_{j,l}(\lambda)=(2^j\lambda)^{2l}\psi_j(\lambda)$ for $l\in \mathbb N$.

By Lemma \ref{lem1} we have
\[
|\psi_{j,\ell+i}(\sqrt{\Delta_d})(n,m)|\le \f{C_N}{2^{-j}}\Big(1+\f{|m-n|}{2^{-j}}\Big)^{-N-2}.
\]
Hence, 
$$
\begin{aligned}
 (1+|n|)^N | \psi_{j,\ell+i}(\sqrt{\Delta_d})g_k(n)|&\les  (1+|n|)^N  \sum_{m\in \ZZ} \f{1}{2^{-j}}\Big(1+\f{|m-n|}{2^{-j}}\Big)^{-N-2}|g_k(m)|\\
 &\les  \|g\|_N  \sum_{m\in \ZZ} \f{1}{2^{-j}}\Big(1+\f{|m-n|}{2^{-j}}\Big)^{-N-2}(1+|m-n|)^N\\
 &\les  2^{-jN}\|g_k\|_N  \sum_{m\in \ZZ} \f{1}{2^{-j}}\Big(1+\f{|m-n|}{2^{-j}}\Big)^{-2} \ \ \ \ \ \ \ \text{(since $j\le 0$)}\\
 &\les  2^{-jN}\|g_k\|_N.
\end{aligned}
$$
Using Lemma \ref{lem1-distribution},
\[
 (1+|n|)^N | \psi_{j,\ell+i}(\sqrt{\Delta_d})g_k(n)|\les 2^{-jN}\|f\|_{N +2(k+\ell)}
\]
for all $0\le i\le N$.

As a consequence, taking $\ell>(N+1)/2$, we have
\[
\begin{aligned}
\|\psi_j(\sqrt{\Delta_d})f\|_{N,k}&\les 2^{2j(\ell+i -N/2)} \|f\|_{N +2(k+\ell)}\\
&\les 2^{j}\|f\|_{N +2(k+\ell)}.
\end{aligned}
	\]
This, along with  the fact that $\mathcal{S}_\vc(\ZZ)$ is complete, we deduce that there exists $h\in \mathcal{S}_\vc(\ZZ)$ so that 
\[
h= \sum_{j\in \mathbb{Z}^-}\psi_j(\sqrt{\Delta_d})f \ \ \text{in $\mathcal{S}_\vc(\ZZ)$}.
\]
On the other hand, by spectral theory we have
\[
f= \sum_{j\in \mathbb{Z}^-}\psi_j(\sqrt{\Delta_d})f \ \ \text{in $\ell^2(\ZZ)$}.
\]
Therefore, $f\equiv h$ and this concludes the proposition.

The second identity can be done similarly.

This completes our proof.
\end{proof}

\section{Hardy spaces}
The theory of Hardy space $H^p(\ZZ)$, $0<p\le 1$ was investigated in \cite{BZ,Boza, KS}. In these papers, the atomic decomposition and the molecular decomposition of the Hardy space $H^p(\ZZ)$ were obtained. The decomposition was inspired from the corresponding decomposition of the classical Hardy spaces. However, the existing decomposition is not suitable to our purpose in studying the boundedness of the spectral multipliers of the discrete Laplacian $\Delta_d$. In this section, we will approach the Hardy space $H^p(\ZZ)$ via new molecular decomposition. This approach is inspired from those in \cite{DY, HM, HLMMY, AMR}.

\subsection{Molecular decomposition}
Consider the area square function $S$ associated to $\Delta_d$ defined by
\begin{equation}
	\label{eq-area integral}
	S_{\Delta_d}f(n)=\Big(\int_0^\vc \sum_{m:|m-n|<t}\big|t^2\Delta_de^{-t^2\Delta_d}f(m)\big|^2\f{dt}{t(t+1)}\Big)^{1/2}.
\end{equation}
Note that the square function $S_{\Dd}f$ is well-defined for distributions $f\in \mathcal S'(\ZZ)$. However, similarly to the classical case, the whole class of distributions $\mathcal S'(\ZZ)$ is too large for our later purpose.  In fact, we need only distributions which vanish weakly at infinity in the following sense.
\begin{definition}\label{defn-vanish at infty}
	We say that a distribution $f\in \mathcal S'(\ZZ)$ vanishes weakly at infinity if $(t\Delta_d)^ke^{-t\Delta_d}f\to 0 $ in $\mathcal S'(\ZZ)$ as $t\to \vc$ for every $k\in \mathbb N$. 
\end{definition}

For $0<p\le 1$, we define the Hardy space $H^p_{S_{\Delta_d}}(\ZZ)$ to be the set of $f\in \mathcal S'$ vanishing weakly at infinity such that 
\[
\|f\|_{H^p_{S_{\Delta_d}}(\ZZ)} := \big\|S_{\Delta_d}f\big\|_{\ell^p(\ZZ)}<\vc.
\]

In \cite{HM, HLMMY, JY}, the Hardy spaces associated to operators were defined via the completion of the set of $L^2$-functions satisfying certain conditions. In this paper, by using the space of distribution $\mathcal S'$ we can define the Hardy spaces directly.

We now adapt ideas in \cite{HLMMY, JY} to define the molecule associated to the discrete Laplacian $\Delta_d$.
\begin{definition}\label{def: L-atom}
	Let $\epsilon>0$, $0<p\le 1$ and $M\in \mathbb{N}$. A function $a(n)$  is called a  $(p,M,\epsilon)_{\Delta_d}$-molecule associated to an interval $I \subset \ZZ$ if there exists a function $b$ such that
	\begin{enumerate}
		\item[{\rm (i)}]  $a=\Delta_d^M b$;
		\item[{\rm (ii)}] $\|\Delta_d^{k}b\|_{\ell^2(S_j(I))}\leq
		2^{-j\epsilon}\ell(I)^{2(M-k)}|2^jI|^{1/2-1/p}$ for all $k=0,1,\dots,M$ and $j=0,1,2\ldots$.
	\end{enumerate}
	
\end{definition}

\begin{definition}
	
	Given  $\epsilon>0$, $0<p\le 1$ and $M\in \mathbb{N}$, we  say that $f=\sum
	\lambda_ja_j$ is a molecule $(p,M,\epsilon)_{\Delta_d}$-representation if
	$\{\lambda_j\}_{j=0}^\infty\in \ell^p$, each $a_j$ is a $(p,M,\epsilon)_{\Delta_d}$-molecule,
	and the sum converges in $\mathcal S'(\ZZ)$. The space $H^{p}_{\Delta_d,{\rm mol},M,\epsilon}(\ZZ)$ is then defined by
	\[
	H^{p}_{\Delta_d,{\rm mol},M,\epsilon}(\ZZ):=\left\{f\in \mathcal S'(\ZZ):f \ \text{has a molecule
		$(p,M,\epsilon)_{\Delta_d}$-representation}\right\},
	\]
	with the norm given by
	$$
	\|f\|^p_{H^{p}_{\Delta_d,{\rm mol},M,\epsilon}(\ZZ)}=\inf\left\{ \sum|\lambda_j|^p :
	f=\sum \lambda_ja_j \ \text{is a molecule $(p,M,\epsilon)_{\Delta_d}$-representation}\right\}.
	$$
\end{definition}

By the standard argument, it can be verified that $H^{p}_{\Delta_d,{\rm mol},M,\epsilon}(\ZZ)$ is complete. We will show that the Hardy spaces $H^p(\ZZ)$ and $H^{p}_{\Delta_d,{\rm mol},M,\epsilon}(\ZZ)$ are coincide. To do this, we need the following technical ingredients.

\begin{proposition}\label{prop- M is bounded from Hp to Lp}
	Let $p\in (0,1]$, $M>\f{1}{2}\big(\f{1}{p}-1\big)$ and $\epsilon>0$. For each $\ell\in \mathbb N$, the operator 
	$$
	\mathcal{M}^*_\ell f(n) =\sup_{t>0}\sup_{|m-n|<t}|(t^2\Delta)^\ell e^{-t^2\Delta_d}f(m)|
	$$
	is bounded from $H^{p}_{\Delta_d,{\rm mol},M,\epsilon}(\ZZ)$ to $\ell^p(\ZZ)$.
\end{proposition}
\begin{proof}
	Since the kernel $h_{t,\ell}( \cdot) \in \mathcal S(\ZZ)$ for each $t>0$ and $\ell \in \mathbb N$, it suffices to show that there exists $C>0$ such that 
	\[
	\|\mathcal{M}^*_\ell a\|_{\ell^p(\ZZ)}\le C
	\]
	for every $(p,M,\epsilon)_{\Delta_d}$ molecule $a$.
	
	Suppose that $a$ is a $(p,M,\epsilon)_{\Delta_d}$ associated to an interval $I$. Then we have
	\[
	\begin{aligned}
		\|\mathcal{M}^*_\ell a\|_{\ell^p(\ZZ)}^p &\les \big\|\mathcal{M}^*_\ell(I-e^{-\ell(I)^2\Delta_d})^M a\big\|_{\ell^p(\ZZ)}^p + \big\|\mathcal{M}^*_\ell\big[I-(I-e^{-\ell(I)^2\Delta_d})^M\big] a\big\|_{\ell^p(\ZZ)}^p\\
		&=: E+ F.
	\end{aligned}
	\]
	
	We  take care of the term $E$ first. To do this, we write
	\[
	\begin{aligned}
		E &\les \sum_{j\ge 0} \big\|\mathcal{M}^*_\ell(I-e^{-\ell(I)^2\Delta_d})^M (a\cdot 1_{S_j(I)})\big\|_{\ell^p(\ZZ)}^p\\
		&=: \sum_{j\ge 0} E_j.
	\end{aligned}
	\]
	For each $j\ge 0$, by H\"older's inequality, 
	\[
	\begin{aligned}
		E_j&\le \sum_{k\ge 0}\|\mathcal{M}^*_\ell(I-e^{-\ell(I)^2\Delta_d})^M (a\cdot 1_{S_j(I)})\|_{\ell^p(S_k(2^jI))}^p\\
		&\le \sum_{k\ge 0}|2^{k+j}I|^{\f{2-p}{2}}\|\mathcal{M}^*_\ell(I-e^{-\ell(I)^2\Delta_d})^M (a\cdot 1_{S_j(I)})\|_{\ell^2(S_k(2^jI))}^p\\
		&=: \sum_{k\ge 0}E_{jk}.
	\end{aligned}
	\]
	For $k=0,1,2,3$, by the $L^2$-boundedness of $\mathcal{M}^*_\ell(I-e^{-\ell(I)^2\Delta_d})^M$ we have
	\[
	\begin{aligned}
		E_{jk}&\le |2^jI|^{\f{2-p}{2}}\|\mathcal{M}^*_\ell(I-e^{-\ell(I)^2\Delta_d})^M (a\cdot 1_{S_j(I)})\|_{\ell^2(\ZZ)}^p\\
		&\les |2^jI|^{\f{2-p}{2}}\|a\|_{\ell^2(S_j(I))}^p\\
		&\les 2^{-j\epsilon}|2^jI|^{\f{2-p}{2}}|2^jI|^{\f{p-2}{2}} \sim  2^{-(j+k)\epsilon}.
	\end{aligned}
	\]
	For $k\ge 4$, note that 
	\[
	(I-e^{-\ell(I)^2\Delta_d})^M =\int_{[0,\ell(I)^2]^M}\Delta^M_de^{-|\vec s|\Delta_d}d\vec s,
	\]
	where $d\vec{s} = ds_1\ldots ds_M$ and $|\vec s| = s_1+\ldots+s_M$.
	
	Then we have, for $n\in S_k(2^j(I))$,
	\[
	\begin{aligned}
		\mathcal{M}^*_\ell(I-e^{-\ell(I)^2\Delta_d})^{M+\ell} (a\cdot 1_{S_j(I)})(n)&\le \sup_{t>0}\sup_{|m-n|<t} \int_{[0,\ell(I)^2]^M}t^{2\ell}|\Delta_d^{M+\ell}e^{-(t^2+|\vec s|)\Delta_d}(a\cdot 1_{S_j(I)})(m)|d\vec s\\
		&\le \sup_{t>2^{k+j}\ell(I)/4}\ldots + \sup_{0<t\le 2^{k+j}\ell(I)/4}\ldots 
	\end{aligned}
	\]
	For $n\in S_k(2^j(I))$ and $m\in I$, we have $t>|m-n|\ge 1$. Hence, applying Lemma \ref{lem1-htk},  we conclude that
	\[
	\begin{aligned}
		\sup_{t>2^{k+j}\ell(I)/4}\sup_{|m-n|<t}& \int_{[0,\ell(I)^2]^M}t^{2\ell}|\Delta_d^{M+\ell}e^{-(t^2+|\vec s|)\Delta_d}(a\cdot 1_{S_j(I)})(m)|d\vec s\\
		&\les \sup_{t>2^{k+j}\ell(I)/4}\sup_{|m-n|<t} \int_{[0,\ell(I)^2]^M}\sum_{l\in S_j(I)}\f{t^{2\ell}}{(t^2+|\vec s|)^{M+\ell+1/2}} |a(l)|d\vec s\\
		&\les 2^{-(k+j)(2M+1)} |I|^{-1} \|a\|_{\ell^1(S_j(B))},
	\end{aligned}
	\]
	and
	\[
	\begin{aligned}
		\sup_{0<t\le 2^{k+j}\ell(I)/4}&\sup_{|m-n|<t} \int_{[0,\ell(I)^2]^M}t^{2\ell}|\Delta_d^{M+\ell}e^{-(t^2+|\vec s|)\Delta_d}(a\cdot 1_{S_j(I)})(m)|d\vec s\\
		&\les \sup_{0<t\le 2^{k+j}\ell(I)/4}\sup_{|m-n|<t} \int_{[0,\ell(I)^2]^M}\sum_{l\in S_j(I)}\f{t^{2\ell}}{(t^2+|\vec s|)^{M+\ell+1/2}}\Big(\f{\sqrt{t^2+|\vec s|}}{|m-l|}\Big)^{2M+2\ell+1} |a(l)|d\vec s\\
		&\les 2^{-(k+j)(2M+1)} |I|^{-1} \|a\|_{\ell^1(S_j(B))} \ \ \ \ \ \text{(since $|m-l|\sim 2^{k+j}\ell(I)$)}.
	\end{aligned}
	\]
	It follows that 
	\[
	\begin{aligned}
		E_{jk}&\les 2^{-(k+j)(2M+1)p}|2^{k+j}I|^{\f{2-p}{2}}|I|^{-p} \|a\|_{\ell^1(S_j(B))}^p|2^{j+k}I|^{p/2}\\
		&\les 2^{-(k+j)[(2M+1)p-1]}.
	\end{aligned}
	\]
	As a consequence,
	\[
	E \les \sum_{j,k\ge 0}2^{-(j+k)\epsilon}\sim 1,
	\]	
	as long as $M>\f{1}{2}\big(\f{1}{p}-1\big)$.

	It remains to show that $F\les 1$. Since 
	\[
	I -(I-e^{-\ell(I)^2\Delta_d})^M = \sum_{i=1}^Mc_i e^{-i\ell(I)^2\Delta_d}, 
	\]
	it suffices to prove that for each $i=1,\ldots, M$,  
	\[
	\big\|\mathcal{M}^*_\ell (\Delta_d^Me^{-i\ell(I)^2\Delta_d} b)\big\|_{\ell^p(\ZZ)}^p\les 1,
	\]
	where $a=\Delta_d^Mb$.
	
	At this stage, we can argue similarly to the estimate of $E$ to come up with
	\[
	\big\|\mathcal{M}^*_\ell (\Delta_d^Me^{-i\ell(I)^2\Delta_d} b)\big\|_{\ell^p(\ZZ)}^p\les 1
	\]
	for each $i=1,\ldots, M$.
	
	This completes our proof.
\end{proof}

The following lemma ensures that the class of distributions vanishing weakly at infinity makes a good sense in the definition of $H^p(\ZZ)$.
\begin{lemma}\label{lem-vanish at infty}
	\label{lem- vanish at vc} Let $p\in (0,1]$, $M>\f{1}{2}\big(\f{1}{p}-1\big)$ and $\epsilon>0$. If $f\in H^{p}_{\Delta_d,{\rm mol},M,\epsilon}(\ZZ)$, then $f$ vanishes weakly at infinity for each $k\in \mathbb N$.
\end{lemma}
\begin{proof}
	Suppose that $f\in H^{p}_{\Delta_d,{\rm mol},M,\epsilon}(\ZZ)$ and $k\in \mathbb N$. Then we have, by Proposition \ref{prop- M is bounded from Hp to Lp},
	\[
	\begin{aligned}
		|(t^2\Delta_d)^k e^{-t^2\Delta_d}f(n)|&\les \inf_{m\in I_n(t)} \mathcal M^*_k f(m)\\
		&\les |I_n(t)|^{-1/p} \|\mathcal M^*_k f\|_{\ell^p(I_n(t))}\\
		&\les (1+t)^{-1/p} \|f\|_{H^{p}_{\Delta_d,{\rm mol},M,\epsilon}(\ZZ)},
	\end{aligned}
	\]
	where we recall that $I_n(t)=[n-t,n+t]_{\mathbb R}\cap \ZZ$.
	
	It follows that $(t^2\Delta_d)^k e^{-t^2\Delta_d}f\to 0 $ in $\mathcal S'(\ZZ)$ as $t\to \vc$.
	
	The proof is hence completed.
\end{proof}

The main result of this section is the following.
\begin{theorem}\label{thm-Hardy space}
	Let $\epsilon>0$, $p\in (0,1]$ and $M>\f{1}{2}\big(\f{1}{p}-1\big)$. Then the Hardy spaces $H^{p}_{\Delta_d,{\rm mol},M,\epsilon}(\ZZ)$ and $H^{p}_{S_{\Delta_d}}(\ZZ)$ coincide and have equivalent norms.
\end{theorem}
Due to Theorem \ref{thm-Hardy space}, for $0<p\le 1$ we define the Hardy spaces $H^p_{\Dd}(\ZZ)$ to be either the space $H^{p}_{S_{\Delta_d}}(\ZZ)$ or the space $H^{p}_{\Delta_d,{\rm mol},M,\epsilon}(\ZZ)$ and $H^{p}_{S_{\Delta_d}}(\ZZ)$ with $\epsilon>0$, $p\in (0,1]$ and $M>\f{1}{2}\big(\f{1}{p}-1\big)$.

\bigskip

The proof of Theorem \ref{thm-Hardy space} is similar to those in \cite{HLMMY, HM}. However, as mentioned earlier since we define the Hardy spaces by using the space of distribution rather than the completion approach as in \cite{HLMMY, HM}. It is worth to provide the readers with the details. Before coming to the details of the proof we need to take care of the following material. We begin with some notations in \cite{CMS}.
\begin{itemize}
	\item For $x\in \ZZ$ and $\beta>0$, we denote $$\Gamma^\beta (n):=\big\{(m,t)\in \ZZ \times \big(0,\infty\big): |m-n| \leq \beta t\big\}.$$
	When $\beta=1$, we briefly write  $\Gamma(x)$ instead of $\Gamma^1(x)$. 
	
	\item For any  subset $F \subset \ZZ$ , define 
	$$ R^\beta (F) = \cup_{n\in F} \Gamma^\beta(n),$$ 
	and we also denote $R^1(F)$ by $R(F)$. 
	
	\item If $O$ is a set in $\ZZ$, then the tent over $\widehat{O}$ is defined as 
	$$\widehat{O}=\big(R(O^c)\big)^c.$$
	
\end{itemize}

For a measurable function $F$ defined in $\ZZ\times (0,\vc)$, we define
$$\mathcal A(F)(n)=\Big(\int_{0}^\vc \sum_{n:\,|m-n|<t}|F(m,t)|^2\frac{dt}{t(t+1)}\Big)^\frac{1}{2}.$$
For each $0<p\le 1$, the tent space $T^p_2$ is defined the set of all measurable functions $F$ such that 
\[
\|F\|_{T^p_2} := \|\mathcal A F\|_{\ell^p(\ZZ)}.
\]

 We now recall the definition of atoms in the tent space  in \cite{CMS}.
\begin{definition}[\cite{CMS}]		
	
	For $ 0<p\leq 1$, a measurable function $a$ on $\ZZ\times (0,\vc)$ is said to be a $T_2^p$-atom if there exists an interval $I\in \ZZ$ such that $a$ is supported in $\widehat{I}$ and 
	$$\int_{0}^\vc \sum_{n\in \ZZ}|a(n,t)|^2\frac{dt}{t} \leq |I|^{1-\frac{2}{p}}.$$
\end{definition}		
Then we have:
\begin{lemma}[\cite{CMS}]\label{atomic decoposition of T2}
	Let $0<p\leq 1$. For every $F\in T_2^p$ there exist a constant $C_p>0$, a sequence of numbers $\big\{\lambda_j\big\}_{j=0}^\infty$ and a sequence of $T_2^p$-atom $\big\{a_j\big\}_{j=0}^\infty$ such that 
	$$F= \sum_{j=0}^{\infty}\lambda_ja_j \text{ in $T_2^p$  a.e in $\ZZ\times (0,\vc)$}$$
	and $$\sum_{j=0}^{\infty}|\lambda_j|^p\leq C_p\|F\|^p_{T_2^p}.$$ 
\end{lemma}

\begin{lemma}\label{lem-Calderon formular hardy}
	If $f\in \mathcal S'(\ZZ)$ vanishes weakly at infinity, then for any $N\in \mathbb N^*$ we have
	\[
	f = \f{1}{(N-1)!} \int_0^\vc (t\Delta_d)^Ne^{-t\Delta_d} f\f{dt}{t} \ \ \ \text{in $\mathcal S'(\ZZ)$},
	\]
\end{lemma}
\begin{proof}
	By the integration by part and the fact that $f$ vanishes weakly at infinity, we have, for $f\in \mathcal S'(\ZZ)$,
	\[
	\begin{aligned}
	\int_0^\vc (t\Delta_d)^Ne^{-t\Delta_d} f\f{dt}{t} &= \f{(-1)^N}{(N-1)!}\int_0^\vc t^{N-1}\f{\partial^N}{\partial t^N}e^{-t\Delta_d}fdt\\
	&=\lim_{t\to 0}\Big(e^{-t\Delta_d}f +\sum_{m=1}^{N-1}c_m t^m\Delta_d^m e^{-t\Delta_d}f\Big)
\end{aligned}
	\]
	for some $c_m$.
	
	For each $g\in \mathcal S$, we can write
	\[
	t^m\Delta_d^m e^{-t\Delta_d}g= t^m e^{-t\Delta_d}(\Delta_d^m g)
	\] 
	and
	\[
	\big(I-e^{-t\Delta_d}\big)g = \int_0^t e^{-s\Delta_d}(\Delta_dg) ds.
	\]
	On the other hand, by using \eqref{Pml norm} and Lemma \ref{lem1-ht} it can be verified that if $g\in \mathcal S(\ZZ)$, then for each $m,k\in \mathbb N$,
	\[
	\|e^{-t\Delta_d}(\Delta_d^k g)\|_m\les \|g\|_{m+2k}.
	\]
	Therefore, for $g\in \mathcal S(\mathbb Z)$ and $t>0$, we have
	\[
	t^m\Delta_d^m e^{-t\Delta_d}g \to 0 \ \ \text{in $\mathcal S(\mathbb Z)$ as  $t\to 0$},
	\]
	and
	\[
	\big(I-e^{-t\Delta_d}\big)g \to 0 \ \ \text{in $\mathcal S(\mathbb Z)$ as  $t\to 0$}.
	\]
	Consequently,
	$$
	\lim_{t\to 0}\Big(e^{-t\Delta_d}g +\sum_{m=1}^{N-1}c_m t^m\Delta_d^m e^{-t\Delta_d}g\Big) = g \ \ \ \text{in $\mathcal S(\ZZ)$}.
	$$
	Hence, by duality
	\[
	\lim_{t\to 0}\Big(e^{-t\Delta_d}f +\sum_{m=1}^{N-1}c_m t^m\Delta_d^m e^{-t\Delta_d}f\Big)=f
	\]
	in $\mathcal S'(\ZZ)$.
	
	This completes our proof.
\end{proof}

\begin{lemma}\label{atom to molecule}
	Let $0< p \leq 1$, $M\in \mathbb{N}$ and $\epsilon>0$. Suppose that $A$ is a $T_2^p$-atom supported in $\widehat{I}$ with some interval  $I \subset \ZZ $ . Then for every $M \geq 1$ and $N \ge  1/p+\epsilon $, the function 
	$$\int_{0}^{\infty} \big(t^2  \Delta_d\big)^{M+N} e^{-t^2  \Delta_d} A(n,t)\frac{dt}{t}$$ define  a multiple of  $(p,M,\epsilon)_{\Delta_d}$-molecule associated to the interval $I$ with a multiple constant depending on $M$ only. 
\end{lemma}

\begin{proof}
	Let $A$ be a $T_2^p$-atom supported in $\widehat{I}$ with some interval $I$. Then we have,
	$$\int_{0}^{\infty}\sum_{n\in I}|A(n,t)|^2\frac{dt}{t}\leq |I|^{1-\frac{2}{p}}.$$
	We now define
	$$a (n) = \int_{0}^{\infty}\big(t^2  \Delta_d\big)^{M+N}e^{-t^2  \Delta_d}\big(A(\cdot,t)\big)(n)\frac{dt}{t}=\Delta_d^M\int_{0}^{\infty}t^{2M} \big(t^2  \Delta_d\big)^{N}e^{-t^2  \Delta_d}\big(A(\cdot,t)\big)(n)\frac{dt}{t}.
	$$
	Hence, $a = \Delta_d^Mb$, where 
	$$b(n) = \int_{0}^{\infty}t^{2M}\big(t^2  \Delta_d\big)^{N} e^{-t^2  \Delta_d}\big(A(\cdot,t)\big)(n)\frac{dt}{t}. 
	$$
	In addition, let $f \in \ell^2(S_j(I))$ such that $\supp f \subseteq S_j(I)$ and $\|f\|_{\ell^2(S_j(I))}=1$. Then, for each $k=0,1,\ldots, M$, 
	\begin{align*}
		\Big|\langle\Delta_d^k b,f \rangle_{\ell^2(\ZZ)}\Big|&=\Big|\sum_{n\in \ZZ} \Delta_d^k b(n) f(n) \Big|\\
		&\leq \int_{0}^\vc \sum_{n\in I} t^{2(M-k)}\Big|A(n,t)\big(t^2  \Delta_d\big)^{N+k} e^{-t^2  \Delta_d} f(n)\Big|\frac{dt}{t}\\
		&= \int_{0}^{\ell(I)} \sum_{n\in I} t^{2(M-k)}\Big|A(n,t)\big(t^2  \Delta_d\big)^{N+k} e^{-t^2  \Delta_d} f(n)\Big|\frac{dt}{t}\\
		&\leq \Big(\int_{0}^{\ell(I)} \sum_{n\in I} |A(n,t)|^2\frac{dt}{t}\Big)^\frac{1}{2} \Big(\int_{0}^{\ell(I)} \sum_{n\in I} t^{4(M-k)}\Big|\big(t^2  \Delta_d\big)^{N+k} e^{-t^2  \Delta_d }f(n)\Big|^2\frac{dt}{t}\Big)^\frac{1}{2}\\
		&\leq |I|^{\frac{1}{2}-\frac{1}{p}}\Big(\int_{0}^{\ell(I)} \sum_{n\in I} t^{4(M-k)}\Big|\big(t^2  \Delta_d\big)^{N+k} e^{-t^2  \Delta_d }f(n)\Big|^2\frac{dt}{t}\Big)^\frac{1}{2}.
	\end{align*}
	For $j=0,1,2$, using the $L^2$-boundedness of $\big(t^2  \Delta_d\big)^{N+k} e^{-t^2  \Delta_d }$, 
	\begin{align*}
		\Big|\langle\Delta_d^k b,f \rangle_{\ell^2(\ZZ)}\Big|
		&\leq |I|^{\frac{1}{2}-\frac{1}{p}}\Big(\int_{0}^{\ell(I)} \sum_{n\in I} t^{4(M-k)}\|f\|_{\ell^2(S_j(\ZZ))}\frac{dt}{t}\Big)^\frac{1}{2}\\
		&\leq (\ell(I))^{2(M-k)}|I|^{\frac{1}{2}-\frac{1}{p}}\\
		&\sim 2^{-j\epsilon }\ell(I)^{2(M-k)} |2^jI|^{\frac{1}{2}-\frac{1}{p}}.
	\end{align*}
By duality, for $j=0,1,2,$
\[
\|\Delta_d^k b\|_{\ell^2(S_j(\ZZ))}\les 2^{-j\epsilon }\ell(I)^{2(M-k)} |2^jI|^{\frac{1}{2}-\frac{1}{p}}.
\]
For $k=0,1,2,\ldots, M$, $j\ge 3$ and $f\in \ell^2(S_j(\ZZ))$, using Lemma \ref{lem1-htk},
\[
\begin{aligned}
	\sum_{n\in I} \Big|\big(t^2  \Delta_d\big)^{N+k} e^{-t^2  \Delta_d }f(n)\Big|^2&\les \sum_{n\in I} \Big|\sum_{m\in S_j(I)}\f{1}{t}\Big(1+\f{|m-n|}{t}\Big)^{-N}|f(m)|\Big|^2\\
	&\les |I| t^{2(N-1)} (2^{j}\ell(I))^{-2N}\|f\|_{\ell^1(S_j(I))}^2 \ \  \quad  (\text{since $|m-n|\sim 2^j\ell(I)$})\\
	&\les  |I| t^{2(N-1)} (2^{j}\ell(I))^{-2N}|2^jI|\|f\|_{\ell^2(S_j(I))}^2  \quad \quad (\text{H\"older's inequality})\\
	&\les 2^{-j(2N-1)}\Big(\f{t}{\ell(I)}\Big)^{2(N-1)}.
\end{aligned}
\]
Therefore,
	\begin{align*}
		\Big(\int_{0}^{\ell(I)} \sum_{n\in I} t^{4(M-k)}\Big|\big(t^2  \Delta_d\big)^{N+k} e^{-t^2  \Delta_d }f(n)\Big|^2\frac{dt}{t}\Big)^\frac{1}{2}		&\leq  2^{-j(N-1/2) }\ell(I)^{2(M-k)}\|f\|_{\ell^2(S_j(\ZZ))}\\
		&\les  2^{-j(N-1/2) }\ell(I)^{2(M-k)}\| f\|_{\ell^2(S_j(\ZZ))}\\
		&\leq 2^{-j(N-1/2) }\ell(I)^{2(M-k)}\| f\|_{\ell^2(S_j(\ZZ))}.
	\end{align*}
	Taking supremum over all $f$ such that $\|f\|_{\ell^2(S_j(I))}=1$ we obtain 
	$$ \|\Delta_d^k b\|_{\ell^2(S_j(I))}\leq c_M  2^{-j\epsilon }\ell(I)^{2(M-k)} |2^jI|^{\frac{1}{2}-\frac{1}{p}}$$
	for all $k=0,1,\ldots, M$ provided that $N \ge   1/p+\epsilon$.
	
	This completes the proof.
	
\end{proof}

\begin{proposition}\label{h-sh subset h-molecule}
	For   $0<p\le  1$, $\epsilon >0$ and $M \geq 1$. Then 
	$$H^p_{S_{\Delta_d }}(\ZZ)\hookrightarrow  H^p_{\Delta_d ,{\rm mol},M,\epsilon}(\ZZ).$$ 
\end{proposition}

\begin{proof} 
	Suppose $f\in  H^p_{S_{\Delta_d }}(\ZZ).$ We set $F(n,t)=t^2 \Delta_d  e^{-t^2 \Delta_d  } f(n)$. From the definition of $H_{S_{\Delta_d }}^p$, $F(n,t) \in T_2^p$. Thus, by using Lemma $\ref{atomic decoposition of T2} $, $F$ can be represented in the form $F=\sum_{i=1}^{\infty}\lambda_i A_i$, where $A_i$ is $T_2^p$-atom. 
	In addition, $$\sum_{i=0}^{\infty}|\lambda_i|^p \leq C\|F\|^p_{T_2^p}=C\|f\|^p_{H^p_{S_{\Delta_d }}}.$$
	
	Since $f$ vanishes weakly at infinity, by Lemma \ref{lem-Calderon formular hardy}, 
	\begin{equation}\label{f(n)=int}
		f = C \int_{0}^{\infty}\big(t^2 \Delta_d \big)^{M+N} e^{-t^2 \Delta_d  } t^2 \Delta_d  e^{-t^2 \Delta_d  }f \frac{dt}{t} \ \ \text{in $\mathcal S'(\ZZ)$},
	\end{equation}
	where $N \ge  1/p+1+\epsilon$.
	

It follows that 
\begin{equation}\label{eq-series}
	\begin{aligned}
		f(n)  &=C \sum_{i=0}^{\infty}\lambda_i\int_{0}^{\infty}\big(t^2 \Delta_d \big)^{M+N} e^{-t^2 \Delta_d  } A_i(n,t) \frac{dt}{t}\\
		&=:C \sum_{i=0}^{\infty}\lambda_i a_i.
	\end{aligned}
\end{equation}
We will show that the identity \eqref{eq-series} holds in $\mathcal S'(\ZZ)$. It suffices to show the series in \eqref{eq-series} converges in $\mathcal S'(\ZZ)$. To do this, we need to prove that for $\phi\in \mathcal S(\ZZ)$,
\[
\lim_{L\to \vc} \sum_{i=L}^\vc \lambda_i\langle a_i,\phi\rangle =0.
\]
Assume $A_i$ is a $T_2^p$-atom supported in $\widehat{I}_i$ with some interval $I_i$ for each $i$. Then we have
\[
\begin{aligned}
	\langle a_i,\phi\rangle &= \int_0^\vc \sum_{n\in \ZZ}\big(t^2 \Delta_d \big)^{M+N} e^{-t^2 \Delta_d  } A_i(n,t) \phi(n) \frac{dt}{t}\\
	&= -\int_0^{\ell(I_i)} \sqrt {t}\sum_{n\in I}A_i(n,t) \sqrt t D\big(t^2 \Delta_d \big)^{M+N-1} e^{-t^2 \Delta_d  }  (D\phi)(n) \frac{dt}{t},
\end{aligned}
\]
where in the last inequality we used the fact that $\Delta_d^{M+N}e^{-t^2 \Delta_d  }  \phi=-D\Delta_d^{M+N-1}e^{-t^2 \Delta_d  }  (D\phi)$.

By H\"older's inequality,
\begin{equation}\label{eq- aphi}
\begin{aligned}
	|\langle a_i,\phi\rangle| &\le \int_0^{\ell(I_i)} \sqrt {t}\|A_i(\cdot,t)\|_{\ell^2(I)} \|\sqrt t D\big(t^2 \Delta_d \big)^{M+N-1} e^{-t^2 \Delta_d  }  (D\phi)\|_{\ell^2(\ZZ)} \frac{dt}{t}.
\end{aligned}
\end{equation}
From the kernel bound of $\sqrt t D\big(t^2 \Delta_d \big)^{M+N-1} e^{-t^2 \Delta_d  }$ in Lemma \ref{lem-htk and difference derivatives} and Lemma \ref{lem-elementary}, we yield that the operator $\sqrt t D\big(t^2 \Delta_d \big)^{M+N-1} e^{-t^2 \Delta_d  }$ is bounded on $\ell^p(\ZZ)$ for all $1<p\le \vc$. Consequently,
\[
\|\sqrt t D\big(t^2 \Delta_d \big)^{M+N-1} e^{-t^2 \Delta_d  }  (D\phi)\|_{\ell^2(\ZZ)}\les  \|D\phi\|_{\ell^2(\ZZ)}.
\]
On the other hand, since $f\in \mathcal S(\ZZ)$, $D\phi\|_{\ell^2(\ZZ)}<\vc$. It follows that 
\[
\|\sqrt t D\big(t^2 \Delta_d \big)^{M+N-1} e^{-t^2 \Delta_d  }  (D\phi)\|_{\ell^2(\ZZ)}<\vc.
\]
Plugging this into \eqref{eq- aphi} and then using H\"older's inequality we further obtain
\[
\begin{aligned}
	|\langle a_i,\phi\rangle| &\les \int_0^{\ell(I)} \sqrt {t}\|A_i(\cdot,t)\|_{\ell^2(I_i)} \frac{dt}{t}\\
	&\les \Big(\int_0^{\ell(I_i)} \|A_i(\cdot,t)\|^2_{\ell^2(I)} \frac{dt}{t}\Big)^{1/2}\Big(\int_0^{\ell(I_i)} t \frac{dt}{t}\Big)^{1/2}\\
	&\les |I_i|^{1/2-1/p}|\ell(I)|^{1/2}\\
	&\les |I_i|^{1-1/p}\\
	&\les 1,
\end{aligned}
\]
where in the last inequality we used  the facts $0<p\le 1$ and $|I_i|\ge 1$ for each $i$.

As a consequence,
\[
\sum_{i=L}^\vc |\lambda_i\langle a_i,\phi\rangle|\les \sum_{i=L}^\vc |\lambda_i| \to 0 \ \ \text{as $L\to \vc$},
\]
since $\sum_{i=0}^{\infty}|\lambda_i|^p \les \|f\|^p_{H^p_{S_{\Delta_d }}}$.

Therefore, the identity \eqref{eq-series} holds in $\mathcal S'(\ZZ)$. In addition, in Lemma \ref{atom to molecule} we have proved that each $a_i$ is multiple of a  $(p,M,\epsilon)_{\Delta_d }$-molecule with a harmless constant. Therefore, $f \in {H}^p_{\Delta_d , {\rm mol},M,\epsilon}(\ZZ)$ and 
$$\|f\|_{H^p_{\Delta_d , {\rm mol},M,\epsilon}(\ZZ)} \lesssim \Big(\sum_{i=0}^{\infty}\lambda_i^p\Big)^\frac{1}{p} \lesssim \|f\|_{H_{S_{\Delta_d }}^p}.$$ 

This completes our proof.
\end{proof}

\begin{proposition}\label{prop-H atom subset H square}
	For   $0<p\le  1$, $\epsilon >0$ and $M \geq 1$. Then 
	$$H^p_{\Delta_d ,{\rm mol},M,\epsilon}(\ZZ)\hookrightarrow  H^p_{S_{\Delta_d }}(\ZZ).$$ 
\end{proposition}

\begin{proof} 
	Since the kernel $\partial_t h_{t}(\cdot) \in \mathcal S(\ZZ)$, it suffices to show that there exists $C>0$ such that 
\[
\|S_{\Delta_d} a\|_{\ell^p(\ZZ)}\le C
\]
for every $(p,M,\epsilon)_{\Delta_d}$ molecule $a$.

Suppose that $a$ is a $(p,M,\epsilon)_{\Delta_d}$ associated to an interval $I\subset \ZZ$. Then we have
\[
\begin{aligned}
	\|S_{\Dd} a\|_{\ell^p(\ZZ)}^p &\les \big\|S_{\Dd}(I-e^{-\ell(I)^2\Delta_d})^M a\big\|_{\ell^p(\ZZ)}^p + \big\|S_{\Dd}\big[I-(I-e^{-\ell(I)^2\Delta_d})^M\big] a\big\|_{\ell^p(\ZZ)}^p\\
	&=: E+ F.
\end{aligned}
\]

We  take care of the term $E$ first. To do this, we write
\[
\begin{aligned}
	E &\les \sum_{j\ge 0} \big\|S_{\Dd}(I-e^{-\ell(I)^2\Delta_d})^M (a\cdot 1_{S_j(I)})\big\|_{\ell^p(\ZZ)}^p\\
	&=: \sum_{j\ge 0} E_j.
\end{aligned}
\]
For each $j\ge 0$, by H\"older's inequality, 
\[
\begin{aligned}
	E_j&\le \sum_{k\ge 0}\|S_{\Dd}(I-e^{-\ell(I)^2\Delta_d})^M (a\cdot 1_{S_j(I)})\|_{\ell^p(S_k(2^jI))}^p\\
	&\le \sum_{k\ge 0}|2^{k+j}I|^{\f{2-p}{2}}\|S_{\Dd}(I-e^{-\ell(I)^2\Delta_d})^M (a\cdot 1_{S_j(I)})\|_{\ell^2(S_k(2^jI))}^p\\
	&=: \sum_{k\ge 0}E_{jk}.
\end{aligned}
\]
For $k=0,1,2,3$, by the $L^2$-boundedness of $S_{\Dd}\circ (I-e^{-\ell(I)^2\Delta_d})^M$ we have
\[
\begin{aligned}
	E_{jk}&\le |2^jI|^{\f{2-p}{2}}\|S_{\Dd}(I-e^{-\ell(I)^2\Delta_d})^M (a\cdot 1_{S_j(I)})\|_{\ell^2(\ZZ)}^p\\
	&\le |2^jI|^{\f{2-p}{2}}\|a\|_{\ell^2(S_j(I))}^p\\
	&\le 2^{-j\epsilon}|2^jI|^{\f{2-p}{2}}|2^jI|^{\f{p-2}{2}} \sim  2^{-(j+k)\epsilon}.
\end{aligned}
\]
For $k\ge 4$, note that 
\[
(I-e^{-\ell(I)^2\Delta_d})^M =\int_{[0,\ell(I)^2]^M}\Delta^M_de^{-|\vec s|\Delta_d}d\vec s,
\]
where $d\vec{s} = ds_1\ldots ds_M$ and $|\vec s| = s_1+\ldots+s_M$.

Then we have, for $n\in S_k(2^j(I))$,
\[
\begin{aligned}
	\big[S_{\Dd}&(I-e^{-\ell(I)^2\Delta_d})^{M} (a\cdot 1_{S_j(I)})(n)\big]^2\\
	&\le \int_0^\vc \sum_{m:|m-n|<t} \Big[\int_{[0,\ell(I)^2]^M}t^{2}|\Delta_d^{M+1}e^{-(t^2+|\vec s|)\Delta_d}(a\cdot 1_{S_j(I)})(m)|d\vec s\Big]^2\f{dt}{t(t+1)}\\
	&\le \int_{2^{k+j}\ell(I)/4}^\vc\ldots + \int_0^{2^{k+j}\ell(I)/4}\ldots\\
	&=: F_1 + F_2. 
\end{aligned}
\]
Since for $n\in S_k(2^j(I))$ and $m\in I$, $t>|m-n|\ge 1$.  By Lemma \ref{lem1-htk},
\[
\begin{aligned}
	F_1&\les \int_{2^{k+j}\ell(I)/4}^\vc\sum_{m:|m-n|<t} \Big[ \int_{[0,\ell(I)^2]^M}\sum_{l\in S_j(I)}\f{t^{2}}{(t^2+|\vec s|)^{M+3/2}} |a(l)|d\vec s\Big]^2\f{dt}{t(t+1)}\\
	&\les 2^{-(k+j)(4M+2)}  |I|^{-2}\|a\|^2_{\ell^1(S_j(B))},
\end{aligned}
\]
and
\[
\begin{aligned}
	F_2&\les \int^{2^{k+j}\ell(I)/4}_0\sum_{m:|m-n|<t} \Big[ \int_{[0,\ell(I)^2]^M}\sum_{l\in S_j(I)}\f{t^{2}}{(t^2+|\vec s|)^{M+3/2}}\Big(\f{\sqrt{t^2+|\vec s|}}{|m-l|}\Big)^{2M+3} |a(l)|d\vec s\Big]^2\f{dt}{t(t+1)}\\
	&\les 2^{-(k+j)(4M+2)}  |I|^{-2}\|a\|^2_{\ell^1(S_j(B))} \ \ \ \ \ \text{(since $|m-l|\sim 2^{k+j}\ell(I)$)}.
\end{aligned}
\]
It follows that 
\[
\begin{aligned}
	E_{jk}&\les 2^{-(k+j)(2M+1)p}|2^{k+j}I|^{\f{2-p}{2}}|I|^{-p}\|a\|_{\ell^1(S_j(B))}^p|2^{j+k}I|^{p/2}\\
	&\les 2^{-(j+k)[(2M+1)p-1]}.
\end{aligned}
\]
As a consequence,
\[
E \les \sum_{j,k\ge 0}2^{-(j+k)\epsilon}\sim 1,
\]	
as long as $M>\f{1}{2}\big(\f{1}{p}-1\big)$.

It remains to show that $F\les 1$. Since 
\[
I -(I-e^{-\ell(I)^2\Delta_d})^M = \sum_{i=1}^Mc_i e^{-i\ell(I)^2\Delta_d}, 
\]
it suffices to prove that for each $i=1,\ldots, M$,  
\[
\big\|S_{\Dd} (\Delta_d^Me^{-i\ell(I)^2\Delta_d} b)\big\|_{\ell^p(\ZZ)}^p\les 1,
\]
where $a=\Delta_d^Mb$.

At this stage, we can argue similarly to the estimate of $E$ to come up with
\[
\big\|S_{\Dd} (\Delta_d^Me^{-i\ell(I)^2\Delta_d} b)\big\|_{\ell^p(\ZZ)}^p\les 1
\]
for each $i=1,\ldots, M$.

This completes our proof.	
\end{proof}	

\begin{remark}\label{rem2}
	\noindent (a) For each $N \in \mathbb N$, we define
	\begin{equation}
		\label{eq-area integral}
		S_{\Delta_d, N}f(n)=\Big(\int_0^\vc \sum_{m:|m-n|<t}\big|(t^2\Delta_d)^Ne^{-t^2\Delta_d}f(m)\big|^2\f{dt}{t(t+1)}\Big)^{1/2}.
	\end{equation}
Then we can define the Hardy spaces $H^p_{S_{\Dd, N}}(\ZZ)$ similarly to $H^p_{S_{\Dd}}(\ZZ)$. By a careful examination, we also have $H^p_{S_{\Dd, N}}(\ZZ) = H^p_{\Dd}(\ZZ)$ for all $0<p\le 1$.

\noindent (b) Let $\psi\in \mathcal S(\mathbb R)$ supported in $[2,8]$. We now define
$$
	S_{\psi}f(n)=\Big(\int_0^\vc \sum_{m:|m-n|<t}\big|\psi(t\sqrt{\Dd})f(m)\big|^2\f{dt}{t(t+1)}\Big)^{1/2}.
$$
For $0<p\le 1$, we define the Hardy space $H^p_{\psi}(\ZZ)$ to be the completion of the set
\[
\big\{f\in \ell^2(\ZZ): S_{\psi}f \in \ell^p(\ZZ)\big\}
\]
under the norm
\[
\|f\|_{H^p_{\psi}(\ZZ)} := \big\|S_{\psi}f\big\|_{\ell^p(\ZZ)}<\vc.
\]
Then by the argument in the proof of Theorem \ref{thm-Hardy space}, it is easy to see that $H^p_{\psi}(\ZZ)=H^p_{\Dd}(\ZZ)$.

\end{remark}
\subsection{Coincidence with the classical Hardy spaces}


Let $0<p<q$ and $p\le 1 \le q\le \vc$. A function $a$ is called a $(p,q)$ atom if there exists an interval $I\subset \ZZ$ such that 
\begin{enumerate}
\item $\supp a\subset I$;
\item $\|a\|_{\ell^q(\ZZ)}\le |I|^{1/q-1/p}$;
\item $\sum_{n\in \ZZ} n^\alpha a(n)=0$ for every $n\in \mathbb N$ such that $n\le 1/p-1$.
\end{enumerate}

The Hardy space $H^{p,q}(\ZZ)$ is the set of all functions $f\in \ell^1(\ZZ)$ if $p=1$ and $f\in \mathcal S'(\ZZ)$ if $0<p<1$ such that there exist a sequence of $(p,q)$ atoms $\{a_j\}_{j\in \mathbb N}$ and a sequence of numbers $\{\lambda_j\}_{j\in \mathbb N}$ satisfying that $\sum_{j\in \mathbb N}|\lambda_j|^p<\vc$ and $f = \sum_{j\in \mathbb N}\lambda_j a_j$, where the series converges in $\ell^1(\ZZ)$ if $p=1$ and in $\mathcal S'(\ZZ)$ if $0<p<1$. For each $f\in H^{p,q}(\ZZ)$, we define
\[
\|f\|_{H^{p,q}(\ZZ)}^p = \inf\Big\{\sum_{j\in \mathbb N}|\lambda_j|^p:  f = \sum_{j\in \mathbb N}\lambda_j a_j\Big\},
\] 
where the infimum is taken over all atomic decomposition of $f$.

It is well-known that for $0<p<q$ and $p\le 1 \le q\le \vc$, we have $H^{p,q}(\ZZ)=H^{p,\vc}(\ZZ)$. For this reason, for $0<p<q$ and $p\le 1 \le q\le \vc$ we define $H^p(\ZZ)$ to be any space $H^{p,q}(\ZZ)$.

Let $0<p\le 1 <q\le \vc$ and $\alpha>1/p-1/q$. A function $b$ is said to be a $(p,q,\alpha)$-molecule centered in $n_0\in \ZZ$ if the following holds
\begin{enumerate}[(i)]
\item  $\|b\|^{1-\theta}_{\ell^q(\ZZ)}\||\cdot-n_0|^\alpha b\|_{\ell^q(\ZZ)}^\theta<\vc$, where $\theta = (1/p-1/q)/\alpha$;
\item $\sum_{n\in \ZZ} n^\beta b(n)=0$ for every $\beta\in \mathbb N$ such that $\beta\le 1/p-1$.
\end{enumerate}

\begin{theorem}
\label{thm-coincidence of Hardy spaces}
For $0<p\le 1$, we have 
\[
H^p_{\Delta_d}(\ZZ)\equiv  H^p(\ZZ)
\]
with equivalent norms.
\end{theorem}
\begin{proof}
If $a = \Delta_d^M b$ is a $(p,M,\epsilon)_{\Delta_d}$ molecule associated to an interval $I\subset \ZZ$ with $M>(1/p-1)/2$ and $\epsilon> 0$, then it suffices to prove that $a$ is a $(p,2,\alpha)$ molecule with $\epsilon> \alpha>0$ centered at any $n_0\in I$. Indeed, we first have
\[
\|a\|_{\ell^2(\ZZ)} \le \sum_{j\ge 0}\|a\|_{\ell^2(S_j(I))}\le \sum_{j\ge 0} 2^{-j\epsilon}|2^jI|^{1/2-1/p}.
\]
Secondly,  for $\theta=(1/p-1/2)/\alpha$,
\[
\begin{aligned}
	\||\cdot-n_0|^\alpha a\|_{\ell^2(\ZZ)}&\le \sum_{j\ge 0}\||\cdot-n_0|^\alpha a\|_{\ell^2(S_j(I))}\\
	&\les \sum_{j\ge 0} 2^{j\alpha} \|a\|_{\ell^2(S_j(I))}\\
	&\les  \sum_{j\ge 0} 2^{j(\alpha-\epsilon)}|2^jI|^{1/2-1/p}.
\end{aligned}
\]

It follows that for $\theta=(1/p-1/2)/\alpha$, 
\[
\begin{aligned}
	\|a\|_{\ell^2(\ZZ)}^{1-\theta}\||\cdot-n_0|^\alpha a\|_{\ell^2(\ZZ)}^\theta& \les \sum_{j\ge 0} 2^{-j(\epsilon-\theta\alpha)}|2^jI|^{1/2-1/p}\\
	&\les \sum_{j\ge 0} 2^{-j\epsilon}|I|^{1/2-1/p}\\
	&<\vc.
\end{aligned}	
\]

On the other hand, for any $\beta\in \mathbb N$ such that $\beta\le 1/p-1$, we have
\[
\sum_{n\in \ZZ} n^\beta a(n) = \sum_{n\in \ZZ} \Delta_d^M n^\beta b(n).
\] 
Since $M> (1/p-1)/2 \ge \beta/2$, $\Delta_d^M n^\beta = 0$. Hence,
\[
\sum_{n\in \ZZ} n^\beta a(n)=0.
\]
As a consequence, each $(p,M,\epsilon)_{\Delta_d}$ molecule is also a $(p,q,\alpha)$-molecule, which implies that 
\[
H^p_{\Delta_d}(\ZZ)\hookrightarrow H^p(\ZZ).
\]

It remains to prove that $H^p(\ZZ)\hookrightarrow H^p_{\Delta_d}(\ZZ)$. To do this, by Remark \ref{rem2} it suffices to prove that there exists $C>0$ such that 
\[
\|S_{\Delta_d,M}a\|_{\ell^p(\ZZ)}\le C
\]
for all $(p,2)$ atom, where $M\in \mathbb N$ such that $M>\lfloor 1/p-1 \rfloor+1=:K$.

Suppose that $a$ is a $(p,2)$ atom associated to an interval $I$. Then we have
\[
\begin{aligned}
	\|S_{\Dd,M} a\|_{\ell^p(\ZZ)}^p &  \les  \big\|S_{\Dd, M}a\big\|_{\ell^p(4I)}^p +\big\|S_{\Dd, M}a\big\|_{\ell^p(\ZZ\backslash 4I)}^p.
\end{aligned}
\]

For the first term, by the $L^2$-boundedness of $S_{\Dd,M}$ we have
\[
\begin{aligned}
	\big\|S_{\Dd, M}a\big\|_{\ell^p(4I)}^p &\le |4I|^{\f{2-p}{2}}\|S_{\Dd,M} a\|_{\ell^2(4I)}^p\\
	&\les |I|^{\f{2-p}{2}}\|a\|_{\ell^2(I)}^p\\
	&\les 1.
\end{aligned}
\]
For $n\in \ZZ\backslash 4I$, we have
\[
\begin{aligned}
	\big[S_{\Dd,M}a(n)\big]^2
	&=\int_0^\vc \sum_{m:|m-n|<t} \Big[ |(t^2\Delta_d)^{M}e^{-t^2 \Delta_d}a(m)|d\vec s\Big]^2\f{dt}{t(t+1)}\\
	&= \int_{0}^{\ell(I)}\ldots + \int_{\ell(I)}^\vc \ldots\\
	&=: E_1 + E_2. 
\end{aligned}
\]
Since $t>|m-n|\ge 1$ for $n\in \ZZ\backslash 4I$ and $m\in I$, by Lemma \ref{lem1-htk}, 
\[
\begin{aligned}
	E_1&\les \int^{\ell(I)}_0\sum_{m:|m-n|<t} \Big[ \sum_{l\in I}\f{1}{t}\Big(\f{t}{|m-l|}\Big)^{2M} |a(l)|\Big]^2\f{dt}{t(t+1)}\\
	&\les \int^{\ell(I)}_0\sum_{m:|m-n|<t} \Big[ \sum_{l\in I}\f{1}{t}\Big(\f{t}{d(n,I)}\Big)^{2M} |a(l)|\Big]^2\f{dt}{t(t+1)}\\
	&\les \int^{\ell(I)}_0 \Big[ \sum_{l\in I}\f{1}{t}\Big(\f{t}{d(n,I)}\Big)^{2M} |a(l)|\Big]^2\f{dt}{t}\\
	&\les \Big[ \|a\|_{\ell^1(\ZZ)}  \f{\ell(I)^{2M-1}}{d(n,I)^{2M}}\Big]^2\\
	&\les \Big[ \|a\|_{\ell^1(\ZZ)}  \f{\ell(I)^{K}}{d(n,I)^{K+1}}\Big]^2.
\end{aligned}
\]
For the second term $E_2$, we have

	Using \eqref{eq-integration by parts fomular for the time derivative of heat kernel}, we have for each $k\in \mathbb N$,
\[
\begin{aligned}
	(-\Delta_d)^{M}e^{-t\Delta_d} a(n)&=\sum_{m\in I}\f{(-1)^{\lfloor (k+1)/2\rfloor}}{\pi {(n-m)}^k}\int_0^\pi \partial_\theta^k\partial_{t}^M  \varphi_{t}(\theta)\Phi_k((n-m)\theta)a(m) d\theta\\
	&:=\sum_{m\in I} g_{M,k,t}(n-m)a(m).
\end{aligned}
\]
Due to the vanishing moment condition of the atom $a$, for a fixed $n_I\in I$ and $K:=\lfloor 1/p-1 \rfloor +1$,
\[
(-\Delta_d)^{M}e^{-t\Delta_d} a(n) =\sum_{m\in I}  \Big[g_{M,K+1,t}(n-m)-\sum_{j=0}^K \f{(n_I-m)^j}{j!}g^{(j)}_{\ell,K+1,t}(n-n_I)\Big]a(m).
\]
Applying Taylor theorem and Lemma \ref{lemma- derivative estimate - Taylor}, for $n\in \ZZ\backslash 4I,$
\begin{equation}\label{eq1- Taylor 1}
	\begin{aligned}
		|(t\Delta_d)^{M}e^{-t\Delta_d}a(n)|&\les  \sum_{m\in I} \Big(\f{\ell(I)}{|n-m|\wedge \sqrt t}\Big)^{K+1}\f{1}{\sqrt t}\Big(\f{|n-m|}{\sqrt t}\Big)^{-K-1}|a(m)|\\
		&\les  \Big(\f{\ell(I)}{d(n,I)\wedge \sqrt t}\Big)^{K+1}\f{1}{\sqrt t}\Big(\f{d(n,I)}{\sqrt t}\Big)^{-K-1}\|a\|_{\ell^1(\ZZ)}.
	\end{aligned}
\end{equation}
On the other hand, we also have
\[
(-\Delta_d)^{M}e^{-t\Delta_d} a(n) =\sum_{m\in I}  \Big[g_{M,K+1,t}(n-m)-\sum_{j=0}^K \f{(n_I-m)^j}{j!}g^{(j)}_{\ell,0,t}(n-n_I)\Big]a(m).
\]
Applying Taylor theorem and Lemma \ref{lemma- derivative estimate - Taylor} again, for $n\in \ZZ\backslash 4I,$
\begin{equation}\label{eq2- Taylor 1}
	\begin{aligned}
		|(t\Delta_d)^{M}e^{-t\Delta_d}a(n)|&\les  \sum_{m\in I} \f{1}{\sqrt t}\Big(\f{\ell(I)}{\sqrt t}\Big)^{K+1} |a(m)|\\
		&\les \f{1}{\sqrt t}\Big(\f{\ell(I)}{d(n,I)\wedge \sqrt t}\Big)^{K+1} \|a\|_{\ell^1(\ZZ)}.
	\end{aligned}
\end{equation}

From \eqref{eq1- Taylor 1} and \eqref{eq2- Taylor 1}, we have, for $n\in \ZZ\backslash 4I,$
\begin{equation}\label{eq3- Taylor 1}
	\begin{aligned}
		|(t\Delta_d)^{M}e^{-t\Delta_d}a(n)| &\les  \Big(\f{\ell(I)}{d(n,I)\wedge \sqrt t}\Big)^{K+1}\f{1}{\sqrt t}\Big(1+\f{d(n,I)}{\sqrt t}\Big)^{-K-1}\|a\|_{\ell^1(\ZZ)}\\
		&\les \f{1}{\sqrt t} \Big(\f{\ell(I)}{d(n,I)}\Big)^{K+1}\|a\|_{\ell^1(\ZZ)}.
	\end{aligned}
\end{equation}

It follows that 
\[
E_2\les \Big[ \|a\|_{\ell^1(\ZZ)}  \f{\ell(I)^{K}}{d(n,I)^{K+1}}\Big]^2.
\]
Therefore, for $n\in \ZZ\backslash 4I$,
\[
	 S_{\Dd,M}a(n)\les \|a\|_{\ell^1(\ZZ)}  \f{\ell(I)^{K}}{d(n,I)^{K+1}}\les |I|^{1-1/p}\f{\ell(I)^{K}}{d(n,I)^{K+1}}.
\]
As a consequence,
\[
\big\|S_{\Dd,M}a\big\|_{\ell^p(\ZZ\backslash 4I)}\le C.
\]

This completes our proof.
\end{proof}

\section{Besov and Triebel--Lizorkin spaces associated to $\Dd$}
The main aim of this section is to develop the theory of Besov and Triebel--Lizorkin spaces associated to $\Dd$. This is motivated by the work of the authors with H-Q. Bui \cite{BBD}. It is interesting to note that our function spaces are identical with the classical spaces defined by \cite{Sun, Torres}. Therefore, the obtained results in this section can be viewed as new characterizations for the classical Besov and Triebel--Lizorkin spaces.

\bigskip

As we mentioned in Section \ref{sec: distributions}, if $\psi\in \mathscr{\mathbb R}$ with $\supp \psi\subset [2,8]$, then 
$$\psi(t\sqrt{\Delta_d})(\cdot, m), \psi(t\sqrt{\Delta_d})(m,\cdot)\in \mathcal S_\vc(\ZZ)$$ for every $t>0$ and $m\in \ZZ$. Therefore, for any $f\in \mathcal S'_\vc(\ZZ)$ we can define
\[
\psi(t\sqrt{\Delta_d})f(n) = \langle \psi(t\sqrt{\Delta_d})(n,\cdot), f\rangle.
\]
\begin{definition}
	Let $\psi$ be  a partition of unity. For $0< p, q\leq \vc$ and $\alpha\in \mathbb{R}$, we define the  homogeneous Besov space $\B^{\alpha,\psi,\Delta_d}_{p,q}(\ZZ)$ as follows 
	$$
	\B^{\alpha,\psi,\Delta_d}_{p,q}(\ZZ) = 
	\Big\{f\in \mathcal{S}'(\ZZ)/\mathcal P(\ZZ):  \|f\|_{\B^{\alpha,\psi,\Delta_d}_{p,q}(\ZZ)}<\vc\Big\},
	$$
	where
	\[
	\|f\|_{\B^{\alpha,\psi,\Delta_d}_{p,q}(\ZZ)}= \Big\{\sum_{j\in \mathbb{Z}^-}\left(2^{j\alpha}\|\psi_j(\sqrt{\Delta_d})f\|_{\ell^p(\ZZ)}\right)^q\Big\}^{1/q}.
	\]
	
	Similarly, for $0< p<\vc$, $0<q\le \vc$ and $\alpha\in \mathbb{R}$, the homogeneous Triebel--Lizorkin space $\F^{\alpha,\psi}_{p,q}(\ZZ)$ is defined by 
	$$
	\F^{\alpha, \psi,\Delta_d}_{p,q}(\ZZ) = 
	\Big\{f\in \mathcal{S}'(\ZZ)/\mathcal P(\ZZ):  \|f\|_{\F^{\alpha,\psi}_{p,q}(\ZZ)}<\vc\},
	$$
	where
	\[
	\|f\|_{\F^{\alpha, \psi,\Delta_d}_{p,q}(\ZZ)}= \Big\|\Big[\sum_{j\in \mathbb{Z}^-}(2^{j\alpha}|\psi_j(\sqrt{\Delta_d})f|)^q\Big]^{1/q}\Big\|_{\ell^p(\ZZ)}.
	\]
\end{definition}

We would like to highlight that in the above definitions of the Besov and Triebel--Lizorkin spaces, the sums are taken over $j\in \mathbb Z^-$ instead of $j\in \mathbb Z$. This is because since the spectral of $\Delta_d$ is contained in $[0,4]$, $\psi_j(\sqrt{\Dd}) =0$ whenever $j\in \mathbb Z^+$.

It is obvious that if $f\in \mathcal{P}(\ZZ)$ then $\psi_j(\sqrt{\Delta_d})f = 0$ for all $j$.  Conversely, we assume that $\psi_j(\sqrt{\Delta_d})f = 0$ for all $j\in \mathbb{Z}$. By Proposition \ref{prop-Calderon1}, $f=0$ in $\mathcal S'/\mathcal P(\ZZ)$, or equivalently, $f$ is a polynomial. As a consequence, if $\|f\|_{\B^{\alpha, \psi,\Delta_d}_{p,q}(\ZZ)}=0$ or  $\|f\|_{\F^{\alpha, \psi,\Delta_d}_{p,q}(\ZZ)}=0$, then $f\in \mathcal P(\ZZ)$.

Therefore, each of the above spaces is a quasi-normed linear space (normed linear space when $p,q\ge 1$).

\bigskip

For $\lambda>0, j\in \mathbb{Z}$ and $\varphi\in \mathcal{S}(\mathbb{R})$ supported in $[2,8]$ the Peetre's type maximal function is defined, for $f\in \mathcal S'(\ZZ)/\mathcal P(\ZZ)$, by
\begin{equation}
	\label{eq-PetreeFunction}
	\varphi_{j,\lambda}^*(\sqrt{\Delta_d})f(n)=\sup_{m\in \ZZ}\f{|\varphi_j(\sqrt{\Delta_d})f(m)|}{(1+2^j|m-n|)^\lambda}\, , n\in \ZZ,
\end{equation}
where $\varphi_j(\lambda)=\varphi(2^{-j}\lambda)$.

Obviously, we have
\[
\varphi_{j,\lambda}^*(\sqrt{\Delta_d})f(n)\geq |\varphi_j(\sqrt{\Delta_d})f(n)|, \ \ \ \ n\in \ZZ.
\]
Similarly, for $s, \lambda>0$ we set
\begin{equation}
	\label{eq2-PetreeFunction}
	\varphi_{\lambda}^*(s\sqrt{\Delta_d})f(n)=\sup_{m\in \ZZ}\f{|\varphi(s\sqrt{\Delta_d})f(m)|}{(1+|m-n|/s)^\lambda}, \ \ \ f\in \mathcal S'(\ZZ)/\mathcal P(\ZZ).
\end{equation}

We next prove the following result which can be viewed as a consequence of  Proposition 3.2 and Proposition 3.3 in \cite{BBD}.
\begin{proposition}
	\label{prop2-thm1}
	Let $\varphi$ and $\psi$ be a partition of unity. Then we have:
	\begin{enumerate}[(a)]
		\item For $0< p, q\le \vc$, $\alpha\in \mathbb{R}$ and $\lambda>1/p$,
		$$\displaystyle \Big\{\sum_{j\in \mathbb{Z}}\left(2^{j\alpha}\|\varphi^*_{j,\lambda}(\sqrt{\Delta_d})f\|_{\ell^p(\ZZ)}\right)^q\Big\}^{1/q}\sim \|f\|_{\B^{\alpha, \psi,\Delta_d}_{p,q}(\ZZ)}.
		$$
		
		\item For $0< p<\vc$, $0<q\le \vc$, $\alpha\in \mathbb{R}$ and $\lambda>\max\{1/q, 1/p\}$,
		$$\displaystyle \Big\|\Big[\sum_{j\in \mathbb{Z}}(2^{j\alpha}|\varphi^*_{j,\lambda}(\sqrt{\Delta_d})f|)^q\Big]^{1/q}\Big\|_{\ell^p(\ZZ)}\sim \|f\|_{\F^{\alpha, \psi,\Delta_d}_{p,q}(\ZZ)}.$$
	\end{enumerate}
\end{proposition}

\bigskip
\begin{remark}
	Proposition \ref{prop2-thm1} tells us that if $\psi, \varphi$ are partition of unities, then $\B_{p,q}^{\alpha,\psi ,\Delta_d}(\ZZ)\equiv \B_{p,q}^{\alpha,\varphi ,\Delta_d}(\ZZ)$ for $0<p,q\le \vc$  and $\F_{p,q}^{\alpha,\psi ,\Delta_d}(\ZZ)\equiv \F_{p,q}^{\alpha,\varphi ,\Delta_d}(\ZZ)$ for $0<p<\vc$ and $0<q\le \vc$. For this reason, for $\alpha \in \mathbb R, 0<p,q\le \vc$, we define the Besov space $\B_{p,q}^{\alpha,\Delta_d}(\ZZ)$ to be the space $\B_{p,q}^{\alpha,\psi ,\Delta_d}(\ZZ)$ with any partition of unity $\psi$. The Triebel--Lizokin space $\F_{p,q}^{\alpha,\Delta_d}(\ZZ)$ for $\alpha \in \mathbb R, 0<p< \vc, 0<q\le \vc$ can be defined similarly.
\end{remark}

 The following theorem regarding the continuous version of the Besov and Triebel--Lizorkin spaces which is taken from \cite[Theorem 3.5]{BBD}. 
\begin{theorem}
	\label{thm1-continuouscharacter}
	Let $\psi\in \mathcal S(\mathbb R)$ supported in $[2,8]$. Then we have:
	\begin{enumerate}[(a)]
		\item For  $0< p, q\le \vc$, $\alpha\in \mathbb{R}$, $\lambda>1/p$ and $f\in \mathcal{S}'(\ZZ)/\mathcal P(\ZZ)$,
		\begin{equation}\label{eq-B continuous cha}
			\begin{aligned}
				\|f\|_{\B_{p,q}^{\alpha ,\Delta_d}(\ZZ)}&\sim \Big(\int_0^\vc \Big[t^{-\alpha}\|\psi(t\sqrt{\Delta_d})f\|_{\ell^p(\ZZ)}\Big]^q\f{dt}{t}\Big)^{1/q}\\
				&\sim \Big(\int_0^\vc \Big[t^{-\alpha}\|\psi^*_\lambda(t\sqrt{\Delta_d})f\|_{\ell^p(\ZZ)}\Big]^q\f{dt}{t}\Big)^{1/q}.
			\end{aligned}
			\end{equation}
			\item For  $0<q\le \vc$, $\alpha\in \mathbb{R}$,  $\lambda>\max\{1/q, 1/p\}$ and $f\in \mathcal{S}'(\ZZ)/\mathcal P(\ZZ)$,
		\begin{equation}\label{eq-F continuous cha}
					\begin{aligned}
				\|f\|_{\F_{p,q}^{\alpha, \Delta_d}(\ZZ)}&\sim \Big\|\Big(\int_0^\vc \Big[t^{-\alpha}|\psi(t\sqrt{\Delta_d})f|\Big]^q\f{dt}{t}\Big)\Big\|_{\ell^p(\ZZ)}\\&\sim \Big\|\Big(\int_0^\vc \Big[t^{-\alpha}\psi^*_\lambda(t\sqrt{\Delta_d})f\Big]^q\f{dt}{t}\Big)\Big\|_{\ell^p(\ZZ)}.
			\end{aligned}
		\end{equation}
	\end{enumerate}
\end{theorem}

Similarly to the classical Triebel--Lizorkin spaces we also have characterizations for Triebel-Lizorkin spaces via Lusin functions and the Littlewood--Paley functions. To see this, for $\alpha\in \mathbb{R}, \lambda, a>0$, $0<q<\vc$ and a measurable function $F$ defined on $\ZZ\times (0,\vc)$ we define the Lusin function and the Littlewood--Paley function by setting
\begin{equation}\label{g-function}
	\mathcal{G}^{\alpha}_{\lambda, q}F(n)=\left[\int_0^\vc\sum_{m\in \ZZ}(t^{-\alpha}|F(m,t)|)^q\Big(1+\f{|m-n|}{t}\Big)^{-\lambda q} \f{  dt}{t(t+1)}\right]^{1/q}
\end{equation}
and
\begin{equation}\label{lusin-function}
	\mathcal{S}^\alpha_{a,q}F(n)=\left[\int_0^\vc\sum_{m:|m-n|<at}(t^{-\alpha}|F(m,t)|)^q\f{  dt}{t(t+1)}\right]^{1/q}
\end{equation}
for every measurable function $F$, respectively.

	We have the following characterization for the Triebel--Lizorkin spaces which can be viewed as a particular case of Proposition 3.13 in \cite{BBD}.
	\begin{proposition}\label{prop4.1b}
		Let $\psi$ be a partition of unity. Then for  $0<p, q<\vc$, and $\alpha\in \mathbb{R}$, we have
		$$
		\|f\|_{\F^{\alpha,\Delta_d}_{p,q}(\ZZ)}\sim  \|\mathcal{G}^{\alpha}_{\lambda, q}(\psi(t\sqrt{\Delta_d})f)\|_{\ell^p(\ZZ)}\sim \|\mathcal{S}^{\alpha}_{q}(\psi(t\sqrt{\Delta_d})f)\|_{\ell^p(\ZZ)}
		$$
		for all $f\in \mathcal{S}_\vc'(\ZZ)$, provided that $\lambda>\gamma(p,q)$, where
		\begin{equation}
			\label{eq-gamma}
			\gamma(p,q)=\begin{cases}
				\f{1}{p\wedge q}, \ \ \ &p\ge q\\
				\f{2}{q}(1-\f{p}{q}), \ \ \ &p< q.
			\end{cases}
		\end{equation}
	\end{proposition}

\subsection{Molecular decomposition}

In this section, we aim to establish molecular decomposition for our new Besov and Triebel--Lizorkin spaces.

\begin{definition}\label{defLmol}
	Let $0< p\leq \vc$  and $N, M\in \mathbb{N}_+$. A function $a$ is said to be a $(M, N, p)_{\Delta_d}$ molecule if there exist a function $b$  and a dyadic interval $I\in \mathcal{I}$ so that
	\begin{enumerate}[{\rm (i)}]
		\item $a=\Delta_d^M b$;
		
		\item $\displaystyle |\Delta_d^k b(n)|\leq \ell(I)^{2(M-k)}|I|^{-1/p}\Big(1+\f{d(n,I)}{\ell(I)}\Big)^{-1-N}$, $k=0,\ldots , 2M$.
	\end{enumerate}
	
\end{definition}

\begin{proposition}
	\label{prop1-maximal function}
	Let $\psi\in \mathcal{S}(\mathbb{R})$ with ${\rm supp}\,\psi\subset [2,8]$ and  $\varphi\in \mathcal{S}(\mathbb{R})$ be a partition of unity. Then for any $\lambda>0$ and $j\in \mathbb{Z}$ we have
	\begin{equation}
		\label{eq-psistar vaphistar}
		\sup_{s\in [2^{-j-1},2^{-j}]}\psi^*_{\lambda}(s\sqrt{\Delta_d})f(n) \les \sum_{\ell=j-2}^{j+3} \varphi^*_{\ell,\lambda}(\sqrt{\Delta_d})f(n), 
	\end{equation}
	for all $f\in \mathcal{S}'(\ZZ)/\mathcal P(\ZZ)$ and $n\in \ZZ$.
\end{proposition}
\begin{proof}
	Fix $j\in \mathbb{Z}$ and $s\in [2^{-j-1},2^{-j}]$. First note that
	\[
	\psi(s\sqrt{\Delta_d})=\sum_{\ell=j-2}^{j+3}\psi(s\sqrt{\Delta_d})\varphi_\ell(\sqrt{\Delta_d}).
	\]
	Since $2^{-j}\sim s$, by Corollary \ref{lem1} we have,  for $m\in \ZZ$ and $N>n$,
	\begin{equation}\label{eq1-pro1}
		\begin{aligned}
			|\psi(s\sqrt{\Delta_d})f(m)|&\leq \sum_{\ell=j-2}^{j+3}|\psi(s\sqrt{\Delta_d})\varphi_\ell(\sqrt{\Delta_d})f(m)|\\
			&\les  \sum_{\ell=j-2}^{j+3}\sum_{k\in \ZZ} \f{1}{(1+2^{-j})}(1+2^j|m-k|)^{-N-\lambda}|\varphi_\ell(\sqrt{\Delta_d})f(k)| .
		\end{aligned}
	\end{equation}
	It follows that
	\begin{equation}\label{eq2-pro1}
		\begin{aligned}
			\f{|\psi(s\sqrt{\Delta_d})f(m)|}{(1+|m-n|/s)^\lambda}&\les  \sum_{\ell=j-2}^{j+3}\sum_{k\in \ZZ} \f{1}{(1+2^{-j})}(1+2^j|m-k|)^{-N-\lambda}\f{|\varphi_\ell(\sqrt{\Delta_d})f(k)|}{(1+2^j|m-n|)^\lambda} \\
			&\les  \sum_{\ell=j-2}^{j+3}\sum_{k\in \ZZ} \f{1}{(1+2^{-j})}(1+2^j|m-k|)^{-N}\f{|\varphi_\ell(\sqrt{\Delta_d})f(k)|}{(1+2^j|m-n|)^\lambda(1+2^j|m-k|)^\lambda} \\
			&\les  \sum_{\ell=j-2}^{j+3}\sum_{k\in \ZZ} \f{1}{(1+2^{-j})}(1+2^j|m-k|)^{-N}\f{|\varphi_\ell(\sqrt{\Delta_d})f(k)|}{(1+2^jd(x,z))^\lambda} .\\
		\end{aligned}
	\end{equation}
	Using the fact that $2^\ell\sim 2^j$, we obtain 
	\begin{equation}\label{eq2s-pro1}
		\begin{aligned}
			\f{|\psi(s\sqrt{\Delta_d})f(m)|}{(1+|m-n|/s)^\lambda}&\les  \sum_{\ell=j-2}^{j+3}\sum_{k\in \ZZ} \f{1}{(1+2^{-j})}(1+2^j|m-k|)^{-N}\f{|\varphi_\ell(\sqrt{\Delta_d})f(k)|}{(1+2^kd(x,z))^\lambda} \\
			&\les \sum_{\ell=j-2}^{j+3} \varphi^*_{\ell,\lambda}(\sqrt{\Delta_d})f(n)\sum_{k\in \ZZ} \f{1}{(1+2^{-j})}(1+2^j|m-k|)^{-N} \\
			&\les \sum_{\ell=j-2}^{j+3} \varphi^*_{\ell,\lambda}(\sqrt{\Delta_d})f(n),
		\end{aligned}
	\end{equation}
	where in the last inequality we use Lemma \ref{lem-elementary}.
	Taking the supremum over $m\in \ZZ$, we derive \eqref{eq-psistar vaphistar}.
\end{proof}

We first state the molecular decomposition result for the Besov space $\B^{\alpha, \Delta_d}_{p,q}(\ZZ)$.
\begin{theorem}\label{thm1- atom Besov}
	Let $\alpha\in \mathbb{R}$, $0<p,q\leq \vc$, $M\in \mathbb{N}_+$. Assume $f\in \B^{\alpha, \Delta_d}_{p,q}(\ZZ)$. Then there exist a sequence of $(M,N,p)_{\Delta_d}$ atoms $\{a_I\}_{I\in \mathcal{I}_\nu, \nu\in \mathbb{Z}}$ and a sequence of coefficients  $\{s_I\}_{I\in \mathcal{I}_\nu, \nu\in\mathbb{Z}^-}$ so that
	$$
	f=\sum_{\nu\in\mathbb{Z}^-}\sum_{I\in \mathcal{I}_\nu}s_Ia_I \ \ \text{in $\mathcal{S}_\vc'(\ZZ)$}.
	$$
	Moreover,
	$$
	\Big[\sum_{\nu\in\mathbb{Z}^-}2^{\nu \alpha q}\Big(\sum_{I\in \mathcal{I}_\nu}|s_I|^p\Big)^{q/p}\Big]^{1/q}\les  \|f\|_{\B^{\alpha, \Delta_d}_{p,q}(\ZZ)}.
	$$
\end{theorem}

\begin{proof}
	Let $\psi$ be a partition of unity. Due to Proposition \ref{prop-Calderon1}, for $f\in \mathcal{S}'(\ZZ)/\mathcal P(\ZZ)$ we have
	$$
	\begin{aligned}
		f&=c\int_0^\vc (t^2\Delta_d)^M\psi(t\sqrt{\Delta_d})\psi(t\sqrt{\Delta_d})f\f{dt}{t}\\
		&=c\int_1^\vc (t^2\Delta_d)^M\psi(t\sqrt{\Delta_d})\psi(t\sqrt{\Delta_d})f\f{dt}{t}
	\end{aligned}
	$$
	in $\mathcal{S}'(\ZZ)/\mathcal P(\ZZ)$, where $\displaystyle c=\Big[\int_0^\vc \xi^{2M}\psi(\xi)^2 \f{d\xi}{\xi}\Big]^{-1}$.
	
	As a consequence, we have
	\begin{equation}
		\label{eq1-atom}
		\begin{aligned}
			f&=c\sum_{\nu\in \mathbb{Z}^-}\int_{2^{-\nu}}^{2^{-\nu+1}}(t^2\Delta_d)^M\psi(t\sqrt{\Delta_d})[\psi(t\sqrt{\Delta_d})f]\f{dt}{t}\\
			&=c\sum_{\nu\in \mathbb{Z}^-}\sum_{I\in \mathcal{I}_\nu}\int_{2^{-\nu}}^{2^{-\nu+1}}(t^2\Delta_d)^M\psi(t\sqrt{\Delta_d})\big[\psi(t\sqrt{\Delta_d})f\cdot 1_{I}\big]\f{dt}{t}.
		\end{aligned}
	\end{equation}

	For each $\nu\in \mathbb{Z}^-$ and $I\in \mathcal{I}_\nu$, we set
	$$
	s_I= |I|^{1/p}\sup_{m\in I}\int_{2^{-\nu}}^{2^{-\nu+1}}|\psi(t\sqrt{\Delta_d})f(m)|\f{dt}{t},
	$$
	and $a_I=\Delta_d^Mb_I$, where
	\begin{equation}\label{eq-bQ}
		b_I=\f{1}{s_I} \int_{2^{-\nu}}^{2^{-\nu+1}} t^{2M}\psi(t\sqrt{\Delta_d})\big[\psi(t\sqrt{\Delta_d})f\cdot 1_I\big]\f{dt}{t}.
	\end{equation}
	Obviously,  we deduce from \eqref{eq1-atom} that
	$$
	f=\sum_{\nu\in\mathbb{Z}^-}\sum_{I\in \mathcal{I}_\nu}s_Ia_I \ \ \text{in $\mathcal{S}'(\ZZ)/\mathcal P(\ZZ)$}.
	$$
	For $k=0,\ldots,2M$, we have
	$$
	\begin{aligned}
		\Delta_d^{k}b_I(n)&=\f{1}{s_I} \int_{2^{-\nu}}^{2^{-\nu+1}} t^{2(M-k)}(t^2\Delta_d)^k\psi(t\sqrt{\Delta_d})[\psi(t\sqrt{\Delta_d})f\cdot 1_I ]\f{dt}{t}\\
		&=\f{1}{s_I} \int_{2^{-\nu}}^{2^{-\nu+1}}\sum_{m\in I} t^{2(M-k)}K_{(t^2L)^k\psi(t\sqrt{\Delta_d})}(n,m)\psi(t\sqrt{\Delta_d})f(m) \f{dt}{t}.
	\end{aligned}
	$$
	By Corollary \ref{lem1} and the fact that $\ell(I)\sim |I|\sim 2^{-\nu}$,
	$$
	\begin{aligned}
		|\Delta_d^{k}b_I(x)|&\les   2^{-2\nu(M-k)}  \int_{2^{-\nu}}^{2^{-\nu+1}} \sum_{m\in I} \f{1}{2^{-\nu}}\Big(1+\f{|m-n|}{2^{-\nu}}\Big)^{-N}\f{dt}{t}\\
		&\les   (\ell(I))^{-2(M-k)}   \Big(1+\f{d(n,I)}{\ell(I)}\Big)^{-N}\\
	\end{aligned}
	$$
	It follows that $a_I$ is (a multiple of) a $(M,N, p)$ molecule.
	
	We now prove that
	$$
	\Big[\sum_{\nu\in\mathbb{Z}^-}2^{\nu\alpha q}\Big(\sum_{I\in \mathcal{I}_\nu}|s_I|^p\Big)^{q/p}\Big]^{1/q}\sim \|f\|_{\B^{\alpha, \Delta_d}_{p,q}(\ZZ)}.
	$$
	
	Indeed, for any $\lambda>0$ we note that
	\[
	\begin{aligned}
		s_I&\sim |I|^{1/p}\sup_{m\in I}\int_{2^{-\nu}}^{2^{-\nu+1}}|\psi(t\sqrt{\Delta_d})f(m)|\f{dt}{t}\le \Big[\sum_{m\in I} |F^*_{\psi,\nu,\lambda}(\sqrt{\Delta_d})f(m)|^p \Big]^{1/p},
	\end{aligned}
	\]
	where 
	\[
	F^*_{\psi,\nu,\lambda}(\sqrt{\Delta_d})f(n)=\sup_{m\in\ZZ}\f{\displaystyle \int_{2^{-\nu}}^{2^{-\nu+1}}|\psi(t\sqrt{\Delta_d})f(m)|\f{dt}{t}}{(1+2^\nu |m-n|)^\lambda}.
	\]
	As a consequence,
	\[
	\sum_{I\in \mathcal{I}_\nu}|s_I|^p\le \sum_{m\in \ZZ} |F^*_{\psi,\nu,\lambda}(\sqrt{\Delta_d})f(m)|^p .
	\]
	By Proposition \ref{prop1-maximal function},
	\[
	F^*_{\psi,\nu,\lambda}(\sqrt{\Delta_d})f(m)\le \sum_{\ell=\nu-1}^{\nu+4} \psi^*_{\ell,\lambda}(\sqrt{\Delta_d})f(m).
	\]
	This, in combination with Proposition \ref{prop2-thm1}, implies
	\begin{equation}
		\label{eq1-atom Besov}
		\begin{aligned}
			\Big[\sum_{\nu\in\mathbb{Z}^-}2^{\nu\alpha q}\Big(\sum_{I\in \mathcal{I}_\nu}|s_I|^p\Big)^{q/p}\Big]^{1/q}&\leq \Big[\sum_{\nu\in\mathbb{Z}^-}2^{\nu\alpha q}\Big(\Big\|\sum_{\ell=\nu-1}^{\nu+4} \psi^*_{\ell,\lambda}(\sqrt{\Delta_d})f\Big\|_{\ell^p(\ZZ)}\Big)^{q}\Big]^{1/q}\\
			&\les \Big[\sum_{\ell\in\mathbb{Z}^-}2^{\ell\alpha q}\Big(\Big\| \psi^*_{\ell,\lambda}(\sqrt{\Delta_d})f\Big\|_{\ell^p(\ZZ)}\Big)^{q}\Big]^{1/q}\\
			&\les \|f\|_{\B^\alpha_{p,q}(\ZZ)}.
		\end{aligned}
	\end{equation}
	
	This completes the  proof.
\end{proof}

\bigskip

For the reverse direction we have:
\begin{theorem}\label{thm2- atom Besov}
	Let $\alpha\in \mathbb{R}$, $0<p,q\leq \vc$. Assume that 
	$$
	f=\sum_{\nu\in\mathbb{Z}^-}\sum_{I\in \mathcal{I}_\nu}s_Ia_I \ \ \text{in $\mathcal{S}'(\ZZ)/\mathcal P(\ZZ)$},
	$$
	where $\{a_I\}_{I\in \mathcal{I}_\nu, \nu\in \mathbb{Z}^-}$ is a sequence of $(M,N,p)_{\Delta_d}$ molecules and $\{s_I\}_{I\in \mathcal{I}_\nu, \nu\in\mathbb{Z}^-}$ is a sequence of coefficients satisfying
	$$
	\Big[\sum_{\nu\in\mathbb{Z}^-}2^{\nu\alpha q}\Big(\sum_{I\in \mathcal{I}_\nu}|s_I|^p\Big)^{q/p}\Big]^{1/q}<\vc.
	$$
	Then $f\in \B^{\alpha, \Delta_d}_{p,q}(\ZZ)$ and
	$$
	\|f\|_{\B^{\alpha, \Delta_d}_{p,q}(\ZZ)} \les  \Big[\sum_{\nu\in\mathbb{Z}^-}2^{\nu\alpha q}\Big(\sum_{I\in \mathcal{I}_\nu}|s_I|^p\Big)^{q/p}\Big]^{1/q}
	$$
	provided $M>\f{1}{2}+\f{1}{2}\max\{\alpha,\f{1}{1\wedge p\wedge q}-\alpha\}$ and $N>1+\f{1}{1\wedge p\wedge q}$.
\end{theorem}	
\begin{proof}
By the standard argument in \cite{FJ2, FJ1, BBD}, it suffices to prove the following estimate: 
	Let $\psi$ be a partition of unity and  let $a_I$ be a $(M,N+1,p)$ molecule with some $I\in \mathcal{I}_\nu$ with $\nu\in \ZZ^-$. Then for any $t>0$ and  $N'<N$ we have:
	\begin{equation}\label{eq- psi atom}
		|\psi(t\sqrt{\Delta_d}) a_I(n)|\les  \Big(\f{t}{2^{-\nu}}\wedge \f{2^{-\nu}}{t}\Big)^{2M-1}|I|^{-1/p} \Big(1+\f{d(n,I)}{2^{-\nu}\vee t}\Big)^{-N'}.
	\end{equation}

In order to prove \eqref{eq- psi atom}, we now consider two cases: $t\le 2^{-\nu}$ and $t>2^{-\nu}$.

	\noindent{\bf Case 1: $t\le 2^{-\nu}$.} Observe that
	$$
	\psi(t\sqrt{\Delta_d}) a_I=t^{2M}\psi_M(t\sqrt{\Delta_d})(\Delta_d^Ma_I) 
	$$
	where $\psi_M(\lambda)=\lambda^{-2M}\psi(\lambda)$.
	
	This, along with Corollary \ref{lem1} and the definition of the molecules, yields
	$$
	\begin{aligned}
		|\psi(t\sqrt{\Delta_d}) a_I(n)|&\les  \sum_{m\in \ZZ} \f{t^{2M}}{(t+1)}\Big(1+\f{|m-n|}{t}\Big)^{-N'}|a_I(m)| \\
		&\les  \sum_{m\in \ZZ} \Big(\f{t}{2^{-\nu}}\Big)^{2M}|I|^{-1/p}\f{1}{(t+1)}\Big(1+\f{|m-n|}{t}\Big)^{-N'}\Big(1+\f{d(m,I)}{2^{-\nu}}\Big)^{-N-1} \\
		&\les  \Big(\f{t}{2^{-\nu}}\Big)^{2M}|I|^{-1/p}\sum_{m\in \ZZ} \f{1}{t}\Big(1+\f{|m-n|}{2^{-\nu}}\Big)^{-N'}\Big(1+\f{d(m,I)}{2^{-\nu}}\Big)^{-N-1}.
	\end{aligned}
	$$
	We now use the inequality
	\[
	\Big(1+\f{|m-n|}{2^{-\nu}}\Big)^{-N'}\Big(1+\f{d(m,I)}{2^{-\nu}}\Big)^{-N'}\le \Big(1+\f{d(n,I)}{2^{-\nu}}\Big)^{-N'}
	\]
	to obtain further that 
	$$
	\begin{aligned}
		|\psi(t\sqrt{\Delta_d}) a_I(n)|& \les  \Big(\f{t}{2^{-\nu}}\Big)^{2M-1}|I|^{-1/p}\Big(1+\f{d(n,I)}{2^{-\nu}}\Big)^{-N'}\sum_{m\in \ZZ} \f{1}{2^{-\nu}} \Big(1+\f{d(n,I)}{2^{-\nu}}\Big)^{-1-N+N'}\\
		& \les  \Big(\f{t}{2^{-\nu}}\Big)^{2M-1}|I|^{-1/p} \Big(1+\f{d(n,I)}{2^{-\nu}}\Big)^{-N'},
	\end{aligned}
	$$ 
	which yields  \eqref{eq- psi atom}.
	
	\medskip
	
	\noindent{\bf Case 2: $t> 2^{-\nu}$.} We first write  $a_I=\Delta_d^{M}b_I$. Hence,
	\[
	\psi(t\sqrt{\Delta_d})a_I=t^{-2M}\tilde \psi_M(t\sqrt{\Delta_d})b_I
	\]
	where $\tilde \psi_M(\lambda)=\lambda^{2M}\psi(\lambda)$.
	
	This, along with Corollary \ref{lem1}, implies that
	$$
	\begin{aligned}
		|\psi(t\sqrt{\Delta_d})a_I(n)|&\les  \sum_{m\in \ZZ}\f{t^{-2M}}{t}\Big(1+\f{|m-n|}{t}\Big)^{-N'}|b_I(m)| \\
		&\les  \Big(\f{2^{-\nu}}{t}\Big)^{2M} |I|^{-1/p}\sum_{m\in \ZZ}\f{1}{t}\Big(1+\f{|m-n|}{t}\Big)^{-N'}\Big(1+\f{d(m,I)}{2^{-\nu}}\Big)^{-N-1}\\
		&\les  \Big(\f{2^{-\nu}}{t}\Big)^{2M} |I|^{-1/p}\sum_{m\in \ZZ}\f{1}{t}\Big(1+\f{|m-n|}{t}\Big)^{-N'}\Big(1+\f{d(m,I)}{t}\Big)^{-N-1}\\
		&\les  \Big(\f{2^{-\nu}}{t}\Big)^{2M} |I|^{-1/p} \Big(1+\f{d(n,I)}{t}\Big)^{-N'}\sum_{m\in \ZZ}\f{1}{t} \Big(1+\f{d(m,I)}{t}\Big)^{-1-N+N'}\\
		&\les  \Big(\f{2^{-\nu}}{t}\Big)^{2M} |I|^{-1/p} \Big(1+\f{d(n,I)}{t}\Big)^{-N'}.
	\end{aligned}
	$$
	Hence \eqref{eq- psi atom} follows.
	
	This completes our proof.
\end{proof}
\medskip

Similarly we have molecular decomposition theorem for the spaces $\F^{\alpha,\Delta_d}_{p,q}(\ZZ)$. 
\begin{theorem}\label{thm1- atom TL spaces}
	Let $\alpha\in \mathbb{R}$, $0<p<\vc$, $0<q\leq \vc$, $N>0$ and $M\in \mathbb{N}_+$. If $f\in \F^{\alpha,\Delta_d}_{p,q}(\ZZ)$,  then there exist a sequence of $(M,N,p)_{\Delta_d}$ molecules $\{a_I\}_{I\in \mathcal{I}_\nu, \nu\in \mathbb{Z}}$ and a sequence of coefficients  $\{s_I\}_{I\in \mathcal{I}_\nu, \nu\in\mathbb{Z}^-}$ so that
	$$
	f=\sum_{\nu\in\mathbb{Z}^-}\sum_{I\in \mathcal{I}_\nu}s_Ia_I \ \ \text{in $\mathcal{S}'(\ZZ)/\mathcal P(\ZZ)$}.
	$$
	Moreover,
	\begin{equation}\label{eq1-thm1 atom TL space}
		\Big\|\Big[\sum_{\nu\in\mathbb{Z}^-}2^{\nu\alpha q}\Big(\sum_{I\in \mathcal{I}_\nu}|I|^{-1/p}|s_I|\cdot 1_I \Big)^q\Big]^{1/q}\Big\|_{\ell^p(\ZZ)}\les  \|f\|_{\F^{\alpha,\Delta_d}_{p,q}(\ZZ)}.
	\end{equation}
\end{theorem}
\begin{proof}
	Recall that in the proof of Theorem \ref{thm1- atom Besov}, we have proved the representation
	$$
	f=\sum_{\nu\in\mathbb{Z}^-}\sum_{I\in \mathcal{I}_\nu}s_Ia_I \ \ \text{in $\mathcal{S}'(\ZZ)/\mathcal P(\ZZ)$},
	$$
	where
	$$
	s_I= |I|^{1/p}\sup_{m\in I}\int_{2^{-\nu}}^{2^{-\nu+1}}|\psi(t\sqrt{\Delta_d})f(m)|\f{dt}{t},
	$$
	and $a_I=\Delta_db_I$ is a $(M,N,p)_{\Delta_d}$ molecule defined by
	$$
	b_I=\f{1}{s_I} \int_{2^{-\nu}}^{2^{-\nu+1}} t^{2M}\psi(t\sqrt{\Delta_d})[\psi(t\sqrt{\Delta_d})f\cdot 1_I ]\f{dt}{t}.
	$$
	
	It remains to prove \eqref{eq1-thm1 atom TL space}. Indeed, for any $\lambda>0$ it is easy to see that
	\[
	\begin{aligned}
		|I|^{-1/p}s_I\cdot 1_I &= \sup_{m\in I}\int_{2^{-\nu}}^{2^{-\nu+1}}|\psi(t\sqrt{\Delta_d})f(m)|\f{dt}{t}\cdot 1_I \les  1_I  .F^*_{\psi,\nu,\lambda}(\sqrt{\Delta_d})f
	\end{aligned}
	\]
	where 
	\[
	F^*_{\psi,\nu,\lambda}(\sqrt{\Delta_d})f(n)=\sup_{m\in \ZZ}\f{\displaystyle \int_{2^{-\nu}}^{2^{-\nu+1}}|\psi(t\sqrt{\Delta_d})f(m)|\f{dt}{t}}{(1+2^\nu |m-n|)^\lambda}.
	\]
	As a consequence,
	\[
	\sum_{I\in \mathcal{I}_\nu}|I|^{-1/p}|s_I|\cdot 1_I \les  F^*_{\psi,\nu,\lambda}(\sqrt{\Delta_d})f.
	\]
	By Proposition \ref{prop1-maximal function}
	\[
	F^*_{\psi,\nu,\lambda}(\sqrt{\Delta_d})f(m)\le \sum_{\ell=\nu-1}^{\nu+4} \psi^*_{\ell,\lambda}(\sqrt{\Delta_d})f(m).
	\]
	This, in combination with Proposition \ref{prop2-thm1}, implies
	\[
	\begin{aligned}
		\Big\|\Big[\sum_{\nu\in\mathbb{Z}^-}2^{\nu\alpha q}\Big(\sum_{I\in \mathcal{I}_\nu}|I|^{-1/p}|s_I|1_I \Big)^q\Big]^{1/q}\Big\|_{\ell^p(\ZZ)}&\les  \Big\|\Big[\sum_{\nu\in\mathbb{Z}^-}2^{\nu\alpha q}\Big(\sum_{\ell=\nu-1}^{\nu+4} \psi^*_{\ell,\lambda}(\sqrt{\Delta_d})f\Big)^q\Big]^{1/q}\Big\|_{\ell^p(\ZZ)}\\
		&\les  \Big\|\Big[\sum_{\ell\in\mathbb{Z}^-}2^{\ell\alpha q}\Big( \psi^*_{\ell,\lambda}(\sqrt{\Delta_d})f\Big)^q\Big]^{1/q}\Big\|_{\ell^p(\ZZ)}\\
		&\les \|f\|_{\F^\alpha_{p,q}(\ZZ)}.
	\end{aligned}
	\]	
	
	This completes our proof.
\end{proof}
\bigskip

Arguing similarly to the proof of Theorem \ref{thm2- atom Besov}, we are able to prove the converse direction.
\begin{theorem}\label{thm2- atom TL spaces}	
	Let $\alpha\in \mathbb{R}$, $0<p<\vc$ and $0<q\le \vc$. If
	$$
	f=\sum_{\nu\in\mathbb{Z}^-}\sum_{I\in \mathcal{I}_\nu}s_Ia_I \ \ \text{in $\mathcal{S}'(\ZZ)/\mathcal P(\ZZ)$},
	$$
	where $\{a_I\}_{I\in \mathcal{I}_\nu, \nu\in \mathbb{Z}}$ is a sequence of $(M,N,p)_{\Delta_d}$ molecules and $\{s_I\}_{I\in \mathcal{I}_\nu, \nu\in\mathbb{Z}^-}$ is a sequence of coefficients satisfying
	$$
	\Big\|\Big[\sum_{\nu\in\mathbb{Z}^-}2^{\nu\alpha q}\Big(\sum_{I\in \mathcal{I}_\nu}|I|^{-1/p}|s_I|1_I \Big)^q\Big]^{1/q}\Big\|_{\ell^p(\ZZ)}<\vc,
	$$
	then $f\in \F^{\alpha,\Delta_d}_{p,q}(\ZZ)$ and
	$$
	\|f\|_{\F^{\alpha,\Delta_d}_{p,q}(\ZZ)} \les  \Big\|\Big[\sum_{\nu\in\mathbb{Z}^-}2^{\nu\alpha q}\Big(\sum_{I\in \mathcal{I}_\nu}|I|^{-1/p}|s_I|\cdot 1_I \Big)^q\Big]^{1/q}\Big\|_{\ell^p(\ZZ)}
	$$
	provided $M>\f{n}{2}+\f{1}{2}\max\{\alpha,\f{1}{1\wedge p\wedge q}-\alpha\}$ and $N>1+\f{1}{1\wedge p\wedge q}$.
\end{theorem}

\begin{remark}
	\label{remark-molecular decomposition}
	We now consider a new variant of molecules. Recall that $\mathfrak{D}$ is either $D$ or $D^*$.
	\begin{definition}\label{defLmol}
		Let $0< p\leq \vc$, $\alpha\in \mathbb{R}$, and $N, M\in \mathbb{N}_+$. A function $a$ is said to be a $(M, N, p)_{\mathfrak{D}}$ molecule if there exist a function $b$  and a dyadic interval $I\in \mathcal{I}$ so that
		\begin{enumerate}[{\rm (i)}]
			\item $a=\Delta_d^M b$;
			
			\item $\displaystyle |\mathfrak{D}^k b(n)|\leq \ell(I)^{2M-k}|I|^{-1/p}\Big(1+\f{d(n,I)}{\ell(I)}\Big)^{-1-N}$, $k=0,\ldots , 4M$.
		\end{enumerate}
		
	\end{definition}

By the similar argument by making use of Lemma \ref{lem1-derivative of psi delta}, we can show that the results in Theorems \ref{thm1- atom Besov}, \ref{thm2- atom Besov}, \ref{thm1- atom TL spaces} and \ref{thm2- atom TL spaces} still hold true if we replace $(M,N,p)$ molecules by the new $(M, N, p)_{\mathfrak{D}}$ molecules. We would like to leave the details to the interested reader.
\end{remark}

\subsection{Coincidence with the classical Besov and Triebel--Lizorkin spaces}
According to \cite{Sun}, we say that  $(\varphi_\nu,\psi_\nu)_{\nu\in \ZZ^-} \in \mathcal S\times \mathcal S$ is an \textit{admissible pair} on $\ZZ$ if the following conditions hold true
\begin{enumerate}[(i)]
	\item $\supp \mathcal F_{\ZZ}({\varphi}_\nu), \supp \FZ({\psi}_\nu) \subset [-2^{\nu}\pi,2^{\nu}\pi]\backslash [-2^{\nu-2}\pi,2^{\nu-2}\pi]$;
	\item $\displaystyle \sum_{\nu\in \ZZ^-}\overline{\FZ({\varphi}_\nu)(x)} \FZ({{\psi}_\nu})(x)=1$ for $x\in [-\pi,\pi]\backslash\{0\}$;
	\item there exist constants $A$ and $A_k, k\ge 0$, independent of $\nu$, such that
          \begin{itemize}
          	\item $\displaystyle |\partial^k_x \FZ({\varphi}_\nu)(x)|, |\partial^k_x \FZ({\psi}_\nu)(x)|\le A_k2^{-k\nu} $ for $x\in [-\pi,\pi]$, $k\in \mathbb N$ and $\nu\in \ZZ^-$;
          	\item $\displaystyle |\FZ({\varphi}_\nu)(x)|, | \FZ({\psi}_\nu)(x)|\ge A$ for $x\in T_\nu$, where
          	\[
          	T_\nu =\begin{cases}
          		[-\pi,\pi]\backslash [-\frac{3}{8}\pi,\frac{3}{8}\pi] \ \ \ \ &\text{if $\nu=0$},\\
          		[-3.2^{\nu-2}\pi,3.2^{\nu-2}\pi]\backslash [-3.2^{\nu-3}\pi,3.2^{\nu-3}\pi]  \ \ \ \ &\text{if $\nu\le -1$}.
          	\end{cases}
          	\]
          \end{itemize}	
\end{enumerate}

\begin{definition}
	Let $(\varphi_\nu,\psi_\nu)_{\nu \in \ZZ}$ be an
admissible pair on $\ZZ$. 
	\begin{enumerate}[(i)]
	 \item For $\alpha\in \mathbb R$ and $0<p,q\le \vc$, we define the sequence spaces $\B^\alpha_{p,q}(\ZZ)$ of Besov type as the collection of complex-valued sequences $f\in \mathcal S'/\mathcal P(\ZZ)$ such that the quasinorms
	 \begin{equation}
	 	\label{eq-Besov}
	 	\|f\|_{\B^\alpha_{p,q}(\ZZ)}:=\Big(\sum_{\nu\in \mathbb Z^-}\big(2^{\nu\alpha}\|\varphi_\nu\ast f\|_p\big)^q\Big)^{1/q}
	 \end{equation}
 are finite.
	 \item For $\alpha\in \mathbb R, 0<p<\vc$ and $0<q\le \vc$, we define the sequence spaces $\F^\alpha_{p,q}(\ZZ)$ of Triebel--Lizorkin type as the collection of complex-valued sequences $f\in \mathcal S'/\mathcal P(\ZZ)$ such that the quasinorms
	 \begin{equation}
	 	\label{eq-Besov}
	 	\|f\|_{\F^\alpha_{p,q}(\ZZ)}:=\Big\|\Big[\sum_{\nu\in \mathbb Z^-}\big(2^{\nu\alpha}|\varphi_\nu\ast f|\big)^q\Big]^{1/q}\Big\|_{\ell^p(\ZZ)}
	 \end{equation}
	 are finite.
\end{enumerate}
Here $\varphi_\nu\ast f(n) = \sum_{m\in \ZZ} \varphi_\nu(n-m)f(m)$.
\end{definition}

\begin{definition}\label{defL-atom}
	Let $0< p\leq \vc$, $\alpha\in \mathbb{R}$, and $N, M,k\in \mathbb{N}_+$. A function $a$ is said to be a $(M, K, p)$ atom if there exists  a dyadic interval $I$ with $\ell(I)\ge 1$ so that the following holds true
	\begin{enumerate}[(i)]
		\item $\supp a \subset 3I$;
		\item $\displaystyle |D^\ell a(n)|\leq \ell(I)^{-2\ell}|I|^{-1/p}$, $\ell=0,\ldots , M$;
		\item $\displaystyle \sum_{n\in \ZZ} n^\beta a(n)=0$ for any $\beta\in \mathbb N$ and $\beta \le K$.
	\end{enumerate}
\end{definition}

\begin{definition}\label{defL-mol}
	Let $0< p\leq \vc$, $\alpha\in \mathbb{R}$, and $N, M,k\in \mathbb{N}_+$. A function $a$ is said to be a $(M, K, p)$ molecule if there exists  a dyadic interval $I$ with $\ell(I)\ge 1$ so that the following holds true
	\begin{enumerate}[(i)]
		\item $\displaystyle |D^\ell a(n)|\leq \ell(I)^{-2\ell}|I|^{-1/p}\Big(1+\f{d(n,I)}{\ell(I)}\Big)^{-1-N}$ for all $\ell=0,\ldots , M$, \ \ $n\in \ZZ$;
		\item $\displaystyle \sum_{n\in \ZZ} n^\beta a(n)=0$ for any $\beta\in \mathbb N$ and $\beta \le K$.
	\end{enumerate}
\end{definition}

For $\alpha\in \mathbb R$ and $0<p,q\le \vc$, we set $M_0, N_0, K_0$ to be the smallest integers such that  
\begin{equation}\label{eq-parameter}
M_0\ge  \alpha +2, K_0 \ge \f{1}{1\wedge p\wedge q} -\alpha\ \ \ \text{and}  \ \ \ N_0\ge \f{1}{1\wedge p\wedge q} +1 
\end{equation}

\begin{theorem}
	\label{thm-BS TL atomic decomposition}
	{\rm (a)} Let $\alpha\in \mathbb R, 0<p<\vc,0<q\le \vc$ and let $M> M_0, N> N_0, K> K_0$. For each $f\in \B^\alpha_{p,q}(\ZZ)$, there exist  is a sequence of $(M, K, p)$ atoms $\{a_I\}_{I\in \mathcal I}$ and a sequence of numbers $ \{s_I\}_{I\in \mathcal I}$ such that   
	\begin{equation*}\label{eq-TL representation}
		f= \sum\limits_{\nu  \in \mathbb{Z}^-} {\sum\limits_{I \in {\mathcal{I}_\nu }} {{s_I}{a_I}} } \text{ in  $\mathcal S'/\mathcal P(\ZZ)$} 
	\end{equation*}
	and
	\[
	\|f\|_{\F^\alpha_{p,q}(\ZZ)}\sim \Big\| {{{\Big[ {\sum\limits_{\nu  \in \mathbb{Z}^-} {{2^{\nu \alpha q}}{{\Big( {\sum\limits_{I \in {\mathcal{I}_\nu }} {|I|^{ - 1/p}}| {{s_I}}|.{1 _I}} } \Big)}^q}} } \Big]}^{1/q}} \Big\|_{\ell^p(\ZZ)}.
	\]
	Conversely, if 
	\[f= \sum\limits_{\nu  \in \mathbb{Z}^-} {\sum\limits_{I \in {\mathcal{I}_\nu }} {{s_I}{b_I}} } \text{ in  $\mathcal S'/\mathcal P(\ZZ)$}\]
	for some sequence of $(M, N, K, p)$ molecules $\{b_I\}_{I\in \mathcal I}$ and some sequence of numbers $ \{s_I\}_{I\in \mathcal I}$, then we have
	\[
	\|f\|_{\F^\alpha_{p,q}(\ZZ)}\les \Big\| {{{\Big[ {\sum\limits_{\nu  \in \mathbb{Z}^-} {{2^{\nu \alpha q}}{{\Big( {\sum\limits_{I \in {\mathcal{I}_\nu }} {|I|^{ - 1/p}}| {{s_I}}|.{1 _I}} } \Big)}^q}} } \Big]}^{1/q}} \Big\|_{\ell^p(\ZZ)}.
	\]
	
	\noindent {\rm (b)} Let $\alpha\in \mathbb R$ and $0<p,q\le \vc$ and let $M\ge M_0, N\ge N_0, K\ge K_0$. For each $f\in \B^\alpha_{p,q}(\ZZ)$, there exist  is a sequence of $(M, K, p)$ atoms $\{a_I\}_{I\in \mathcal I}$ and a sequence of numbers $ \{s_I\}_{I\in \mathcal I}$ such that   
	\begin{equation*}\label{eq-Besov representation}
	f= \sum\limits_{\nu  \in \mathbb{Z}^-} {\sum\limits_{I \in {\mathcal{I}_\nu }} {{s_I}{a_I}} } \text{ in  $\mathcal S'/\mathcal P(\ZZ)$} 
	\end{equation*}
	and
	\[
	\|f\|_{\B^\alpha_{p,q}(\ZZ)}\sim \Big[ {\sum\limits_{\nu  \in \ZZ^-} {{2^{\nu \alpha q}}{{\Big( {\sum\limits_{I \in {\mathcal{I}_\nu }} {|{s_I}{|^p}} } \Big)}^{q/p}}} } \Big]^{1/q}.
	\]
	Conversely, if 
	\[f= \sum\limits_{\nu  \in \mathbb{Z}^-} {\sum\limits_{I \in {\mathcal{I}_\nu }} {{s_I}{b_I}} } \text{ in  $\mathcal S'/\mathcal P(\ZZ)$}\]
	for some sequence of $(M, N, K, p)$ molecules $\{b_I\}_{I\in \mathcal I}$ and some sequence of numbers $ \{s_I\}_{I\in \mathcal I}$, then we have
	\[
	\|f\|_{\B^\alpha_{p,q}(\ZZ)}\les \Big[ {\sum\limits_{\nu  \in \ZZ^-} {{2^{\nu \alpha q}}{{\Big( {\sum\limits_{I \in {\mathcal{I}_\nu }} {|{s_I}{|^p}} } \Big)}^{q/p}}} } \Big]^{1/q}.
	\]
\end{theorem}
\begin{proof}
	(a) The first direction was proved in \cite[Theorem 3]{Sun}. The converse direction was also proved in \cite[Theorem 3]{Sun} for the atomic decomposition; however, the argument still works well for the molecular decomposition.
	
	(b) Once the proof for the Triebel--Lizorkin spaces has been done, the proof for the Besov spaces can be done similarly as a folklore. Hence, we would like to leave the details to the interested reader.
	
	This completes our proof.
\end{proof}

Our main result in this section is the following.

\begin{theorem}
	\label{thm-coincidence Besov and TL}
	We have the following identification
	\[
	\B^{\alpha}_{p,q}(\ZZ) =\B^{\alpha, \Delta_d}_{p,q}(\ZZ), \ \ \ \ 0<p,q\le \vc, \alpha \in \mathbb R,
	\]
	and
	\[
	\F^{\alpha}_{p,q}(\ZZ) =\F^{\alpha, \Delta_d}_{p,q}(\ZZ), \ \ \ \ 0<p< \vc, 0<q\le \vc, \alpha \in \mathbb R.
	\]
\end{theorem}
\begin{proof}
	We only prove 
	\[
	\F^{\alpha}_{p,q}(\ZZ) =\F^{\alpha, \Delta_d}_{p,q}(\ZZ), \ \ \ \ 0<p< \vc, 0<q\le \vc, \alpha \in \mathbb R,
	\]
	since the proof of the identification of the Besov spaces can be done similarly. 
	
	Firstly, let $a$ be a $(M, N, p)_{\Delta_d}$ molecule. Then $a$ is also a $(M,N,2M-1,p)$ molecule. This implies $\F^{\alpha, \Delta_d}_{p,q}(\ZZ)\hookrightarrow\F^{\alpha}_{p,q}(\ZZ)$.

	By the standard argument as in \cite{FJ2, FJ1, BBD}, the inclusion $\F^{\alpha}_{p,q}(\ZZ)\hookrightarrow \F^{\alpha, \Delta_d}_{p,q}(\ZZ)$ follows directly from the following estimates: Let $a$ be a $(M,K, p)$ atom with some $I\in \mathcal{I}_\nu$ with $\nu\in \ZZ^-$. Let $\psi\in \mathcal S(\mathbb R)$ supported in $[2,8]$. Then for any $t>0$ and  $N>0$ we have:
		\begin{equation}\label{eq2- psi atom}
			|\psi(t\sqrt{\Dd}) a(n)|\les  \Big(\f{t}{2^{-\nu}}\wedge \f{2^{-\nu}}{ t}\Big)^{M-1}|I|^{-1/p} \Big(1+\f{d(n,I)}{2^{-\nu}\vee  t}\Big)^{-N}.
		\end{equation}
	
	Indeed, we first claim that for $t>0$ and $N>0$,
	\begin{equation}\label{eq3- psi atom}
		|(t\Delta_d)^{M(N+1)}e^{-t\Delta_d} a(n)|\les  \Big(\f{\sqrt t}{2^{-\nu}}\wedge \f{2^{-\nu}}{\sqrt t}\Big)^{M-1}|I|^{-1/p} \Big(1+\f{d(n,I)}{2^{-\nu}\vee \sqrt t}\Big)^{-N},
	\end{equation}
where $a$ is a $(M,K, p)$ atom with some $I\in \mathcal{I}_\nu$ with $\nu\in \ZZ^-$.

To do this, we now consider two cases: $t\le 2^{-\nu}$ and $t>2^{-\nu}$.

	\noindent{\bf Case 1: $\sqrt t\le 2^{-\nu}$.} Observe that
	$$
	(t\Delta_d)^{M(N+1)}e^{-t\Delta_d} a=t^{M}D^M(t\Delta_d)^{MN}e^{-t\Delta_d} (D^Ma). 
	$$

    Applying Lemma \ref{lem-htk and HIGHER difference derivatives}, then arguing similarly to the proof of \eqref{eq- psi atom}, 	we obtain
	$$
	\begin{aligned}
		|(t\Delta_d)^{M(N+1)}e^{-t\Delta_d}a(n)|
		& \les  \Big(\f{\sqrt t}{2^{-\nu}}\Big)^{M-1}|I|^{-1/p} \Big(1+\f{d(n,I)}{2^{-\nu}}\Big)^{-N},
	\end{aligned}
	$$ 
	which yields  \eqref{eq- psi atom}.
	
	\medskip
	
	\noindent{\bf Case 2: $\sqrt t> 2^{-\nu}$.} 
	
	Using \eqref{eq-integration by parts fomular for the time derivative of heat kernel}, we have for each $k,\ell\in \mathbb N$,
	\[
	\begin{aligned}
	 (-\Delta_d)^{\ell}e^{-t\Delta_d} a(n)&=\sum_{m\in I}\f{(-1)^{\lfloor (k+1)/2\rfloor}}{\pi {(n-m)}^k}\int_0^\pi \partial_\theta^k\partial_{t}^\ell  \varphi_{t}(\theta)\Phi_k((n-m)\theta)a(m) d\theta\\
	 &:=\sum_{m\in I}  g_{\ell,k,t}(n-m)a(m).
\end{aligned}
	\]
	Due to the vanishing moment condition of the atom $a$, for a fix $n_I\in I$,
	\[
	(-\Delta_d)^{\ell}e^{-t\Delta_d} a(n) =\sum_{m\in I}  \Big[g_{\ell,k,t}(n-m)-\sum_{j=0}^K \f{(n_I-m)^j}{j!}g^{(j)}_{\ell,k,t}(n-n_I)\Big]a(m).
	\]
Arguing similarly to \eqref{eq3- Taylor 1}, we obtain
	\begin{equation}\label{eq- Taylor 1}
	\begin{aligned}
		|(t\Delta_d)^{M(N+1)}e^{-t\Delta_d}a(n)|&\les  \sum_{m\in I} \Big(\f{2^{-\nu}}{|n-m|\wedge \sqrt t}\Big)^{K+1}\f{1}{\sqrt t}\Big(1+\f{|n-m|}{\sqrt t}\Big)^{-N}|a(m)|.
	\end{aligned}
\end{equation}
	On the other hand, from \eqref{eq2-heat kernel},
	\[
	h_t(n) =\f{1}{\pi}\int_0^\pi e^{-2t(1-\cos\theta)}\cos(n\theta) d\theta
	\]
	Extending the function $h_t$ from $\ZZ$ to $\mathbb R$ by setting
	\[
	h_t(x) =\f{1}{\pi}\int_0^\pi e^{-2t(1-\cos\theta)}\cos(x\theta) d\theta, \ \ \ \ x\in \mathbb R.
	\]
	It is easy to  see that for  $k, \ell \in \mathbb N$,
	\[
	|\partial_x^k\partial_t^\ell h_t(x)| \les   \int_0^\pi \theta^k(1-\cos\theta)^\ell e^{-2t(1-\cos\theta)} d\theta.
	\]
	Since $1-\cos\theta \sim \theta^2$, by  a straightforward calculation,
	 \[
	 |\partial_x^k\partial_t^\ell h_t(x)| \les   \f{1}{t^{\ell+k/2}}.
	 \]
	 This, along with the vanishing moment condition of the atom and Taylor's theorem, implies
	 \[
\begin{aligned}
	|(t\Delta_d)^{M(N+1)}&e^{-t\Delta_d}a(n)| \\
	&=\sum_{m\in I}  t^{M(N+1)}\Big[\partial_t^{M(N+1)}h_t(m-n)-\sum_{j=0}^K \f{(n_I-m)^j}{j!}\partial_x^j\partial_t^{M(N+1)}h_t(n-n_I)\Big]a(m)\\
	&\les \sum_{m\in I} \Big(\f{2^{-\nu}}{\sqrt t}\Big)^{K+1}|a(m)|.
\end{aligned}
	 \]
From this and \eqref{eq- Taylor 1} we conclude that
\[
\begin{aligned}
		|(t\Delta_d)^{M(N+1)}e^{-t\Delta_d}a(n)|&\les  \sum_{m\in I} \Big(\f{2^{-\nu}}{\sqrt t}\Big)^{K+1}\f{1}{\sqrt t}\Big(1+\f{|n-m|}{\sqrt t}\Big)^{-N}|a(m)|\\
		&\les  \sum_{m\in I} \Big(\f{2^{-\nu}}{\sqrt t}\Big)^{K+1}\f{2^{-\nu}}{\sqrt t}\Big(1+\f{d(n, I)}{\sqrt t}\Big)^{-N},
\end{aligned}
\]	 
which ensures \eqref{eq2- psi atom}.

We now turn to the proof  of  \eqref{eq3- psi atom}. We also consider two cases: $t\le 2^{-\nu}$ and $t>2^{-\nu}$. 

\medskip	
	
	If $t\le 2^{-\nu}$, we write
	\[
	\psi(\lambda) = \lambda^{M(N+1)}e^{-\lambda^2}\widetilde\psi(\lambda),
	\]
	where $\widetilde\psi(\lambda)=\lambda^{-M(N+1)}e^{\lambda^2}\psi(\lambda)$.
	
	We then apply  \eqref{eq- psi atom} and Lemma \ref{lem1} to obtain
	\[
	\begin{aligned}
	|\psi(t\sqrt {\Delta_d})a(n)| & = |\widetilde\psi(t\sqrt{\Dd})\circ (t^2\Dd)^{M(N+1)}e^{-t^2\Dd}a(n)|\\
	&\les \sum_{m\in \ZZ}\Big(\f{ t}{2^{-\nu}} \Big)^{M-1}|I|^{-1/p} \f{1}{t}\Big(1+\f{|m-n|}{t}\Big)^{-N-2}\Big(1+\f{d(m,I)}{2^{-\nu}}\Big)^{-N}.	
\end{aligned}
	\]
	Using the inequality
	\[
	\Big(1+\f{|m-n|}{t}\Big)^{-N-2}\Big(1+\f{d(m,I)}{2^{-\nu}}\Big)^{-N}\le \Big(1+\f{|m-n|}{t}\Big)^{-2}\Big(1+\f{d(n,I)}{2^{-\nu}}\Big)^{-N},
	\]
	we further obtain
	\[
	\begin{aligned}
		|\psi(t\sqrt {\Delta_d})a(n)| &\les \Big(\f{ t}{2^{-\nu}} \Big)^{M-1}|I|^{-1/p}\Big(1+\f{d(n,I)}{2^{-\nu}}\Big)^{-N}\sum_{m\in \ZZ} \f{1}{t}\Big(1+\f{|m-n|}{t}\Big)^{-2}\\	
		&\les \Big(\f{  t}{2^{-\nu}} \Big)^{M-1}|I|^{-1/p}\Big(1+\f{d(n,I)}{2^{-\nu}}\Big)^{-N}.
	\end{aligned}
	\]
	The second case $t\ge  2^{-\nu}$ can be done similarly.
	
	This completes the proof of \eqref{eq3- psi atom} and hence the proof of the inclusion $\F^{\alpha}_{p,q}(\ZZ)\hookrightarrow \F^{\alpha, \Delta_d}_{p,q}(\ZZ)$.
	
	This completes our proof.
\end{proof}

We also have  further identifications.
\begin{theorem}
	We have the following identifications.
	\begin{enumerate}[\rm (a)]
		\item $\F_{p,2}^0(\ZZ)= H^p(\ZZ)$ for all $0<p\le 1$;
		\item $\F_{p,2}^0(\ZZ)= \ell^p(\ZZ)$ for all $1<p<\vc$.
	\end{enumerate}
\end{theorem}
\begin{proof}
\noindent	(a) This item follows directly from Theorem \ref{thm-coincidence Besov and TL}, Proposition \ref{prop4.1b} and Remark \ref{rem2} (b).
	
\noindent	(b) For this identification we refer \cite{Sun}.
	
	This completes our proof. 
\end{proof}

We also recall an interpolation result for the  Triebel--Lizorkin spaces.
\begin{proposition}
	\label{prop-comple interpolation}
	We have
	\begin{equation}
		\label{complex interpolation}
		\left(\F^{\alpha_0}_{p_0,q_0}(\ZZ),\F^{\alpha_1}_{p_1,q_1}(\ZZ)\right)_\theta = \F^{\alpha}_{p,q}(\ZZ)
	\end{equation}
	for all $\alpha_0, \alpha_1 \in \mathbb{R}$, $0<p_0,p_1, q_0, q_1<\vc$, $\theta\in (0,1)$ and
	\[
	\alpha=(1-\theta)\alpha_0 +\theta \alpha_1, \ \ \ \f{1}{p} =\f{1-\theta}{p_0}+\f{\theta}{p_1}, \ \ \ \f{1}{q} =\f{1-\theta}{q_0}+\f{\theta}{q_1}
	\] 	
	where $(\cdot, \cdot)_\theta$ stands for the complex interpolation brackets.
\end{proposition}
\section{Applications}

\subsection{Spectral multipliers}

\begin{theorem}\label{mainthm-spectralmultipliers} 	Let   $\alpha\in \mathbb{R}$ and $0<p,q<\vc$.  Suppose that $F: [0,\vc)\to \mathbb C$ be a bounded Borel function  satisfying 
	\[
	\sup_{t\ge 1/2}\|\eta\, \delta_tF\|_{W^{s}_r}<\vc
	\]
for $s> \f{1}{1\wedge p\wedge  q}-\f{1}{2} $ and for some $r\in (4,\vc]$, 	where $\delta_tF(\cdot)= F(t\cdot)$ and $\eta$ is a $C^\vc_c(\mathbb R_+)$ function supported in $[2,8]$, not identically zero. Then we have:
	\begin{enumerate}[\rm (a)]
		\item the spectral multiplier $F(\sqrt \Dd)$ is bounded on $\F^{\alpha}_{p,q}(\ZZ)$, i.e.,
		\[
		\|F(\sqrt{\Dd})\|_{\F^{\alpha}_{p,q}(\ZZ)\to \F^{\alpha}_{p,q}(\ZZ)}\les |F(0)|+\sup_{t\ge 1/2}\|\eta\, \delta_tF\|_{W^{s}_r};
		\]
		\item the spectral multiplier $F(\sqrt \Dd)$ is bounded on $\B^{\alpha}_{p,q}(\ZZ)$, i.e.,
		$$
		\|F(\sqrt \Dd)\|_{\B^{\alpha}_{p,q}(\ZZ)\to \B^{\alpha}_{p,q}(\ZZ)}\les |F(0)| +\sup_{t\ge 1/2}\|\eta\, \delta_tF\|_{W^{s}_r}.
		$$		
	\end{enumerate}
\end{theorem}

By the argument used in the proof of \cite[Theorem 1.1]{BD}, it suffices to prove the following result.
\begin{theorem}\label{thm-spectralmultipliers} 		
	\begin{enumerate}[(a)]
		\item Let   $\alpha\in \mathbb{R}$ and $0<p<\vc$.  Suppose that $F: [0,\vc)\to \mathbb C$ be a bounded Borel function satisfying 
		\[
		\sup_{t\ge 1/2}\|\eta\, \delta_tF\|_{W^{s}_r}<\vc
		\]
		for $s> \f{1}{1\wedge p}-\f{1}{2} $ and for some $r\in (4,\vc]$. Then the spectral multiplier $F(\sqrt{\Dd})$ is bounded on $\F^{\alpha}_{p,2}(\ZZ)$; moreover,
		\begin{equation}\label{eq1-thm1}
			\|F(\sqrt{\Dd})\|_{\F^{\alpha}_{p,2}(\ZZ)\to \F^{\alpha}_{p,2}(\ZZ)}\les  \|F\|_{L^\vc} +\sup_{t>0}\|\eta\, \delta_tF\|_{W^{s}_r}.
		\end{equation}
		\item Let   $\alpha\in \mathbb{R}$ and $0<p,q<\vc$.  Suppose that $F: [0,\vc)\to \mathbb C$ be a bounded Borel function  satisfying
		\[
		\sup_{t\ge 1/2}\|\eta\, \delta_tF\|_{W^{s}_r}<\vc
		\]
		for $s> \f{1}{1\wedge p\wedge  q}$ and for some $r\in (4,\vc]$. Then the spectral multiplier $F(\sqrt{\Dd})$ is bounded on $\F^{\alpha}_{p,q}(\ZZ)$ provided that $\alpha\in \mathbb{R}$, $0<p,q< \vc$ and $s>\f{1}{1\wedge p\wedge q} $; moreover,
		\begin{equation}\label{eq2-thm1}
			\|F(\sqrt{\Dd})\|_{\F^{\alpha}_{p,q}(\ZZ)\to \F^{\alpha}_{p,q}(\ZZ)}\les \|F\|_{L^\vc}+ \sup_{t>0}\|\eta\, \delta_tF\|_{W^{s}_r}.
		\end{equation}
	\end{enumerate}
\end{theorem}

\begin{proof} 
(a) We first prove that if $\sup_{t>0}\|\eta\, \delta_tF\|_{W^{s}_r}$ for some $r\in (4,\vc]$,  $s>\f{1}{p}-\f{1}{2}$ for $0<p\le 1$, then $F(\sqrt{\Dd})$ is bounded on $H^p(\ZZ)\equiv \F^0_{p,2}(\ZZ)$. Indeed, let $\varphi$ be a partition of unity. Define
\[
S_\varphi f(n) = \Big[\int_0^\vc |\varphi(t\sqrt{\Delta_d})f(n)|^2\f{dt}{t}\Big]^{1/2}.
\]
Since $\supp \varphi \subset [2,8]$ and the spectral of $\Delta_d$ is contained in $[0,4]$,
\[
S_\varphi f(n) = \Big[\int_1^\vc |\varphi(t\sqrt{\Delta_d})f(n)|^2\f{dt}{t}\Big]^{1/2}.
\]

Suppose $a=\Delta_d^Mb$ is a $(p,M,\epsilon)_{\Delta_d}$ molecule associated to an interval $I$. Then we have
\[
\begin{aligned}
	\varphi(t\sqrt{\Delta_d})&F(\sqrt{\Delta_d})a\\
	& = \varphi(t\sqrt{\Delta_d})F(\sqrt{\Delta_d})(I-e^{-\ell(I)^2\Delta_d})^Ma + \sum_{i=1}^M c_i \varphi(t\sqrt{\Delta_d})F(\sqrt{\Delta_d})\Delta_d^Me^{-i\ell(I)^2\Delta_d} b,
\end{aligned}
\]
which implies that 
\begin{equation}\label{eq1-proof of spectral multiplier}
	\begin{aligned}
		\big\|S_\varphi (F(\sqrt{\Delta_d})a)\big\|^p_{\ell^p(\ZZ)}&\les \big\|S_\varphi \big[F(\sqrt{\Delta_d})(I-e^{-\ell(I)^2\Delta_d})^Ma\big]\big\|^p_{\ell^p(\ZZ)} + \sum_{i=1}^M\big\|S_\varphi\big[F(\sqrt{\Delta_d})\Delta_d^Me^{-i\ell(I)^2\Delta_d} b\big]\big\|^p_{\ell^p(\ZZ)}\\
		&=: G + H.
	\end{aligned}
\end{equation}
We need only to estimate the first term $G$ since the second term $H$ can be done similarly.

To do this,  we write
\[
\begin{aligned}
	G&\le \sum_{j\ge 1} \big\|S_\varphi \big[F(\sqrt{\Delta_d})(I-e^{-\ell(I)^2\Delta_d})^M(a\cdot 1_{S_j(I)})\big]\big\|^p_{\ell^p(\ZZ)}\\
	&\le \sum_{\ell\ge 0}\sum_{j\ge 1} \big\|S_\varphi \big[F(\sqrt{\Delta_d})(I-e^{-\ell(I)^2\Delta_d})^M(a\cdot 1_{S_j(I)})\big]\big\|^p_{\ell^p(S_\ell(2^jI))}\\
	&:=\sum_{\ell\ge 0}\sum_{j\ge 1}G_{j\ell}.
\end{aligned}
\]
We now apply H\"older's inequality to obtain
\[
G_{jk}  \le |2^{j+\ell}I|^{\f{2-p}{2}}\big\|S_\varphi \big[F(\sqrt{\Delta_d})(I-e^{-\ell(I)^2\Delta_d})^M(a\cdot 1_{S_j(I)})\big]\big\|^p_{\ell^2(S_\ell(2^jI))}.
\]
For $\ell=0,1,2,3$ and $j\ge 0$, by H\"older's inequality and the $L^2$-boundedness of $S_\varphi$ and $(I-e^{-\ell(I)^2\Delta_d})^M$, 
\[
\begin{aligned}
	G_{jk} & \le |2^jI|^{\f{2-p}{2}}\big\|S_\varphi \big[F(\sqrt{\Delta_d})(I-e^{-\ell(I)^2\Delta_d})^M(a\cdot 1_{S_j(I)})\big]\big\|^p_{\ell^2(S_\ell(2^jI))}	\\
	& \les |2^jI|^{\f{2-p}{2}}\|a\|_{\ell^2(S_j(I))}^p\\
	& \les 2^{-j\epsilon p} \sim 2^{-(j+\ell)\epsilon p}.	
\end{aligned}
\]
For $\ell\ge 4$, we have 
\[
\begin{aligned}
	\big\|S_\varphi \big[F(\sqrt{\Delta_d})&(I-e^{-\ell(I)^2\Delta_d})^M(a\cdot 1_{S_j(I)})\big]\big\|^2_{\ell^2(S_\ell(2^jI))}\\
	&=\int_1^\vc \Big\|\varphi(t\sqrt{\Delta_d})F(\sqrt{\Delta_d})(I-e^{-\ell(I)^2\Delta_d})^M(a\cdot 1_{S_j(I)})\Big\|^2_{\ell^2(S_\ell(2^jI))}\f{dt}{t}\\
	&= \int_1^\vc \Big\|\sum_{m\in S_j(I)}K_{\varphi(t\sqrt{\Delta_d})F(\sqrt{\Delta_d})(I-e^{-\ell(I)^2\Delta_d})^M}(n,m)a(m)\Big\|^2_{\ell_n^2(S_\ell(2^jI))}\f{dt}{t}\\
	&\les  \int_1^\vc \Big[\sum_{m\in S_j(I)}\big\|K_{\varphi(t\sqrt{\Delta_d})F(\sqrt{\Delta_d})(I-e^{-\ell(I)^2\Delta_d})^M}(n,m)\big\|_{\ell^2_n(S_\ell(2^jI))}|a(m)|\Big]^2\f{dt}{t},
\end{aligned}
\] 
where in the last inequality we used  Minkowski's inequality.

We note that in this situation $|m-n|\sim 2^{j+\ell}\ell(I)$. We thus apply Proposition \ref{prop1-spectral multiplier} to deduce that for $s>s'>1/p-1/2$,
\[
\begin{aligned}
	\big\|S_\varphi &\big[F(\sqrt{\Delta_d})(I-e^{-\ell(I)^2\Delta_d})^M(a\cdot 1_{S_j(I)})\big]\big\|^2_{\ell^2(S_\ell(2^jI))}\\
	&\les \int_1^\vc t^{-1}\Big\|\delta_{t^{-1}}\big[\varphi(t\sqrt{\Delta_d})F(\sqrt{\Delta_d})(I-e^{-\ell(I)^2\Delta_d})^M\big]\Big\|^2_{W^s_2}t^{2s'}(2^{j+\ell}\ell(I))^{-2s'}\|a\|^2_{\ell^1(S_j(I))}\f{dt}{t}. 
\end{aligned}
\]
Note that 
\[
\Big\|\delta_{t^{-1}}\big[\varphi(t\sqrt{\lambda})F(\sqrt{\lambda})(I-e^{-\ell(I)^2\lambda})^M\big]\Big\|^2_{W^s_2}\les \min\{1, (\ell(I)/t)^{4M}\},
\]
and
\[
\|a\|^2_{\ell^1(S_j(I))}\le |2^jI| \|a\|^2_{\ell^2(S_j(I))}\les 2^{-2j\epsilon} |2^jI|^{2-2/p}.
\]
Therefore,
\[
\begin{aligned}
	\big\|S_\varphi \big[F(\sqrt{\Delta_d})&(I-e^{-\ell(I)^2\Delta_d})^M(a\cdot 1_{S_j(I)})\big]\big\|^2_{\ell^2(S_\ell(2^jI))}\\
	&\les \int_1^\vc t^{-1} \min\{1, (\ell(I)/t)^{4M}\} t^{2s'}(2^{j+\ell}\ell(I))^{-2s'}2^{-2j\epsilon} |2^jI|^{2-2/p} \f{dt}{t}\\
	&\les 2^{-2j\epsilon}2^{-2s'(j+\ell)}|2^jI|^{2-2/p}\int_0^\vc t^{-1}(t/\ell(I))^{2s'} \min\{1, (\ell(I)/t)^{4M}\} \f{dt}{t}\\
	&\les 2^{-2j\epsilon}2^{-2s'(j+\ell)}|2^jI|^{2-2/p} \ell(I)^{-1}.
\end{aligned}
\]
It follows that for $\ell\ge 4,$
\[
\begin{aligned}
	G_{j\ell}&\les 2^{-jp\epsilon}2^{-s'p(j+\ell)}|2^jI|^{p-1} |I|^{-p/2} |2^{j+\ell}I|^{\f{2-p}{2}}\\
	&\les 2^{-jp\epsilon} 2^{-jp(s'-1/2)}2^{-\ell p(s' - (1/p-1/2))}.
\end{aligned}
\]
Since $s'>1/p-1/2\ge 1/2$, we have
\[
G\le \sum_{j,\ell\ge 0} G_{j\ell}\les 1,
\]
which implies
$$
\big\|S_\varphi \big[F(\sqrt{\Delta_d})(I-e^{-\ell(I)^2\Delta_d})^Ma\big]\big\|^p_{\ell^p(\ZZ)}\les 1.
$$

We have just proved that 
\begin{equation*}
	\|F(\sqrt{\Dd})\|_{\F^{0}_{p,2}(\ZZ)\to \F^{0}_{p,2}(\ZZ)}\les  \|F\|_{L^\vc} +\sup_{t>0}\|\eta\, \delta_tF\|_{W^{s}_r}
\end{equation*}
for for some $r\in (4,\vc]$, $s>\f{1}{p}-\f{1}{2}$ and $0<p\le 1$.

On the other hand, we have $\|F(\sqrt \Dd)\|_{\ell^2(\ZZ)\to \ell^2(\ZZ)}\le \|F\|_{\vc}$. This, along with the interpolation in Proposition \ref{prop-comple interpolation}, implies 
\begin{equation}\label{eq1-thm1-Hardy}
	\|F(\sqrt{\Dd})\|_{\F^{0}_{p,2}(\ZZ)\to \F^{0}_{p,2}(\ZZ)}\les  \|F\|_{L^\vc} +\sup_{t>0}\|\eta\, \delta_tF\|_{W^{s}_r}
\end{equation}
for for some $r\in (4,\vc]$, $s>\f{1}{p}-\f{1}{2}$ and $0<p<\vc$.

In addition, for $\alpha\in \mathbb R$
\[
\begin{aligned}
	\|F(\sqrt \Dd) f\|_{\F^{\alpha}_{2,2}(\ZZ)}&=\|F(\sqrt \Dd) f\|_{\B^{\alpha}_{2,2}(\ZZ)}\\
	&\sim \Big(\int_0^\vc \Big[t^{-\alpha}\|\psi(t\sqrt{\Delta_d})F(\sqrt \Dd)f\|_{\ell^2(\ZZ)}\Big]^2\f{dt}{t}\Big)^{1/2}\\
	&\les \|F\|_\vc \Big(\int_0^\vc \Big[t^{-\alpha}\|\psi(t\sqrt{\Delta_d})f\|_{\ell^2(\ZZ)}\Big]^2\f{dt}{t}\Big)^{1/2}\\
	&\sim \|f\|_{\F^{\alpha}_{2,2}(\ZZ)}.
\end{aligned}
\]
This, along with \eqref{eq1-thm1-Hardy} and Proposition \ref{prop-comple interpolation}, yields that 
\begin{equation}\label{eq2-thm1-Hardy}
	\|F(\sqrt{\Dd})\|_{\F^{\alpha}_{p,2}(\ZZ)\to \F^{\alpha}_{p,2}(\ZZ)}\les  \|F\|_{L^\vc} +\sup_{t\ge 1/2}\|\eta\, \delta_tF\|_{W^{s}_r}
\end{equation}
for some $r\in (4,\vc]$, $s>\f{1}{p}-\f{1}{2}$ with $0<p<\vc$ and $\alpha\in \mathbb R$.

This completes the proof of (a).

(b) For the item (b), we need the following estimates whose proof will be given at the end of this section.

	\begin{lemma}\label{lem2-thm2 atom Besov-spectral}
	Let $\psi\in \mathcal S(\mathbb R)$ supported in $[2,8]$ and  let $a_I$ be an $(M, N, p)_{\Delta_d}$ molecule with some $I\in \mathcal{I}_\nu$. Then for any bounded Borel function $F$ satisfying $\sup_{t>1/2}\|\eta\, \delta_tF\|_{W^{s}_r}$ for some $s>\f{n}{1\wedge p\wedge q}$ and some $r\in (4,\vc]$, we have
	\begin{equation}\label{eq- psi atom-spectral}
		|\psi(t\sqrt{\Delta_d})F(\sqrt{\Delta_d}) a_I(n)|\les \Big(\f{t}{2^{-\nu}}\wedge \f{2^{-\nu}}{t}\Big)^{2M-1}|I|^{-1/p} \Big(1+\f{d(n,I)}{2^{-\nu}\vee t}\Big)^{-s'} 
	\end{equation}
	for all $t>0$, $s>s'>\f{n}{1\wedge p\wedge q}$ and $N>s+1$.
\end{lemma}

With the estimates in Lemma \ref{lem2-thm2 atom Besov-spectral}, we can proceed similarly to the proof of Theorem \ref{thm2- atom Besov}. 

This completes the proof of (b).

\end{proof}

\begin{proof}[Proof of Lemma \ref{lem2-thm2 atom Besov-spectral}:]
	We now consider two cases: $t\le 2^{-\nu}$ and $t>2^{-\nu}$.

\noindent{\bf Case 1: $t\le 2^{-\nu}$.} Observe that
$$
\psi(t\sqrt{\Delta_d}) a_I=t^{2M}\psi_M(t\sqrt{\Delta_d})(\Delta_d^Ma_I) 
$$
where $\psi_M(\lambda)=\lambda^{-2M}\psi(\lambda)$.

This, along with Proposition \ref{prop1-pointwise-spectral multiplier} and the definition of the molecules, yields
$$
\begin{aligned}
	|\psi(t\sqrt{\Delta_d}) a_I(n)|&\les  \sum_{m\in \ZZ} \f{t^{2M}}{(t+1)}\Big(1+\f{|m-n|}{t}\Big)^{-s'}|a_I(m)| \\
	&\les  \sum_{m\in \ZZ} \Big(\f{t}{2^{-\nu}}\Big)^{2M}|I|^{-1/p}\f{1}{(t+1)}\Big(1+\f{|m-n|}{t}\Big)^{-s'}\Big(1+\f{d(m,I)}{2^{-\nu}}\Big)^{-N} \\
	&\les  \Big(\f{t}{2^{-\nu}}\Big)^{2M}|I|^{-1/p}\sum_{m\in \ZZ} \f{1}{t}\Big(1+\f{|m-n|}{2^{-\nu}}\Big)^{-s'}\Big(1+\f{d(m,I)}{2^{-\nu}}\Big)^{-N}.
\end{aligned}
$$
We now use the inequality
\[
\Big(1+\f{|m-n|}{2^{-\nu}}\Big)^{-s'}\Big(1+\f{d(m,I)}{2^{-\nu}}\Big)^{-s'}\le \Big(1+\f{d(n,I)}{2^{-\nu}}\Big)^{-s'}
\]
to obtain further that 
$$
\begin{aligned}
	|\psi(t\sqrt{\Delta_d}) a_I(n)|& \les  \Big(\f{t}{2^{-\nu}}\Big)^{2M-1}|I|^{-1/p}\Big(1+\f{d(n,I)}{2^{-\nu}}\Big)^{-s'}\sum_{m\in \ZZ} \f{1}{2^{-\nu}} \Big(1+\f{d(m,I)}{2^{-\nu}}\Big)^{-N+s'}\\
	& \les  \Big(\f{t}{2^{-\nu}}\Big)^{2M-1}|I|^{-1/p}\Big(1+\f{d(n,I)}{2^{-\nu}}\Big)^{-s'},
\end{aligned}
$$ 
which yields  \eqref{eq- psi atom-spectral}.

\medskip

\noindent{\bf Case 2: $t> 2^{-\nu}$.} We first write  $a_I=\Delta_d^{M}b_I$. Hence,
\[
\psi(t\sqrt{\Delta_d})a_I=t^{-2M}\tilde \psi_M(t\sqrt{\Delta_d})b_I
\]
where $\tilde \psi_M(\lambda)=\lambda^{2M}\psi(\lambda)$.

This, along with Proposition \ref{prop1-pointwise-spectral multiplier}, implies that
$$
\begin{aligned}
	|\psi(t\sqrt{\Delta_d})a_I(n)|&\les  \sum_{m\in \ZZ}\f{t^{-2M}}{t}\Big(1+\f{|m-n|}{t}\Big)^{-s'}|b_I(m)| \\
	&\les  \Big(\f{2^{-\nu}}{t}\Big)^{2M} |I|^{-1/p}\sum_{m\in \ZZ}\f{1}{t}\Big(1+\f{|m-n|}{t}\Big)^{-s'}\Big(1+\f{d(m,I)}{2^{-\nu}}\Big)^{-N}\\
	&\les  \Big(\f{2^{-\nu}}{t}\Big)^{2M} |I|^{-1/p}\sum_{m\in \ZZ}\f{1}{t}\Big(1+\f{|m-n|}{t}\Big)^{-s'}\Big(1+\f{d(m,I)}{t}\Big)^{-N}\\
	&\les  \Big(\f{2^{-\nu}}{t}\Big)^{2M-1} |I|^{-1/p} \Big(1+\f{d(n,I)}{t}\Big)^{-s'}\sum_{m\in \ZZ}\f{1}{t}\Big(1+\f{d(m,I)}{t}\Big)^{-N+s'}\\
	&\les  \Big(\f{2^{-\nu}}{t}\Big)^{2M-1} |I|^{-1/p} \Big(1+\f{d(n,I)}{t}\Big)^{-s'}.
\end{aligned}
$$
Hence \eqref{eq- psi atom} follows.
	The desired estimate  \eqref{eq- psi atom} then follows.
\end{proof}

\subsection{Riesz transforms}
\begin{theorem}\label{mainthm-Riesz transforms}
	The Riesz transforms $D\Dd^{-1/2}$ and $D^* \Dd^{-1/2}$ are bounded on the Besov spaces $\B^{\alpha}_{p,q}(\ZZ)$ for $\alpha\in \mathbb{R}$, $0<p<\vc$, $0<q\le \vc$ and bounded on the Triebel--Lizorkin spaces $\F^{\alpha}_{p,q}(\ZZ)$ for $\alpha\in \mathbb{R}$, $0<p,q<\vc$.
\end{theorem}

Before coming to the details, we will define the operator $\Delta_d^{-1/2}$. Note that the subordination formula
\[
\Delta_d^{-1/2}f = c\int_0^\vc \sqrt t e^{-t\Delta_d}f\f{dt}{t}.
\]
is not absolutely convergent in $\ell^2(\ZZ)$. However, fortunately we can show that the above formula is well-defined on $\mathcal S'(\ZZ)/\mathcal P(\ZZ)$. Indeed, by duality it suffices to show that if $f\in \mathcal S_\vc(\ZZ)$, then 
\[
\Delta_d^{-1/2}f: = c\int_0^\vc \sqrt t e^{-t\Delta_d}f\f{dt}{t}
\]
belongs to $\mathcal S_\vc(\ZZ)$.

To do this, we write
\[
\int_0^\vc \sqrt t e^{-t\Delta_d}f\f{dt}{t} = \int_0^1 \sqrt t e^{-t\Delta_d}f\f{dt}{t}+\int_1^\vc \sqrt t e^{-t\Delta_d}f\f{dt}{t}.
\]
For the first term we have, for $N\ge 0$,
\[
\begin{aligned}
	\sup_{n\in \ZZ}\sup_{0\le \ell \le N} &\int_0^1 \sqrt t (1+|n|)^N|\Delta_d^\ell e^{-t\Delta_d}f(n)|\f{dt}{t}\\
	&\le \sup_{n\in \ZZ}\sup_{0\le \ell \le N} \int_0^1 \sqrt t (1+|n|)^N\sum_{m\in \ZZ} h_t(n-m)|\Delta_d^\ell f(m)|\f{dt}{t}\\
	&\le \sup_{n\in \ZZ}\sup_{0\le \ell \le N} \int_0^1 \sqrt t (1+|n|)^N\sum_{m\in \ZZ: |m-n|\le N+1} h_t(n-m)|\Delta_d^\ell f(m)|\f{dt}{t}\\
	& \ \ \ \ +\sup_{n\in \ZZ}\sup_{0\le \ell \le N} \int_0^1 \sqrt t (1+|n|)^N\sum_{m\in \ZZ: |m-n|> N+1} h_t(n-m)|\Delta_d^\ell f(m)|\f{dt}{t}\\
	&=:E_1+E_2.
\end{aligned}
\]
For the first term $E_1$, since $(1+|n|)\le C_N (1+|m|)$ as long as $|m-n|\le N+1$, we have
\[
\begin{aligned}
	E_1&\les  \sup_{n\in \ZZ}\sup_{0\le \ell \le N} \int_0^1 \sqrt t \sum_{m\in \ZZ: |m-n|\le N+1} h_t(n-m)(1+|m|)^N|\Delta_d^\ell f(m)|\f{dt}{t}\\
	&\les  \|f\|_N \sup_{n\in \ZZ}\sup_{0\le \ell \le N} \int_0^1 \sqrt t \sum_{m\in \ZZ} h_t(n-m) \f{dt}{t}\\
	&\les  \|f\|_N  \int_0^1 \sqrt t\f{dt}{t} \sim \|f\|_N.
\end{aligned}
\]

We now take care of $E_2$. By Lemma \ref{lem1-ht},
\[
\begin{aligned}
	E_2&\les \sup_{n\in \ZZ}\sup_{0\le \ell \le N} \int_0^1 \sqrt t (1+|n|)^N\sum_{m\in \ZZ: |m-n|> N+1} \f{1}{\sqrt t}\f{t^{N/2+1}}{|m-n|^{N+2}}|\Delta_d^\ell f(m)|\f{dt}{t}\\
	&\les \sup_{n\in \ZZ}\sup_{0\le \ell \le N} \int_0^1 \sqrt t \sum_{m\in \ZZ: |m-n|> N+1}  \f{1}{|m-n|^{2}}(1+|m|)^N|\Delta_d^\ell f(m)|\f{dt}{t}\\
	&\les \|f\|_N.
\end{aligned}
\]
It follows that $\displaystyle \int_0^1 \sqrt t e^{-t\Delta_d}f\f{dt}{t} \in \mathcal S_\vc(\ZZ)$ whenever $f\in \mathcal S_\vc(\ZZ)$.

Since $f\in \mathcal S_\vc(\ZZ)$, we can write $f=\Delta_d^{N+1}g$ with $g\in \mathcal S(\ZZ)$.  Therefore,
\[
\begin{aligned}
	\sup_{n\in \ZZ}\sup_{0\le \ell \le N} &\int_1^\vc \sqrt t (1+|n|)^N|\Delta_d^\ell e^{-t\Delta_d}f(n)|\f{dt}{t}\\
	&\le \sup_{n\in \ZZ}\sup_{0\le \ell \le N} \int_1^\vc \sqrt t (1+|n|)^N|\Delta_d^{\ell+N+1} e^{-t\Delta_d}g(n)|\f{dt}{t}.
\end{aligned}
\]
At this stage, using the kernel estimate in Lemma \ref{lem2-htk}, we can come up with
\[
\sup_{n\in \ZZ}\sup_{0\le \ell \le N} \int_1^\vc \sqrt t (1+|n|)^N|\Delta_d^{\ell+N+1} e^{-t\Delta_d}f(n)|\f{dt}{t}\les \|g\|_N.
\]
This implies that $\displaystyle \int_1^\vc \sqrt t e^{-t\Delta_d}f\f{dt}{t} \in \mathcal S_\vc(\ZZ)$ whenever $f\in \mathcal S_\vc(\ZZ)$.

As a consequence, $\Delta^{-1/2}f \in \mathcal S_\vc(\ZZ)$ whenever $f\in \mathcal S_\vc(\ZZ)$. Hence, by duality, $\Delta^{-1/2}f \in \mathcal S'_\vc(\ZZ)$ whenever $f\in \mathcal S'_\vc(\ZZ)$.

\begin{theorem}\label{thm1- proof of Riesz}
	The operator $D$ and $D^*$ maps 
	\begin{enumerate}[(i)]
		\item from $\B^{\alpha}_{p,q}(\ZZ)$ to $\B^{\alpha-1}_{p,q}(\ZZ)$ continuously for $0<p,q\le \vc$ and $\alpha\in \mathbb R$;
		\item from $\F^{\alpha}_{p,q}(\ZZ)$ to $\F^{\alpha-1}_{p,q}(\ZZ)$ continuously for $0<p< \vc, 0<q\le  \vc$ and $\alpha\in \mathbb R$. 
	\end{enumerate}
\end{theorem}
\begin{proof}
	By the similar argument in the proof of Theorem \ref{thm2- atom Besov}, it suffices to prove the following estimates: Let $\psi$ be a partition of unity and  let $a_I$ be a $(M,N,p)_{\Dd}$ molecule with some $I\in \mathcal{I}_\nu$ with $\nu\in \ZZ^-$. Then for any $t>0$ and  $N>0$ we have:
	\begin{equation}\label{eq- psi atom}
		|\sqrt t\psi(t\sqrt{\Delta_d}) D a_I(n)|+|\sqrt t\psi(t\sqrt{\Delta_d}) D^* a_I(n)|\les  \Big(\f{t}{2^{-\nu}}\wedge \f{2^{-\nu}}{t}\Big)^{2M-1}|I|^{-1/p} \Big(1+\f{d(n,I)}{2^{-\nu}\vee t}\Big)^{-N}.
	\end{equation}
	Since $\Delta^M_d D a_I = D\Delta^M_d  a_I$ and $\Delta^M_d D^* a_I = D^*\Delta^M_d  a_I$, we can proceed the proof similarly to that of \eqref{eq- psi atom}, in which we use the estimate in Lemma \ref{lem1-derivative of psi delta} instead of Corollary \ref{lem1}. We would like to leave the details to the interested reader.
	
	This completes our proof.
\end{proof}

We are ready to give the proof of Theorem \ref{mainthm-Riesz transforms}.

\begin{proof}[Proof of Theorem \ref{mainthm-Riesz transforms}:]
	Due to Theorem \ref{thm1- proof of Riesz}, it suffices to prove that the operator $\Delta_d^{-1/2}$ maps \begin{enumerate}[(i)]
		\item from $\B^{\alpha}_{p,q}(\ZZ)$ to $\B^{\alpha+1}_{p,q}(\ZZ)$ continuously for $0<p,q\le \vc$ and $\alpha\in \mathbb R$;
		\item from $\F^{\alpha}_{p,q}(\ZZ)$ to $\F^{\alpha+1}_{p,q}(\ZZ)$ continuously for $0<p< \vc, 0<q\le  \vc$ and $\alpha\in \mathbb R$. 
	\end{enumerate}

To do this, let $\psi$ be a partition of unity. Then by Theorem \ref{thm1-continuouscharacter},
\[
\begin{aligned}
	\|\Delta_d^{-1/2}f\|_{\B^{\alpha+1}_{p,q}(\ZZ)}&\sim \Big(\int_0^\vc \Big[t^{-\alpha-1}\|\psi(t\sqrt{\Delta_d})\Dd^{-1/2}f\|_{\ell^p(\ZZ)}\Big]^q\f{dt}{t}\Big)^{1/q}\\
	&=\Big(\int_0^\vc \Big[t^{-\alpha}\|\widetilde \psi(t\sqrt{\Delta_d})f\|_{\ell^p(\ZZ)}\Big]^q\f{dt}{t}\Big)^{1/q}\\
	&\sim \|f\|_{\B^{\alpha}_{p,q}(\ZZ)},
\end{aligned}
\]
where $\widetilde \psi(\lambda) = \lambda \psi(\lambda)$.

This implies (i). Similarly, we also obtain the statement (ii).
This completes our proof.
\end{proof}

\bigskip

\textbf{Acknowledgement:} The authors were supported by the Australian Research Council through the grant DP220100285. The authors would like to thank the referee for his/her useful comments and suggestions  which helped to improve the paper.

\end{document}